\documentclass[opre,nonblindrev]{informs3} 

\usepackage{lmodern}
\usepackage{caption}
\usepackage{anyfontsize}
\usepackage{subfiles}
\usepackage{xr}
\usepackage{float}
\def\biblio{\bibliographystyle{informs2014}\bibliography{../DataDrivenRobustBibliography}}

\OneAndAHalfSpacedXI 

\usepackage{accents}

\usepackage{natbib}
 \bibpunct[, ]{(}{)}{,}{a}{}{,}%
 \def\bibfont{\small}%
\TheoremsNumberedThrough     
\ECRepeatTheorems

\EquationsNumberedThrough    

\RequirePackage[shortlabels]{enumitem}
\usepackage{mathtools} 
\usepackage{enumitem} 
\usepackage{dsfont} 
\usepackage{mathrsfs}  

\newtheorem{innercustomgeneric}{\customgenericname}
\providecommand{\customgenericname}{}
\newcommand{\newcustomtheorem}[2]{%
	\newenvironment{#1}[1]
	{%
		\renewcommand\customgenericname{#2}%
		\renewcommand\theinnercustomgeneric{##1}%
		\innercustomgeneric
	}
	{\endinnercustomgeneric}
}

\newcustomtheorem{customthm}{Theorem}
\newcustomtheorem{customlemma}{Lemma}

\allowdisplaybreaks

\usepackage{epsfig}

\let\oldnl\nl
\newcommand{\nonl}{\renewcommand{\nl}{\let\nl\oldnl}}

\usepackage{IEEEtrantools}
\usepackage{multicol}
\usepackage{caption}
\usepackage{subcaption}

\usepackage[margin=2.54cm]{geometry}
\usepackage{lipsum}
\usepackage{bigints}
\usepackage{framed}
\usepackage{colortbl}
\usepackage{calligra}
\usepackage[export]{adjustbox}
\usepackage{booktabs}
\usepackage{multirow}
\usepackage{makecell}

 \usepackage{wrapfig}

\usepackage{bbm}
\usepackage[normalem]{ulem}

\usepackage[colorlinks=true,linkcolor=black,citecolor=black]{hyperref}
\usepackage{cleveref}

\newcommand{\R}{{\mathcal R}}

\newcommand{\N}{{\mathbb N}}
\newcommand{\I}{{\mathbb I}}

\newcommand{\PP}{\mathbb P}
\newcommand{\Prob}{{\textup{Prob}}}

\newcommand{\ZZ}{\mathbb Z}

\newcommand{\PPhat}{{\hat{\mathbb P}_N}}

\newcommand{\mP}{\mathcal{P}}

\newcommand{\eps}{\epsilon}

\newcommand{\bm}{\boldsymbol}

\def\RR{ {\mathbb{R}}}

\newcommand{\X}{{\mathcal X}}

\newcommand{\HH}{{{\mathcal H}}}

\newcommand{\M}{{{\mathbb M}} }

\newcommand{\txi}{\tilde{\bm{\xi}}}
\newcommand{\tchi}{\tilde{\bm{\chi}}}
\newcommand{\tomega}{\tilde{\bm{\omega}}}
\newcommand{\bchi}{{\bm{\chi}}}
\newcommand{\bomega}{{\bm{\omega}}}

\newcommand{\x}{{\bm{x}}}
\newcommand{\y}{{\bm{y}}}

\newcommand{\bc}{{\bm{c}}}
\newcommand{\tc}{{\tilde{\bm{c}}}}

\newcommand{\z}{{\bm{z}}}
\newcommand{\tz}{{\tilde{\bm{z}}}}
\newcommand{\ts}{{\tilde{{s}}}}

\newcommand{\lbar}{{\overline{\ell}}}

\newcommand{\lu}{{\underline{\ell}}}

\newcommand{\indep}{\perp \!\!\! \perp}

\newcommand{\bxi}{{\bm\xi}}

\newcommand{\C}{{\mathcal C}}

\newcommand{\QQ}{{\mathbb{Q}}}

\newcommand{\EE}{{\mathbb{E}}}

\newcommand{\var}{\textrm{Var}}

\let\R\Real

\def\st{\textup{ s.t. }}

\def\M{\mathcal{M}}

\renewcommand{\qed}{{\hfill\halmos}}

\usepackage{graphicx}
\usepackage{tikz,ifthen}

   \usetikzlibrary{math}
   \usetikzlibrary{shapes.misc}
   \usetikzlibrary{shapes.multipart, positioning, patterns,backgrounds}
\usepackage{algorithm, algpseudocode}

\usepackage{multirow}
\usepackage{array}
\usepackage{cleveref} 

\usepackage{booktabs}
\usepackage{multirow}
\usepackage{makecell}

\usepackage{array}
\newcolumntype{x}[1]{>{\centering\arraybackslash}p{#1}}

\begin{document}

\def\biblio{}




\RUNAUTHOR{Fang et al.}
\RUNTITLE{Distributionally Robust Optimization via Targeted Integral Probability Metrics}

\ARTICLEAUTHORS{Lanran Fang\thanks{Department of Industrial and Enterprise Systems Engineering, University of Illinois Urbana-Champaign, Urbana, IL 61801, USA. Email: {\tt \{lanranf2,gah\}@illinois.edu}.}, Jianqiang Cheng\thanks{Department of Systems and Industrial Engineering, 
University of Arizona, Tucson, AZ 85721. Email: {\tt jqcheng@arizona.edu}.},  
Grani A. Hanasusanto\footnotemark[1], and Yijie Wang\thanks{School of Economics and Management, Tongji University, Shanghai, China. Email: {\tt yijiewang@tongji.edu.cn}.}
}


\TITLE{Distributionally Robust Optimization via Targeted Integral Probability Metrics for General Data Processes}

\ABSTRACT{%
Distributionally robust optimization (DRO) provides a principled framework for decision-making under distributional uncertainty. Classical data-driven DRO frameworks typically construct ambiguity sets from distributional information, such as moment constraints, divergence neighborhoods, or Wasserstein balls, specified before the downstream loss is considered. We propose a task-aware DRO framework based on targeted integral probability metrics. The ambiguity set is defined directly through the loss functions induced by feasible decisions, thereby controlling the loss discrepancy between an adversarial distribution and a data-driven reference distribution. This construction leads to an expected hinge-constrained formulation that is equivalent to an infinitely constrained loss-discrepancy formulation. It also yields finite-sample guarantees that bypass the ambient curse of dimensionality: whenever an appropriate scalar pointwise concentration inequality is available for the induced loss estimator, the ambiguity radius can be calibrated at the canonical $\widetilde{\mathcal O}(N^{-1/2})$ rate after uniformization over the decision class. As a result, the framework applies broadly to settings including heavier-tailed sub-Weibull losses, Markovian data, outlier-corrupted data, and incomplete data. We derive exact infinite-dimensional dual reformulations, establish out-of-sample and excess-risk guarantees, and develop a conservative Monte Carlo approximation scheme with convergence and suboptimality guarantees. For piecewise affine losses, the sampled problems admit tractable conic reformulations. Numerical experiments in inventory management under heavy-tailed demand and regression with outlier corruption demonstrate strong out-of-sample performance relative to existing approaches.
}

\KEYWORDS{distributionally robust optimization, integral probability metric, semi-infinite programming, heavy-tailed distributions, Markovian data, outlier-corrupted data, incomplete data}

\maketitle

\section{Introduction} \label{sec:intro}

Distributionally Robust Optimization (DRO) provides a principled mathematical framework for optimization under uncertainty.  Rather than placing blind faith in a single nominal distribution (e.g., the empirical distribution), DRO identifies a decision that minimizes the worst-case expected loss over an \emph{ambiguity set} $\mP$, defined as a family of probability measures that are statistically plausible given the available data:
\begin{equation}\label{eq:general_dro}
    \min_{\x\in\X} \sup_{\QQ\in\mP}  \EE_{\QQ}[\ell(\x,\txi)],
\end{equation}
where $\ell(\x,\bxi)$ represents the loss function parameterized by the decision $\x \in \X$ and the uncertain parameter $\bxi \in \Xi$. By optimizing over the most adverse distribution $\QQ$ within $\mP$, the DRO framework yields decisions that remain reliable even when deployed in unseen environments.  The central design question is therefore the choice of $\mP$: it must be rich enough to capture meaningful uncertainty, yet structured enough to admit tractable computation and valid statistical guarantees~\citep{rahimian2019distributionally,kuhn2025distributionally}.

Early DRO ambiguity sets were commonly built from moment constraints or statistical divergences. Moment ambiguity sets~\citep{delage2010distributionally,goh2010distributionally, wiesemann2014distributionally} impose generalized moment conditions, but they are not asymptotically consistent because low-order moments do not identify the true distribution. Discrepancy-based sets centered at the empirical distribution $\PPhat$, such as $\phi$-divergence balls~\citep{ben2013robust,bertsimas2018data,hu2013kullback}, restore consistency but have a restrictive support property: assigning positive mass outside the reference support leads to infinite divergence. Wasserstein ambiguity sets address this topological limitation by allowing probability mass to be transported from observed samples to nearby unobserved regions~\citep{mohajerin2018data,blanchet2019quantifying, gao2023distributionally}. A substantial literature has further established dual reformulations, finite-sample performance guarantees, and tractable data-driven models for Wasserstein DRO~\citep{mohajerin2018data,blanchet2019quantifying,gao2023distributionally}. However, empirical Wasserstein balls can inherit ambient-dimensional concentration rates, with standard calibrations deteriorating as $\mathcal O(N^{-1/\max\{2,d\}})$ in $d$ dimensions~\citep{fournier2015rate,mohajerin2018data}. Existing work mitigates this bottleneck through targeted coverage, regularization equivalences, and refined concentration analyses~\citep{blanchet2021sample,shafieezadeh2019regularization,duchi2019variance,gao2023finite,gao2024wasserstein,wu2025generalization}.

The common feature of these ambiguity sets is that they are specified through distributional information of the uncertainty $\txi$, such as moments, divergences, transport costs, or kernels, rather than through the expected-loss profile itself. The DRO problem, however, depends on a candidate distribution $\QQ$ only through its expected-loss profile $\x\mapsto \EE_\QQ[\ell(\x,\txi)]$ over feasible decisions, not through the full law of $\txi$ on $\Xi$. This distinction matters statistically: an ambient Wasserstein ball is typically calibrated in a geometry designed to cover the full law of $\txi$, while the DRO objective only requires control of the decision-induced expected losses. This creates a mismatch: the ambiguity set may protect against distributional changes that do not materially change the expected losses relevant to the decision maker.

An elegant and growing line of work has addressed this alignment issue within the Wasserstein framework by incorporating task knowledge into the ground transportation cost. Such approaches encode component scaling and heterogeneous perturbations, learn cost geometry from data, distinguish feature and label perturbations, accommodate mixed-feature spaces, or use predictors and loss evaluations to shape transport costs~\citep{mohajerin2018data,blanchet2019data,shafieezadeh2015distributionally,li2019first,selvi2022wasserstein,belbasi2026mix,wang2025knowledge,ohnemus2025loss}. These approaches are valuable when the relevant task geometry is available or can be reliably learned. More fundamentally, however, they still anchor the ambiguity set to a distributional distance on $\Xi$; statistical calibration of the radius therefore remains problem-specific, and their extension to non-standard data regimes such as heavy-tailed, Markovian, or corrupted data requires dedicated analysis.

We take a different approach based on loss discrepancies.  Rather than measuring how far $\QQ$ is from $\PPhat$ in the space of distributions, we ask directly: can $\QQ$ induce a loss discrepancy over the feasible decision set?  A distribution $\QQ$ is distant from $\PPhat$ when its loss discrepancy
\[
\sup_{\x\in\X}\bigl[\EE_\QQ[\ell(\x,\txi)]
- \EE_\PPhat[\ell(\x,\txi)]\bigr]
\]
is large.  Defined over the decision space $\X$ through the loss function alone, this quantity does not ask whether $\QQ$ is close to $\PPhat$ as a full distribution on $\Xi$; it asks whether $\QQ$ produces similar expected losses for the decisions we care about.

This loss-discrepancy view can be interpreted through the lens of Integral Probability Metrics (IPMs), which compare distributions through expectation gaps over a chosen function class~\citep{muller1997integral,kuhn2025distributionally}:
\[
D_{\mathcal F}(\PP,\hat\PP)
=\sup_{f\in\mathcal F}\left|\EE_{\PP}[f(\txi)]-\EE_{\hat\PP}[f(\txi)]\right|.
\]
The IPM framework recovers several classical discrepancies through different choices of $\mathcal F$~\citep{muller1997integral,sriperumbudur2016optimal,birrell2022f,kuhn2025distributionally}: indicator functions of Borel sets yield total variation, 1-Lipschitz functions yield the 1-Wasserstein distance via Kantorovich--Rubinstein duality~\citep{villani2008optimal}, and the unit ball of a Reproducing Kernel Hilbert Space (RKHS) yields the Maximum Mean Discrepancy (MMD)~\citep{zhu2021kernel}. In these classical cases, however, $\mathcal F$ is defined by the ambient geometric or analytic structure of $\Xi$, so the discrepancy is specified without reference to the loss function $\ell$ or the feasible decisions $\mathcal{X}$.
Our construction instead takes the test-function class to be the decision-induced loss class $\mathcal F_\ell=\{\ell(\x,\cdot):\x\in\X\}$ and constrains only the expectation increases that can affect the DRO objective. Thus the ambiguity set is targeted to the feasible decisions and loss values that matter for optimization, without requiring a separately designed ground metric on $\Xi$. We refer to this loss-driven discrepancy as a \textbf{Targeted Integral Probability Metric (TIPM)}.

The same loss-class viewpoint also connects naturally to statistical learning theory. Foundational uniform convergence theory~\citep{vapnik1998statistical,bartlett2002rademacher} establishes that generalization is controlled not by distributional distance in any ambient metric, but by the complexity and tail behavior of the induced loss class. These bounds depend on properties of the loss class rather than on concentration of the empirical distribution in the ambient space of $\txi$, and apply whether $\txi$ is continuous, discrete, or of mixed type. Our TIPM construction adopts this perspective: by defining the ambiguity set through $\mathcal{F}_\ell$, radius calibration reduces to a uniform bound over the induced loss class. This perspective explains both the dimension-free rate and the breadth of applicable data regimes.

Concretely, we propose an expected hinge-constrained ambiguity set:
\begin{equation}
\label{eq:ambiguity_set_intro}
\mP_\epsilon\coloneqq \left\{\QQ\in\mathscr P(\Xi):
\begin{array}{l}
\EE_\ZZ\!\left[\left[\EE_\QQ[\ell(\tz,\txi)]
    -\EE_\PPhat[\ell(\tz,\txi)]-\eps\right]_+\right] \leq 0,\\[4pt]
\EE_\QQ[\|\txi\|^2]\leq \Omega,
\end{array}
\right\}
\end{equation}
where $\eps$ is a tunable size parameter, $\ZZ$ is an auxiliary probability distribution with full support on the decision space $\X$, $\PPhat$ is the reference distribution constructed from historical data, and the second-moment constraint ensures basic statistical regularity.  The hinge constraint forces the inner term to be non-positive $\ZZ$-almost everywhere; under the continuity conditions used below, this is equivalent to enforcing infinitely many loss-discrepancy constraints over the decision space.

\medskip
\noindent\textbf{Informal Proposition.}
\emph{
Under mild conditions, the ambiguity set \eqref{eq:ambiguity_set_intro}
is equivalent to:
\[
\mP_\epsilon\coloneqq \left\{\QQ\in\mathscr P(\Xi):\begin{array}{l}
\EE_\QQ[\ell(\z,\txi)] - \EE_\PPhat[\ell(\z,\txi)]\leq \epsilon
\quad\forall\z\in\X,\\[4pt]
\EE_\QQ[\|\txi\|^2]\leq \Omega.
\end{array}\right\}
\]
}

This reformulation reveals the nature of the proposed set: $\mP_\epsilon$ is the family of distributions whose expected loss exceeds the empirical benchmark by at most $\epsilon$, uniformly across all feasible decisions.  A key departure from classical IPM constructions is that, by dropping the absolute value required for a symmetric distance, our TIPM is not a symmetric statistical distance, but an upper bound on the loss discrepancy relevant to DRO.  Since the supremum in the DRO problem
\eqref{eq:general_dro} seeks adversarial distributions that
\emph{increase} the expected loss, symmetric penalization is
redundant; the resulting upper bound accommodates optimistic distributions while constraining the loss discrepancies that matter for out-of-sample performance.

This paper makes the following contributions.
\begin{itemize}

\item \textbf{A task-aware ambiguity set with exact characterization.}  We introduce an expected hinge-constrained ambiguity set based on a Targeted Integral Probability Metric (TIPM), in which the discrepancy is defined directly through the loss functions induced by feasible decisions.  We establish that this ambiguity set is equivalent to an infinitely constrained loss-discrepancy set, aligning robustness with the downstream decision problem.  We further derive an exact infinite-dimensional dual reformulation of the resulting DRO problem.

\item \textbf{Dimension-free statistical and excess-risk guarantees under broad data regimes.}\footnote{We use ``dimension-free'' in the rate sense: the sample-size exponent does not deteriorate with the ambient dimension \(D_\bxi\), although constants may depend on structural problem parameters.} The TIPM construction yields finite-sample guarantees that bypass the curse of dimensionality. Whenever an appropriate scalar pointwise concentration inequality is available for the loss, radius calibration reduces to uniformizing this scalar bound over the decision class, yielding $\widetilde{\mathcal O}(1/\sqrt{N})$ out-of-sample guarantees. We further establish a dimension-free excess-risk bound: the gap between the robust solution and the oracle optimal decision under the true distribution is $\widetilde{\mathcal O}(1/\sqrt{N})$, quantifying the oracle optimality gap beyond the standard DRO out-of-sample certificate. These guarantees extend to a broad range of data-generating processes, including sub-Weibull losses, outlier-corrupted data, Markovian data, and incomplete data, and naturally accommodate mixed discrete and continuous uncertainty.

\item \textbf{A conservative Monte Carlo approximation with convergence, tractability, and unbounded-support guarantees.}  To overcome the intractability of the continuous auxiliary distribution $\ZZ$, we introduce a Monte Carlo approximation that draws $M$ samples from $\ZZ$ to replace the ambiguity set with a conservative sampled counterpart.  For piecewise affine losses, we establish convergence of optimal values and solutions at a $\widetilde{\mathcal O}(1/\sqrt{M})$ rate with explicit suboptimality guarantees, and show that the resulting sampled problems admit tractable conic reformulations, including second-order cone approximations for two-stage linear losses.  We further extend this framework to unbounded support under a single-threshold tail bound implied by sub-Weibull tails, showing that the Monte Carlo principle remains valid up to explicit tail-dependent terms while preserving tractability.

\end{itemize}

\subsection{Related Literature}

\noindent\textbf{IPM ambiguity sets and distributional robustness.}  Integral Probability Metrics (IPMs) provide a unifying language for measuring discrepancy between probability distributions through a prescribed class of test functions. Different choices of $\mathcal{F}$ recover classical discrepancies including total variation, Wasserstein distances, and kernel-based metrics~\citep{muller1997integral,sriperumbudur2016optimal,villani2008optimal}.  The formal theory and topological properties are developed in~\cite{muller1997integral,sriperumbudur2016optimal,birrell2022f}.  A comprehensive treatment of IPM-based DRO, including its relations to Wasserstein DRO and regularization, is given in~\cite{kuhn2025distributionally}.  DRO formulations based on specific IPM instances have been extensively studied:~\cite{husain2020distributional} establishes connections between IPM DRO, Tikhonov regularization, and generative adversarial networks;~\cite{zhu2021kernel} develops kernel DRO based on MMD with stochastic approximation algorithms;~\cite{zeng2022generalization} derives generalization bounds via DRO with minimal hypothesis-class dependency; and~\cite{iyengar2023hedging} develops computationally tractable IPM DRO via parametric approximation.  In each of these works, the test-function class is specified by an ambient smoothness, kernel, or norm constraint that is independent of the decision problem. The present work departs from this literature by taking $\mathcal{F} = \{\ell(x,\cdot) : x \in \mathcal{X}\}$, so the function class is generated by the downstream optimization problem itself and the resulting discrepancy is shaped by downstream losses rather than by a symmetric distributional distance.

\noindent\textbf{Structured and task-aware Wasserstein ambiguity sets.}  Wasserstein DRO can incorporate problem structure through the choice of transportation cost.  General norm-based costs encode component scaling and heterogeneous perturbations~\citep{mohajerin2018data}, and data-driven selection learns the ambiguity geometry from data~\citep{blanchet2019data}.  In supervised learning, label-aware transport costs distinguish feature perturbations from label perturbations~\citep{shafieezadeh2015distributionally,li2019first}, and recent work develops Wasserstein models for mixed numerical and categorical features~\citep{selvi2022wasserstein,belbasi2026mix}. More recent knowledge-guided and loss-aware formulations use external predictors or loss information to shape the optimal transport ambiguity set~\citep{wang2025knowledge,ohnemus2025loss}.  These approaches are complementary to ours: they encode task relevance through a designed or learned ground geometry, whereas we define discrepancy directly through the decision-induced loss class and the expected losses it generates.

\noindent\textbf{Generalization guarantees for DRO.}  A central question in data-driven DRO is how to calibrate ambiguity radii to obtain reliable out-of-sample guarantees.  Moment and divergence ambiguity sets admit dimension-independent calibration through scalar concentration~\citep{delage2010distributionally, ben2013robust,bertsimas2018data,duchi2021statistics,lam2019recovering}, whereas empirical Wasserstein balls suffer from ambient-dimensional concentration rates~\citep{fournier2015rate}.  A large literature has developed sharper Wasserstein guarantees through targeted coverage, regularization equivalences, and exact generalization analyses~\citep{blanchet2021sample,blanchet2022confidence, gao2023finite,gao2024wasserstein,wu2025generalization,lee2018minimax,sinha2017certifying}, and transportation-information inequalities have been used to obtain dimension-free bounds for general learning problems~\citep{xu2017information,wang2019information, xie2022distributionally,he2025information}.  Our contribution is different in mechanism: instead of improving concentration of an ambient distributional distance, we project the ambiguity set onto the decision-induced loss class, so radius calibration reduces to scalar pointwise concentration plus uniformization over the decision class, without Talagrand assumptions or problem-specific structural analysis.

\noindent\textbf{DRO under nonstandard data-generating processes.} In practice, decision-makers frequently operate in environments where available data violate the standard i.i.d.\ assumption. Time-series and sequential decision problems often involve Markovian data~\citep{fan2021hoeffding,li2021distributionally, nagaraj2020least}.  In selection problems such as hiring and college admissions, outcomes are observable only for accepted candidates, leading to incomplete data regimes~\citep{najafi2019robustness}. Modern decision-making increasingly relies on auxiliary side information, requiring optimization over conditional distributions in contextual settings~\citep{bertsimas2020predictive,sadana2025survey}.  Furthermore, data collected in real-world environments are often contaminated by noise or adversarial perturbations~\citep{nietert2023outlier, nietert2024robust,blanchet2024automatic,jiang2024distributionally}. Existing DRO approaches in these settings are often tailored to a specific type of data structure, with ambiguity set designs adapted to the particular anomaly at hand.  Because the TIPM radius can be calibrated whenever an appropriate pointwise concentration inequality is available for the induced loss estimator, the same ambiguity set design accommodates sub-Weibull losses, Markovian data, outlier-corrupted observations, and incomplete or contextual data without modification.

\subsection{Paper Structure and Notation}

\noindent\textbf{Paper structure.} The remainder of the paper is organized as follows. Section~\ref{sec:IPMDRO} formally introduces the expected hinge-constrained TIPM ambiguity set, establishes its equivalent infinitely constrained loss-discrepancy representation, and derives coverage, out-of-sample, and excess-risk guarantees.  It also shows how the same framework applies to several data-generating regimes through pointwise concentration inequalities, including sub-Weibull losses, Markovian data, outlier-corrupted observations, incomplete data, and contextual information, and develops exact dual reformulations.  Section~\ref{sec:mc_sampling} studies a Monte Carlo approximation of the ambiguity set, proves its conservative containment and convergence properties, and derives finite conic reformulations for piecewise affine losses. Section~\ref{sec:unbounded_support} extends these sampling and reformulation results to unbounded support under a single-threshold tail bound implied by sub-Weibull tails. Section~\ref{sec:experiments} evaluates the proposed framework in multi-item newsvendor and outlier-corrupted regression experiments. The appendices collect all proofs, a two-stage linear decision-rule approximation, and detailed experimental specifications.

\noindent\textbf{Notation.} For a positive integer $n$, we write $[n]\coloneqq\{1,\ldots,n\}$. Random variables are designated by tilde signs (e.g., $\txi$), while their realizations are denoted by the same symbols without tildes (e.g., $\bxi$). The true data-generating distribution is denoted by $\PP^\star$, the data-driven empirical distribution by $\PPhat$, a generic distribution in the ambiguity set by $\QQ$, and the sampling distribution over decisions by $\ZZ$. We use $\mathscr P(\Xi)$ for the set of probability distributions on $\Xi$, $\EE_\PP[\cdot]$ for expectation under $\PP$, and $\I_A$ or $\I[A]$ for the indicator of an event $A$.

The feasible decision set is $\X\subseteq\RR^{D_\x}$, the uncertainty support is $\Xi\subseteq\RR^{D_\bxi}$, and $R_\x$ and $R_\bxi$ denote corresponding norm bounds when they are finite. The ambiguity radius is denoted by $\epsilon$, while $\varepsilon(\cdot)$ denotes the nondecreasing function appearing in pointwise concentration bounds. The data and Monte Carlo sample sizes are $N$ and $M$, respectively. For a scalar $u$, $[u]_+\coloneqq\max\{u,0\}$. The Euclidean norm is denoted by $\|\cdot\|$ and the operator norm of a matrix by $\|\cdot\|_{\textup{op}}$. For a set $\Xi$, its support function is $\sigma_\Xi(\bm z)\coloneqq\sup_{\bxi\in\Xi}\bm z^\top\bxi$. The all-ones vector of appropriate dimension is denoted by $\mathbf{e}$. We use $\mathcal O(\cdot)$ for upper bounds up to absolute constants, and $\widetilde{\mathcal O}(\cdot)$ when polylogarithmic factors are additionally suppressed. For a point $\x$ and a set $\mathcal S$, $\textup{dist}(\x,\mathcal S)$ denotes the distance from $\x$ to $\mathcal S$. The limit $\overset{p}{\to}$ denotes convergence in probability. Additional notation is introduced where it is used.

\section{Distributionally Robust Optimization with Expected Hinge-Constrained Ambiguity Sets}\label{sec:IPMDRO}

We consider the distributionally robust optimization problem
\begin{equation}
\label{eq:DRO}
\min_{\x\in\X} \sup_{\QQ\in\mP_\epsilon}  \EE_{\QQ}[\ell(\x,\txi)],
\end{equation}
where the ambiguity set is given by
\begin{equation}
\label{eq:ambiguity_set}
\mP_\epsilon\coloneqq \left\{\QQ\in\mathscr P(\Xi):
\begin{array}{l}\EE_\ZZ\left[\left[\EE_\QQ[\ell(\tz,\txi)]-\EE_\PPhat[\ell(\tz,\txi)]-\eps\right]_+\right] \leq 0\\\;\EE_\QQ[\|\txi\|^2]\leq \Omega
\end{array}
\right\},
\end{equation}
and the feasible set $\X$ is compact and convex. 
Here, $\epsilon\geq 0$ is a size parameter and $\ZZ$ is a probability distribution with full support on a compact decision set $\bar\X$ satisfying $\X\subseteq\bar\X$. The auxiliary random vector $\tz\sim\ZZ$ is used only to define the ambiguity set; the implementable decision in \eqref{eq:DRO} remains constrained to lie in $\X$. The set $\bar\X$ may be chosen as $\X$ itself or as a simple compact outer approximation of $\X$ for computational convenience. The latter choice only enlarges the collection of test decisions and thus yields a more restrictive ambiguity set; the guarantees below continue to hold, with constants that depend on $\bar\X$. To keep the notation light, we state the subsequent theoretical results for the case $\bar\X=\X$. In this data-driven ambiguity set, the reference distribution $\PPhat$ is constructed from the available data, whose description depends on the specific applications. As we detail later, the hinge constraint ensures that the distribution $\QQ$ is close to $\PPhat$ in terms of the expected loss. The second-moment constraint ensures regularity of the distributions in the ambiguity set, enabling rigorous theoretical development. 

We consider a piecewise affine loss function of the form
\begin{equation}
\label{eq:pw_loss}
\ell(\x,\bxi)\coloneqq \max_{j\in[J]} \left(\bm a_j(\x)^\top\bxi+b_j(\x)\right),
\end{equation}
where \(\bm a_j(\x)\) and \(b_j(\x)\), \(j\in[J]\), are affine functions of \(\x\) given by
\begin{equation*}
\bm a_j(\x) \coloneqq \bm A_j\bm x+ \overline{\bm a}_j,
\qquad
b_j(\x)\coloneqq \bm b_j^\top\bm x + \overline b_j. 
\end{equation*}
This class is broad enough to capture many models arising in operations research and management science, and is standard in the distributionally robust optimization literature \citep{wiesemann2014distributionally,mohajerin2018data} because it often admits tractable conic reformulations.

It also includes two-stage linear loss functions as a special case:
\begin{equation}
\label{eq:two-stage_primal}
\begin{aligned}
\ell(\x,\bxi)\coloneqq\bc^\top\x + \inf_{\bm y\in\RR^{D_{\bm y}}}\;\;& \left\{\bm q^\top\bm y~:~\bm T(\bm x)\bm\xi + \bm h(\bm x)\leq \bm W\bm y\right\}
\end{aligned}
\end{equation}
where \(\bm T(\bm x)\) and \(\bm h(\bm x)\) are a matrix and a vector that depend affinely on \(\x\), and \(\bm W\) is a fixed recourse matrix. In this context, \(\x\) is a here-and-now decision made before the realization of uncertainty, while \(\bm y\) is a wait-and-see recourse decision that can be chosen after the uncertainty is realized.
The linear program in \eqref{eq:two-stage_primal} can be dualized, yielding the representation
\begin{equation}
\label{eq:two-stage_dual}
\begin{aligned}
\ell(\x,\bxi)=\bc^\top\x +\sup_{\bm p\in\RR_+^{L}}\;\;& \left\{(\bm T(\bm x)\bxi+\bm h(\bm x))^\top\bm p~:~\bm q=\bm W^\top\bm p\right\}
\end{aligned}
\end{equation}
Strong linear programming duality holds under the standard \emph{relatively complete recourse} assumption that the primal problem is feasible for any \(\x\in\X\) and \(\bxi\in\Xi\). Since the feasible set of the dual problem is a polyhedron with finitely many extreme points \(\{\bm p_1,\ldots,\bm p_Q\}\), the loss function can be rewritten in the piecewise affine form \eqref{eq:pw_loss}, where \(J=Q\), \(\bm a_j(\bm x)=\bm T(\bm x)^\top\bm p_j\), and \(b_j(\bm x)=\bm h(\bm x)^\top\bm p_j\).

We make the following assumptions:
\begin{enumerate}[(A)]
    \item \label{as:linear_growth}  \emph{Linear growth in $\bxi$}: There exist constants $c_1,c_2\in\RR_+$ such that 
    \begin{equation*}
    |\ell(\x,\bxi)| \leq c_1 + c_2 \|\bxi\|\qquad\forall \x\in\X \;\;\forall\bxi\in\Xi.
    \end{equation*}
    \item \label{as:second_moment} \emph{Second-moment budget}: The second-moment budget \(\Omega\) satisfies
\[
\EE_{\PPhat}[\|\txi\|^2]<\Omega
\qquad\text{and}\qquad
\EE_{\PP^\star}[\|\txi\|^2]\leq\Omega .
\]
\end{enumerate}
The first assumption is satisfied under the piecewise-affine loss function \eqref{eq:pw_loss}. The second assumption serves two purposes. The strict inequality for \(\PPhat\) ensures the Slater condition required for the generalized moment duality arguments; for empirical reference distributions, this part is not restrictive because \(\Omega\) can always be chosen strictly larger than the empirical second moment. The weak inequality for \(\PP^\star\) ensures that the true distribution satisfies the moment constraint in the ambiguity set, which is needed for coverage. If \(\Omega\) is calibrated from data, the same coverage guarantees hold on the corresponding high-confidence moment event, with an additional union-bound term. For clarity of exposition, we do not pursue this extension.

We first derive a fundamental property of this ambiguity set that will be useful in developing applications and solution schemes. 
\begin{proposition}\label{prop:semiinf_ambiguity_set}
The ambiguity set \eqref{eq:ambiguity_set} is equivalent to the infinitely-constrained ambiguity set: 
\begin{equation}
\label{eq:semiinf_ambiguity_set}
\mP_\epsilon\coloneqq \left\{\QQ\in\mathscr P(\Xi):\begin{array}{l}\EE_\QQ[\ell(\z,\txi)] - \EE_\PPhat[\ell(\z,\txi)]\leq \epsilon \quad\forall\z\in\X\\
\EE_\QQ[\|\txi\|^2]\leq \Omega
\end{array}\right\}.
\end{equation}
\end{proposition}

Proposition~\ref{prop:semiinf_ambiguity_set} clarifies the role of the expected hinge constraint in \eqref{eq:ambiguity_set}. The auxiliary distribution $\ZZ$ only serves to express the continuum of loss-discrepancy constraints over $\X$ as a single expected-hinge constraint. Under the standing continuity conditions, any full-support choice of $\ZZ$ induces the same exact ambiguity set. Thus, $\mP_\epsilon$ can be viewed as a feasibility region defined by an infinite family of loss-discrepancy inequalities together with a moment constraint. This semi-infinite perspective highlights a key advantage of our targeted formulation: rather than controlling a problem-agnostic statistical distance, it restricts discrepancies through the lens of the specific loss function. By enforcing $\EE_\QQ[\ell(\z,\txi)] - \EE_\PPhat[\ell(\z,\txi)]\leq \epsilon$ for every feasible decision $\z\in\X$, the formulation rules out adversarial distributions that could induce unexpectedly large losses for any feasible decision. Thus, the model does not require the adversarial distribution to remain close to the empirical distribution in all statistical directions, but only in those directions that are operationally relevant.

\begin{remark}[Task-aware discrepancy]
Compared with Wasserstein DRO, our targeted formulation natively captures task relevance. This advantage is reflected in two ways.
\begin{enumerate}[label=(\roman*)]
\item The structural benefit becomes particularly evident under heterogeneous uncertainty, such as when $\txi$ contains both continuous and discrete components. A Wasserstein formulation can also reflect such structure, but only after specifying a suitable representation and ground cost. Our formulation instead measures the induced loss discrepancies directly.
\item The issue is not the loss-dependent inner maximization itself, but the ambient Wasserstein geometry used to define and calibrate the ambiguity set before the inner maximization is performed. For instance, consider $\x=(x_1,x_2)\in\RR^2$ and $\bxi=(\xi_1,\xi_2)\in\RR^2$ with loss $
\ell(\x,\bxi)=|x_1-\xi_1|+\varrho |x_2-\xi_2|$, where $\varrho\gg1$. This example illustrates that the ambient Wasserstein geometry need not match the loss-induced profile. Since the loss is separable, any two distributions with identical marginals induce the same expected loss for every $\x$, even if their dependence structures are far apart in Wasserstein distance. Conversely, $\delta_{(0,0)}$ and $\delta_{(0,\varepsilon)}$ are only $\varepsilon$ apart under the usual Euclidean Wasserstein distance, but at $\x=(0,0)$ their expected losses differ by $\varrho\varepsilon$. Thus, the usual Wasserstein transport cost can overstate loss-irrelevant differences and understate loss-sensitive ones.
\end{enumerate}
\end{remark}

\subsection{Performance Guarantees}
In the next two subsections, we show that our DRO framework can be applied to a broad spectrum of problem classes, and establish that the radius $\epsilon$ can be systematically adjusted with the number of samples $N$ at the canonical $\widetilde{\mathcal O}(1/\sqrt{N})$ rate. The proofs of all subsequent coverage guarantees share the same two-step structure: (i) a \emph{pointwise} concentration inequality for a fixed decision $\z\in\X$, and (ii) an extension to a \emph{uniform} bound over all $\z\in\X$ via a covering number argument. We capture this common structure in the following proposition, so that each specific application needs only establish the pointwise bound in step~(i). 
\begin{proposition}
\label{prop:distribution_coverage}
Let $\X\subseteq\RR^{D_\x}$ be compact with $\sup_{\x\in \X}\|\x\|\leq R_\x$. Suppose for every fixed $\z\in\X$, we have the pointwise concentration inequality
\begin{equation}
\label{eq:pointwise_conc}
\EE_{\PP^\star}[\ell(\z,\txi)] - \EE_\PPhat[\ell(\z,\txi)]\leq \frac{1}{\sqrt{N}}\varepsilon\left(\log\left(\frac{1}{\delta}\right)\right)
\end{equation}
with probability at least $1-\delta$, where $\varepsilon:\RR_+\to\RR_+$ is nondecreasing and has at most polynomial growth. Let $L\coloneqq \max_{j\in[J]}\|\bm A_j\|_{\textup{op}}\sqrt{\Omega}+\max_{j\in[J]}\|\bm b_j\|$ be the Lipschitz constant from Lemma~\ref{lem:Lipschitz_condition}. Then, by setting the radius to
\begin{equation}
\label{eq:radius}
\epsilon\coloneqq \frac{1}{\sqrt{N}}\left(\varepsilon\left(D_\x\log\left(1+2{R_\x}L\sqrt{N}\right)+\log\left(\frac{1}{\delta}\right)\right)+2\right),
\end{equation}
one can ensure the distributional coverage  
\begin{equation*}
\PP^\star\in\mP_\epsilon
\end{equation*}
with probability at least $1-\delta$. 
\end{proposition}
The proof combines the Lipschitz estimate in Lemma~\ref{lem:Lipschitz_condition} with a covering argument; see Appendix~\ref{appendix:proofs_sec2}. The Lipschitz estimate controls nearby decisions, while monotonicity of $\varepsilon$ lets the pointwise bound be lifted over the cover. Polynomial growth then keeps the extra covering penalty at most polylogarithmic in $N$ and avoids a curse-of-dimensionality rate loss.

Proposition~\ref{prop:distribution_coverage} provides the link between pointwise concentration and the applications below. Once \eqref{eq:pointwise_conc} is established for any fixed decision $\z$, the covering argument in the proposition yields the uniform coverage condition $\PP^\star\in\mP_\epsilon$. Thus, in the canonical-rate cases, each application needs only to specify the reference estimator $\PPhat$ and the corresponding function $\varepsilon(\cdot)$ in \eqref{eq:pointwise_conc}; the radius is then determined by \eqref{eq:radius}. The same pointwise-to-uniform argument applies to estimators with different rates, with the rate reflected in the radius calibration.

In the following, we derive out-of-sample performance and suboptimality guarantees for the solutions to the DRO problem \eqref{eq:DRO}. 
To this end, we denote $\x^\star$ and $\hat\x$ as the optimal solutions of the true stochastic optimization problem and the DRO problem, respectively, i.e.,
\begin{equation}
\label{eq:minimizers}
\x^\star\in\arg\min_{\x\in\X}  \EE_{\PP^\star}[\ell(\x,\txi)]\quad\textup{and}\quad \hat\x\in\arg\min_{\x\in\X} \sup_{\QQ\in\mP_\epsilon}  \EE_{\QQ}[\ell(\x,\txi)].
\end{equation}
Note that the solution $\x^\star$ is attained since the mapping $\x\rightarrow \EE_{\PP^\star}[\ell(\x,\txi)]$ is continuous, and the feasible set $\X$ is compact. Furthermore, the objective function of the DRO problem constitutes a pointwise supremum of a family of lower semicontinuous (piecewise-affine) functions. Therefore, it is also lower semicontinuous, and the set of optimal solutions is nonempty \citep[Theorem~1.9]{Rockafellar2002}. That is, $\hat\x$ is attained. 

\begin{corollary}
\label{cor:OOS_guarantee}
Let $\hat J$ be the optimal value of the DRO problem \eqref{eq:DRO}. Then, setting the radius to \eqref{eq:radius}, we can ensure the out-of-sample performance guarantee
\begin{equation*}
\EE_{\PP^\star}[\ell(\hat\x,\txi)] \leq \hat J
\end{equation*}
with probability at least $1-\delta$. 
\end{corollary}

Remarkably, our DRO framework enables us to establish a guarantee on the excess risk of the distributionally robust solution, as stated in the following theorem. 
\begin{theorem}\label{thm:excess_risk}
Suppose the pointwise concentration inequality in Proposition \ref{prop:distribution_coverage} holds in terms of absolute error, i.e., 
\begin{equation}
\label{eq:twosided_pointwise_conc}
\left|\EE_{\PP^\star}[\ell(\z,\txi)] - \EE_\PPhat[\ell(\z,\txi)]\right|\leq \frac{1}{\sqrt{N}}\varepsilon\left(\log\left(\frac{1}{\delta}\right)\right)
\end{equation}
with probability at least $1-\delta$ for any fixed $\z\in\X$. 
 Then, setting the radius $\epsilon$ to \eqref{eq:radius}, we get
\begin{equation*}
\EE_{\PP^\star}[\ell(\hat\x,\txi)]\leq \EE_{\PP^\star}[\ell(\x^\star,\txi)] +  \frac{1}{\sqrt{N}}\varepsilon\left(\log\left(\frac{1}{\delta}\right)\right)+ \epsilon, 
\end{equation*}
with probability at least $1-2\delta$. 
\end{theorem}
This theorem shows that the excess risk is controlled by a single pointwise estimation term together with the ambiguity radius \(\epsilon\). Hence, whenever the radius in Proposition~\ref{prop:distribution_coverage} is calibrated at the order \(\widetilde{\mathcal O}(N^{-1/2})\), the robust solution attains the same \(\widetilde{\mathcal O}(N^{-1/2})\) excess-risk rate. Thus, the DRO solution enjoys the out-of-sample protection of the ambiguity set without incurring a slower excess-risk convergence rate relative to empirical risk minimization, up to logarithmic factors. In particular, this rate does not deteriorate with the ambient dimension \(D_\bxi\) of the uncertainty \(\txi\).

\begin{remark}[{Decision-dependent ambiguity sets}]
One could envisage an alternative decision-dependent ambiguity set
\begin{equation*}
\mP_\epsilon(\x)\coloneqq \left\{\QQ\in\mathscr P(\Xi):\begin{array}{l}\EE_\QQ[\ell(\x,\txi)] - \EE_\PPhat[\ell(\x,\txi)]\leq \epsilon\\
\EE_\QQ[\|\txi\|^2]\leq \Omega
\end{array}\right\}.
\end{equation*}
Then, for any fixed $\x\in\X$, \eqref{eq:pointwise_conc} implies the pointwise coverage guarantee
\begin{equation*}
\label{eq:pointwise_coverage}
\textup{Prob}\left(\PP^\star\in\mP_\epsilon(\x)\right)\geq 1-\delta
\end{equation*}
However, this pointwise coverage does not generally imply the uniform coverage property
\begin{equation*}
\label{eq:uniform_coverage}
\textup{Prob}\left(\PP^\star\in\mP_\epsilon(\x) \;\;\forall \x\in\X)\right)\geq 1-\delta,
\end{equation*}
which is typically required to deduce an out-of-sample guarantee for the data-dependent optimizer $\hat\x$. 
 Indeed, $\hat \x$ depends on the same sample used to form $\PPhat$, so a guarantee that holds for each fixed $\x$ need not hold at the random choice $\x=\hat \x$ without a uniform bound. 
\end{remark}

\subsection{Applications}
We present four applications of the framework: \emph{sub-Weibull losses}, \emph{Markovian data}, \emph{outlier-corrupted data} via median-of-means estimation, and \emph{incomplete data} via inverse probability weighting. In each case, we verify the pointwise concentration~\eqref{eq:pointwise_conc}, which calibrates the radius~\eqref{eq:radius} and yields the coverage and out-of-sample guarantees of Proposition~\ref{prop:distribution_coverage} and Corollary~\ref{cor:OOS_guarantee}. A further extension to contextual optimization is deferred to Appendix~\ref{sec:contextual_appendix}.

\subsection*{Sub-Weibull loss}
In many applications such as financial risk, insurance claims, and inventory problems, the random loss $\ell(\z,\txi)$ may exhibit substantial skewness and nontrivial tail behavior. At the same time, finite-sample DRO guarantees are often derived under light-tailed assumptions on the underlying uncertainty \citep{mohajerin2018data, gao2023finite}. It is therefore important to ask whether the same finite-sample guarantees can still be obtained under weaker tail assumptions.

Our framework extends naturally to losses with sub-Weibull tails, a broad family that unifies several tail regimes. It includes the sub-Gaussian and subexponential cases as special instances. The latter covers many familiar distributions, including the exponential, Gamma, $\chi^2$, Poisson, and geometric distributions, as well as the square of any sub-Gaussian random variable; see \cite[Example 2.8.7]{vershynin2025high}. It also extends to the heavier-than-exponential regime when $\vartheta\in(0,1)$. This allows us to treat light- and moderately heavy-tailed losses within a common framework.

For a real-valued random variable $X$ and $\vartheta>0$, we write
\[
\|X\|_{\psi_\vartheta}
\coloneqq
\inf\Bigl\{c>0:\EE\bigl[\exp\bigl((|X|/c)^\vartheta\bigr)-1\bigr]
\le 1\Bigr\}.
\]
Following \citet{kuchibhotla2022moving}, a centered random variable $X$ is called sub-Weibull of order $\vartheta\in(0,2]$ if $\|X\|_{\psi_\vartheta}<\infty$.
Under this parameterization, $\vartheta=2$ corresponds to the sub-Gaussian case, $\vartheta=1$ to the subexponential case, and $0<\vartheta<1$ to heavier-than-exponential tails. For $\vartheta<1$, the quantity $\|\cdot\|_{\psi_\vartheta}$ is only a quasi-norm, which makes sharp Bernstein-type concentration substantially more delicate.

The following proposition provides a pointwise concentration bound for the empirical loss under a sub-Weibull assumption.
\begin{proposition}\label{prop:subweibull_coverage}
Suppose that there exist $\vartheta\in(0,2]$ and $K>0$ such that
\[
\|\ell(\z,\txi)-\EE_{\PP^\star}[\ell(\z,\txi)]\|_{\psi_\vartheta}\le K
\qquad\forall \z\in\X.
\]
Then there exist constants $C_\vartheta,c_\vartheta>0$, depending only on $\vartheta$, such that the pointwise concentration~\eqref{eq:pointwise_conc} holds with $\varepsilon(y)=C_\vartheta K\sqrt{1+y}$. More precisely,
\begin{equation}\label{eq:eps_subweibull}
\EE_{\PP^\star}[\ell(\z,\txi)]-\EE_\PPhat[\ell(\z,\txi)]
\le
\frac{C_\vartheta K}{\sqrt N}\sqrt{1+\log\left(\frac{1}{\delta}\right)}
\end{equation}
with probability at least $1-\delta$, provided that
\[
\delta\in
\begin{cases}
\left[2\exp\!\left(-c_\vartheta N^{\vartheta/(2-\vartheta)}\right),\,1\right), & \vartheta\in(0,1),\\[2mm]
\left[2\exp\!\left(-c_\vartheta N\right),\,1\right), & \vartheta\in[1,2),\\[2mm]
(0,1), & \vartheta=2.
\end{cases}
\]
\end{proposition}

\begin{remark}[{Sub-Weibull confidence regime}]
Proposition~\ref{prop:subweibull_coverage} highlights several features of the sub-Weibull setting.
\begin{enumerate}[label=(\roman*)]
\item The restriction on $\delta$ comes from expressing the sub-Weibull concentration result in a pure $1/\sqrt{N}$ form. Over the confidence regime considered here, the generalized Bernstein inequality of \citet[Theorem~3.1]{kuchibhotla2022moving} reduces to the same pointwise concentration template used in Proposition~\ref{prop:distribution_coverage}. More generally, if the prescribed $\delta\in(0,1)$ lies outside this regime, an additional higher-order term of rate $\mathcal{O}(N^{-1/\vartheta})$ appears, which decays faster than the leading $\mathcal{O}(N^{-1/2})$ term for all $\vartheta\in(0,2)$. 
\item A key advantage of Proposition~\ref{prop:subweibull_coverage} is that it remains applicable under substantially heavier-tailed uncertainty distributions than those covered by existing finite-sample DRO guarantees. In particular, \citet{gao2023finite} relies on a transportation-information condition on the data-generating law, which implies standard sub-Gaussian-type tail assumptions,  while the classical Wasserstein guarantee of \citet{mohajerin2018data}, via \citet{fournier2015rate}, requires the stronger stretched-exponential moment condition $\EE[\exp(\|\txi\|^r)]<+\infty$ for some $r>1$ on the raw uncertainty. By contrast, Proposition~\ref{prop:subweibull_coverage} directly covers the broader sub-Weibull regime.
\end{enumerate}
\end{remark}

\subsection*{Markovian data} 
Our framework also accommodates settings in which the samples are not i.i.d. Consider the setting where ${\{\hat{\bxi}_i\}}_{i\in[N]}$ is a sample trajectory of a Markov chain with invariant distribution $\PP^\star$. In this case, our mean estimator $\EE_\PPhat[\ell(\z,\txi)]$ is no longer constructed using i.i.d. samples, which precludes the direct use of classical i.i.d.~concentration inequalities. Nevertheless, using Hoeffding's inequality for general Markov chains \citep{fan2021hoeffding}, we can obtain the desired pointwise concentration inequality. 

\begin{proposition}\label{prop:markov_coverage}
Assume the loss function is bounded: $\ell(\z,\bxi)\in[\lu,\lbar]$ for all $\z\in\X$ and $\bxi\in\Xi$. Suppose that the Markov chain is time-homogeneous with an absolute spectral gap of $1-\lambda>0$, where $\lambda\in[0,1)$ is the operator norm of the Markov transition kernel acting on the Hilbert space of square-integrable mean-zero functions under $\PP^\star$. Then, the pointwise concentration~\eqref{eq:pointwise_conc} holds with $\varepsilon(y)=\frac{\lbar-\lu}{2}\sqrt{\frac{2(1+\lambda)}{1-\lambda}y}$, i.e.,
\begin{equation}\label{eq:eps_markov}
\EE_{\PP^\star}[\ell(\z,\txi)] - \EE_\PPhat[\ell(\z,\txi)]\leq \frac{1}{\sqrt{N}}\frac{\lbar-\lu}{2}\sqrt{\frac{2(1+\lambda)}{1-\lambda}\log\left(\frac{1}{\delta}\right)}
\end{equation}
with probability at least $1-\delta$. The same result holds for time-inhomogeneous chains with $\lambda$ replaced by $\max_{i\in[N]}\lambda_i$, where $1-\lambda_i$ is the spectral gap of the Markov transition kernel for step $i$.
\end{proposition}

\begin{remark}[Markovian data with general state spaces] Our proposed DRO model applies to (time-inhomogeneous) Markovian data with general state spaces. We note that 
DRO with time-homogeneous Markovian data has been studied in \citep{li2021distributionally} in a discrete-state-space setting. Unlike ours, their statistical guarantee is asymptotic as it relies on large deviation principles; hence, it does not provide a prescription for adjusting the radius with respect to the sample size $N$.  
\end{remark}

\subsection*{Outlier-Corrupted Data}\label{sec:outlier}

In many practical settings, the observed data may be contaminated by outliers, i.e., samples that deviate arbitrarily from the underlying distribution $\PP^\star$. Under such corruption, the standard empirical-mean reference that uses all data points is sensitive to these atypical observations and may no longer reliably approximate $\EE_{\PP^\star}[\ell(\z,\txi)]$, potentially invalidating the coverage guarantee. To address this, we replace the empirical-mean reference with a \emph{Median-of-Means} (MoM) reference functional~\citep{laforgue2021generalization}, which retains a $1/\sqrt{N}$ convergence rate while remaining robust to a constant fraction of outliers.

\paragraph{MoM estimator.}
Given $N$ observations $\hat{\bxi}_1,\ldots,\hat{\bxi}_N$, fix an integer $K\geq 1$ and partition the index set $[N]$ into $K$ disjoint blocks $\mathcal{B}_1,\ldots,\mathcal{B}_K$ of equal size $B=\lfloor N/K\rfloor$, independently of the data. For each block $k\in[K]$, define the block mean
\begin{equation*}
f_k(\z)\coloneqq \frac{1}{B}\sum_{i\in\mathcal{B}_k}\ell(\z,\hat{\bxi}_i).
\end{equation*}
Then the MoM reference functional is defined as
\begin{equation*}\label{eq:mom_ref}
\widehat\mu_{\mathrm{MoM}}(\z)\coloneqq \mathrm{median}\bigl(f_1(\z),\ldots,f_K(\z)\bigr).
\end{equation*}
In the outlier-corrupted setting, we use $\widehat\mu_{\mathrm{MoM}}(\z)$ in place of $\EE_\PPhat[\ell(\z,\txi)]$ as the reference functional in the loss constraints.

Suppose the sample points $\{\hat{\bxi}_1,\ldots,\hat{\bxi}_N\}$ contain $N-n_\mathrm{o}$ inliers drawn i.i.d.\ from $\PP^\star$, and $n_\mathrm{o}$ arbitrary outliers. We assume that the fraction of outliers is $\omega\coloneqq n_\mathrm{o}/N\in[0,1/2)$.
Following~\citep{laforgue2021generalization} with the arithmetic-mean calibration, we define the inflation constant
\[
\Gamma(\omega)\coloneqq\frac{\sqrt{2(1+2\omega)}}{(1-2\omega)^{3/2}},
\]
which satisfies $\Gamma(\omega)\to\infty$ as $\omega\to 1/2$.

\begin{proposition}[Proposition~2 (eq.~(3)) of~\citet{laforgue2021generalization}]\label{prop:mom_coverage}
Suppose the inlier loss has bounded variance: $\var_{\PP^\star}[\ell(\z,\txi)]\leq\sigma^2$ for all $\z\in\X$. Set $K=\Bigl\lceil\tfrac{4(1+2\omega)}{(1-2\omega)^2}\log\tfrac{1}{\delta}\Bigr\rceil$ and let $\delta\in\Bigl[e^{-(1-2\omega)^2N/(4(1+2\omega))},\,e^{-(1-2\omega)^2N/8}\Bigr]$.  Then, the MoM reference satisfies the pointwise concentration bound
\[
\EE_{\PP^\star}[\ell(\z,\txi)]-\widehat\mu_{\mathrm{MoM}}(\z)\leq 4\sqrt{e}\,\sigma\,\Gamma(\omega)\sqrt{\frac{1+\log(1/\delta)}{N}}.
\]
 with probability at least $1-\delta$.
\end{proposition}
We note that the covering argument in Proposition \ref{prop:distribution_coverage} carries over under this replacement because it only requires a pointwise concentration bound and a Lipschitz reference functional. The latter property holds for $\widehat\mu_{\mathrm{MoM}}$: the median satisfies $|\mathrm{median}(\bm{a})-\mathrm{median}(\bm{b})|\leq\max_k|a_k-b_k|$ for any two vectors $\bm a,\bm b\in\RR^K$, and each block mean satisfies $|f_k(\z)-f_k(\z')|\leq L\|\z-\z'\|$ by the pathwise Lipschitz bound $|\ell(\z,\bxi)-\ell(\z',\bxi)|\leq L\|\z-\z'\|$ (see proof of Lemma~\ref{lem:Lipschitz_condition}).

This replacement does not require interpreting $\widehat\mu_{\mathrm{MoM}}$ as an expectation under a new empirical distribution. When we later derive the exact dual reformulation, the same generalized moment problem argument applies to the MoM-based ambiguity set as long as this set satisfies the usual Slater condition for the chosen radius $\epsilon$. Under this condition, strong duality and dual attainment hold with $\widehat\mu_{\mathrm{MoM}}(\z)$ replacing $\EE_{\PPhat}[\ell(\z,\txi)]$ in the dual objective.

\begin{remark}[{Outlier robustness through robust estimators}]
Existing DRO approaches to outlier-corrupted data use problem-specific constructions tailored to particular outlier structures: mixed Wasserstein-TV ambiguity sets \citep{nietert2023outlier,nietert2024robust}, concave-cost optimal transport \citep{blanchet2024automatic}, and distributionally favorable paradigms \citep{jiang2024distributionally}. Our TIPM framework offers a unified alternative: since the coverage guarantee depends only on a scalar concentration inequality, any statistically robust estimator (such as MoM) can be substituted directly, preserving tractability and bypassing the curse of dimensionality.
\end{remark}

\subsection*{Incomplete Data}
Our framework is applicable to settings where a subset of the random parameters is not always observable. This is relevant to many selection problems, such as company hiring, college admissions, and loan approval, where outcomes (e.g., high-performing candidates, qualified students, or applicants who could repay the loan) are only observed for selected individuals. Formally, we define the random parameters as $\txi=(\tchi,\tomega)\in\mathscr X\times \Omega_{IPW}$, where $\tchi$ are entries that are always observable, while $\tomega$ are those that will be observed only if the candidate is selected. 

Let $\tilde S\in\{0,1\}$ denote the historical selection outcome with $\tilde S = 1$ if the candidate is selected to advance and $\tilde S = 0$ otherwise.  Define $\overline\PP$ to be the joint distribution of $(\tilde S,\tchi,\tomega)$.  This binary random variable is governed by a logging policy
\begin{equation*}
\pi(\bchi)\coloneqq \overline\PP(\tilde S=1|\tchi=\bchi),
\end{equation*}
which represents the probability that a candidate with covariate $\bchi$ is selected. Based on the logging policy, the available dataset is given by
\begin{equation*}
\{(\bchi_n,\bomega_n)\}_{n\in\mathcal I}\cup\{\bchi_n\}_{n\in[N]\setminus\mathcal I},
\end{equation*}
where $\mathcal I$ is the set of candidates who were selected and have the outcome observed. Under the mild conditional exchangeability and positivity assumptions~\citep[Assumptions 1 and 2]{jia2024learning}, the inverse probability weighting estimator of the expected loss is given by 
\begin{equation}
\label{eq:IPW_estimator}
\EE_{\PPhat}[\ell(\z,\txi)]\coloneqq \frac{1}{N} \sum_{n\in[N]}\frac{\I[S_n=1]}{\pi(\bchi_n)} \ell(\z,\bxi_n), 
\end{equation}
where each full sample point is weighted by the reciprocal of its propensity score $\pi(\bchi_n)$. We use this IPW estimator as the reference functional in the loss constraints and obtain the following concentration bound.

\begin{proposition}\label{prop:incomplete}
Suppose conditional exchangeability and positivity hold, with $\pi(\bchi)\ge \underline\pi>0$, and suppose the losses are uniformly sub-Gaussian, i.e., $\|\ell(\z,\txi)\|_{\psi_2}\le\sigma$ for all $\z\in\X$. The pointwise concentration \eqref{eq:pointwise_conc} holds with $\varepsilon(y)=C_{\mathrm{IPW}}\frac{\sigma}{\underline\pi}\sqrt{y}$ for a universal constant $C_{\mathrm{IPW}}>0$, i.e., we have
\begin{equation*}
\EE_{\PP^\star}[\ell(\z,\txi)] - \EE_\PPhat[\ell(\z,\txi)]\leq C_{\mathrm{IPW}}\frac{\sigma}{\underline\pi}\sqrt{\frac{\log(1/\delta)}{N}}
\end{equation*}
with probability at least $1-\delta$. 
\end{proposition}

\subsection{Dual Formulations}


In this subsection, we derive an exact dual reformulation of the DRO problem \eqref{eq:DRO} and establish its convergence to empirical risk minimization as the ambiguity parameter $\epsilon$ approaches $0$. We start with the reformulation. 
We denote by $\C(\X)$ the space of continuous functions on $\X$, by $\C_+(\X)\coloneqq\{f\in\C(\X):f(\x)\geq 0\;\forall \x\in\X\}$ the cone of nonnegative continuous functions, and by $\M_+(\X)$ the cone of finite nonnegative Borel measures on $\X$.
\begin{theorem}\label{thm:dual_exact}
Fix $\epsilon>0$. The DRO problem \eqref{eq:DRO} is equivalent to the infinite linear program 
\begin{equation} 
\label{eq:dual_exact}
\begin{aligned}
\min \;&  \alpha  + \beta\Omega  + \int_\X\left(\EE_\PPhat[\ell(\z,\txi)]+ \epsilon\right)\nu(\mathrm d\z) \\
\st & \x\in\X, \; \alpha\in\RR,\;\beta\in\RR_+,\;\nu\in\M_+(\X)\\
&\ell(\x,\bxi)\leq\alpha+\beta\|\bxi\|^2+\int_\X\ell(\z,\bxi)\nu(\mathrm d\z)\qquad\forall\bxi\in\Xi.
\end{aligned}
\end{equation}
Moreover, the infimum in \eqref{eq:dual_exact} is attained by some $(\x,\alpha,\beta,\nu)\in\X\times\RR\times\RR_+\times\M_+(\X)$. 
\end{theorem}
The exact formulation should be viewed as a characterization of the target DRO model rather than the computational endpoint. It exposes the dual structure induced by the ambiguity set. The implementable model is the sampled approximation developed in Section~\ref{sec:mc_sampling}: it replaces the continuum of test decisions by finitely many sampled decisions, retains the conservative outer-approximation property, converges to the exact DRO model, and enables finite conic reformulations.

We close the section by showing that the optimal value and solutions of the DRO problem converge to those of the (non-robust) empirical risk minimization problem as $\epsilon$ approaches $0$. 
\begin{proposition}
\label{prop:convergence_eps_to_0}
As $\epsilon\rightarrow 0$, the optimal value of \eqref{eq:dual_exact} converges to that of the empirical risk minimization problem
\begin{equation}
\label{eq:ERM}
\min_{\x\in\X}   \EE_{\PPhat}[\ell(\x,\txi)].
\end{equation}
In addition, every cluster point $\hat \x^\star$ of a sequence $\{\hat\x_\epsilon\}_{\epsilon\downarrow 0}$ of minimizers for \eqref{eq:dual_exact} is a minimizer for \eqref{eq:ERM}. 
\end{proposition}

\section{A Monte Carlo Sampling Approach}\label{sec:mc_sampling}

The exact formulation remains difficult to solve: the ambiguity set contains an expectation over the continuous auxiliary distribution $\ZZ$, while the dual reformulation \eqref{eq:dual_exact} contains a measure-valued multiplier $\nu$. By Proposition~\ref{prop:semiinf_ambiguity_set}, the DRO problem \eqref{eq:DRO} can also be viewed as a problem with an infinitely constrained ambiguity set. Such problems have been studied by \citet{chen2019distributionally}, who develop a convergent cutting-plane algorithm for compact-support settings. We instead develop a Monte Carlo sampling approximation that replaces the continuum of test decisions with randomly sampled ones, yielding a conservative finite approximation with explicit convergence and suboptimality guarantees. As shown in Section~\ref{sec:unbounded_support}, the same sampling scheme can also be extended to unbounded-support settings.

We draw $M$ samples from $\ZZ$ and obtain the approximate ambiguity set:
\begin{equation}
\label{eq:sample_ambiguity_set}
\mP^M_\epsilon\coloneqq \left\{\QQ\in\mathscr P(\Xi):
\begin{array}{l}\displaystyle\frac{1}{M}\sum_{m\in[M]}\left[\EE_\QQ[\ell(\z_m,\txi)]-\EE_\PPhat[\ell(\z_m,\txi)]-\eps\right]_+ \leq 0\\
\EE_\QQ[\|\txi\|^2]\leq \Omega
\end{array}
\right\}.
\end{equation}
When direct sampling from $\X$ is inconvenient, we use the outer test-decision set $\bar\X$ discussed in Section~\ref{sec:IPMDRO}; the samples $\z_m$ below are then drawn from the corresponding full-support distribution on $\bar\X$. The containment, convergence, and suboptimality guarantees in this section continue to hold under this convention, with constants evaluated on $\bar\X$.
By construction, the constraint in \eqref{eq:sample_ambiguity_set} enforces the loss-discrepancy bound only at the sampled test decisions \(\{\z_m\}_{m=1}^M\). Hence \(\mP_\epsilon^M\supseteq\mP_\epsilon\), and the sampled DRO problem is a conservative approximation of the exact DRO problem. In particular, the out-of-sample guarantee in Corollary~\ref{cor:OOS_guarantee} remains valid. Unlike classical sample-average approximations for stochastic programming, which often introduce optimistic bias, this approximation is pessimistically biased because it relaxes the ambiguity-set constraints and therefore enlarges the adversary's feasible set. This conservative structure is also central to the convergence and suboptimality guarantees developed next.

Using the ambiguity set \eqref{eq:sample_ambiguity_set}, we obtain the following dual reformulation: 
\begin{equation} 
\label{eq:dual_SAA}
\begin{aligned}
\min \;&  \alpha  + \beta\Omega + \frac{1}{M}\sum_{m\in[M]}\nu_m(\EE_\PPhat[\ell(\z_m,\txi)]+ \epsilon)  \\
\st & \x\in\X, \; \alpha\in\RR,\;\beta\in\RR_+,\;\bm\nu\in\RR_+^M\\
&\ell(\x,\bxi)\leq\alpha+\beta\|\bxi\|^2+\frac{1}{M}\sum_{m\in[M]}\nu_m\ell(\z_m,\bxi) \qquad\forall\bxi\in\Xi.
\end{aligned}
\end{equation}
This formulation constitutes the sample-average counterpart of Theorem~\ref{thm:dual_exact}. 
By replacing the expectation under $\ZZ$ with an empirical average over $\{\z_m\}_{m=1}^M$, the functional multiplier $\nu(\cdot)$ reduces to a finite-dimensional vector $\bm\nu\in\RR_+^M$. 

In the following, we develop the theoretical properties and tractable conic reformulations of this approximation scheme. 
 We first consider the case where the support set $\Xi$ is compact. 

\subsection{Theoretical Guarantees}

Our first theoretical result establishes a high-confidence guarantee for the containment of the sample-average-based ambiguity set $\mP_\epsilon^M$ within a relaxed exact ambiguity set 
\begin{equation}
\label{eq:eta_ambiguity_set}
\mP_\epsilon(\eta)\coloneqq \left\{\QQ\in\mathscr P(\Xi):
\begin{array}{l}\EE_\ZZ\left[\left[\EE_\QQ[\ell(\tz,\txi)]-\EE_\PPhat[\ell(\tz,\txi)]-\eps\right]_+\right] \leq \eta\\\;\EE_\QQ[\|\txi\|^2]\leq \Omega
\end{array}
\right\}
\end{equation}
parameterized by $\eta>0$.
\begin{theorem}
\label{coro:set_containment}
Assume $\sup_{\x\in \X}\|\x\|\leq R_\x$ and $\sup_{\bxi\in \Xi}\|\bxi\|\leq R_\bxi$, and let 
\begin{equation*}
R\coloneqq \max_{j\in[J]}\left\|\begin{bmatrix}&\bm A_j^\top &\bm b_j & \\ &\overline {\bm a}_j^\top & \overline b_j\end{bmatrix}\right\|_{\textup{op}}\sqrt{R_\bxi^2+1}\sqrt{R_\x^2+1}. 
\end{equation*}
 Then, setting 
\begin{equation}\label{eq:eta}
\eta\coloneqq \mathcal O\left( R\sqrt{\frac{J\log J(\log M)^3}{M}}\right) +2(c_1 + c_2\sqrt{\Omega})  \sqrt{\frac{1}{2M}\log\left(\frac{1}{\tau}\right)}, 
\end{equation}
we can ensure with probability at least $1-\tau$:
\begin{equation*}
 \mP_\epsilon^M \subseteq \mP_\epsilon(\eta). 
\end{equation*}
\end{theorem}
This result shows that using the sample-average-based ambiguity set is akin to relaxing the expected hinge constraint by $\eta\coloneqq \widetilde{\mathcal O}(1/\sqrt{M})$.  
The result relies on a generalization bound for the sample-average approximation errors, stated as Proposition~\ref{prop:hinge_SAA_guarantee} in Appendix~\ref{appendix:proofs_sec3}. In the proof of that proposition, we bound the Rademacher complexity of the function class
\begin{equation*}
\HH\coloneqq \left\{\z\mapsto[\EE_\QQ[\ell(\z,\txi)]-\EE_\PPhat[\ell(\z,\txi)]-\eps]_+:\QQ\in\mathscr P(\Xi),\;\EE_\QQ[\|\txi\|^2]\leq \Omega\right\},
\end{equation*}
which contains functions parametrized by distributions $\QQ\in\mathscr P(\Xi)$ satisfying the second-moment constraint $\EE_\QQ[\|\txi\|^2]\leq\Omega$. Unlike in traditional settings, the parameter space is infinite-dimensional, rendering classical techniques inapplicable. We employ the contraction principle to reduce the analysis to the Rademacher complexity of the functions $\z\mapsto\EE_\QQ[\ell(\z,\txi)]$. Using our loss function structure, we further reduce it to the complexity of $J$-fold maxima of hyperplanes \cite[Corollary 5]{attias2024fat}. This function-class argument is the key ingredient behind Theorem~\ref{coro:set_containment}.  


\begin{remark}[Role of function-class structure]
The dimension-free $\widetilde{\mathcal O}(1/\sqrt{M})$ sampling rate in Proposition~\ref{prop:hinge_SAA_guarantee} is a key benefit of exploiting the piecewise-affine structure of the loss class, rather than a generic consequence of Lipschitz continuity. Indeed, for a generic class of bounded $L$-Lipschitz functions on a metric space with dimension $D_\x$ of the decision $\x$, the empirical Rademacher complexity can scale as $\mathcal O(L/M^{1/(D_{\x}+1)})$; see \cite[Theorem 4.3]{gottlieb2016adaptive}. In contrast, the piecewise-affine loss class considered here reduces to the class of $J$-fold maxima of hyperplanes, for which the sharper Rademacher bound in Proposition~\ref{prop:hinge_SAA_guarantee} is available.
\end{remark}

We next establish the convergence of the optimal value and optimal solution of the approximate DRO problem. Let $\hat v$ and $\hat v_M$ be the optimal values of the exact problem and the approximate problem, respectively, i.e., 
\begin{equation*}
\hat v\coloneqq\min_{\x\in\X} \sup_{\QQ\in\mP_\epsilon}  \EE_{\QQ}[\ell(\x,\txi)]\quad\textup{and}\quad \hat v_M\coloneqq\min_{\x\in\X} \sup_{\QQ\in\mP^M_\epsilon}  \EE_{\QQ}[\ell(\x,\txi)].
\end{equation*}
We define $\mathcal S$ as the set of minimizers of the exact problem. The minimizers $\hat\x\in\mathcal S$ and $\hat\x_M$ of the sampled problem are attained because the corresponding objectives are lower semicontinuous on the compact set $\X$.


\begin{theorem}
\label{thm:approximation_convergence}
As $M\rightarrow \infty$, the following convergences in probability hold: 
\begin{equation*}
\hat v_M\overset{p}{\to} \hat v \quad\textup{ and } \quad\textup{dist}(\hat\x_M,\mathcal S)\overset{p}{\to} 0.
\end{equation*}
Furthermore, assume there exists an optimal dual measure $\nu$ in Theorem \ref{thm:dual_exact} such that $\gamma\ZZ-\nu\in\M_+(\X)$ for some $\gamma\in\RR_+$. Let $\gamma^\star\coloneqq\inf\{\gamma\in\RR_+:\gamma\ZZ-\nu\in\M_+(\X)\}$.  
Then, the following suboptimality bound holds with probability at least $1-\tau$:
\begin{equation*}
\begin{aligned}
\sup_{\QQ\in\mP_\epsilon}  \EE_{\QQ}[\ell(\hat \x_M,\txi)] &\leq \min_{\x\in\X} \sup_{\QQ\in\mP_\epsilon}  \EE_{\QQ}[\ell(\x,\txi)] \\
&\qquad+ \gamma^\star\left(\mathcal O\left( R\sqrt{\frac{J\log J(\log M)^3}{M}}\right)+ 2(c_1 + c_2\sqrt{\Omega}) \sqrt{\frac{1}{2M}\log\left(\frac{1}{\tau}\right)}\right).
\end{aligned}
\end{equation*}
\end{theorem}

\begin{remark}[Existence of the dominance constant]\label{rem:dominance_constant}
The domination condition used in the suboptimality bound of Theorem~\ref{thm:approximation_convergence} is a condition on the auxiliary sampling distribution, not on the underlying data-generating model. Since $\ZZ$ is chosen by the modeler, it can be enriched, if needed, to dominate the relevant optimal dual measure, as we now show. Specifically, if $\nu$ is any optimal dual measure in Theorem~\ref{thm:dual_exact}, then for any $\lambda\in(0,1)$, the full-support mixture $\ZZ_\lambda\coloneqq(1-\lambda)\ZZ+\lambda\nu/\nu(\X)$ satisfies $(\nu(\X)/\lambda)\ZZ_\lambda-\nu\in\M_+(\X)$, and hence $\gamma_\lambda^\star\leq\nu(\X)/\lambda<\infty$. Thus the required dominance constant can always be made finite through the modeling choice of $\ZZ$. By Proposition~\ref{prop:semiinf_ambiguity_set}, all full-support choices of the auxiliary distribution induce the same semi-infinite ambiguity set, so replacing $\ZZ$ with $\ZZ_\lambda$ does not alter the DRO problem itself.
\end{remark}

\begin{figure}[t]
    \centering
    \includegraphics[width=0.5\textwidth]{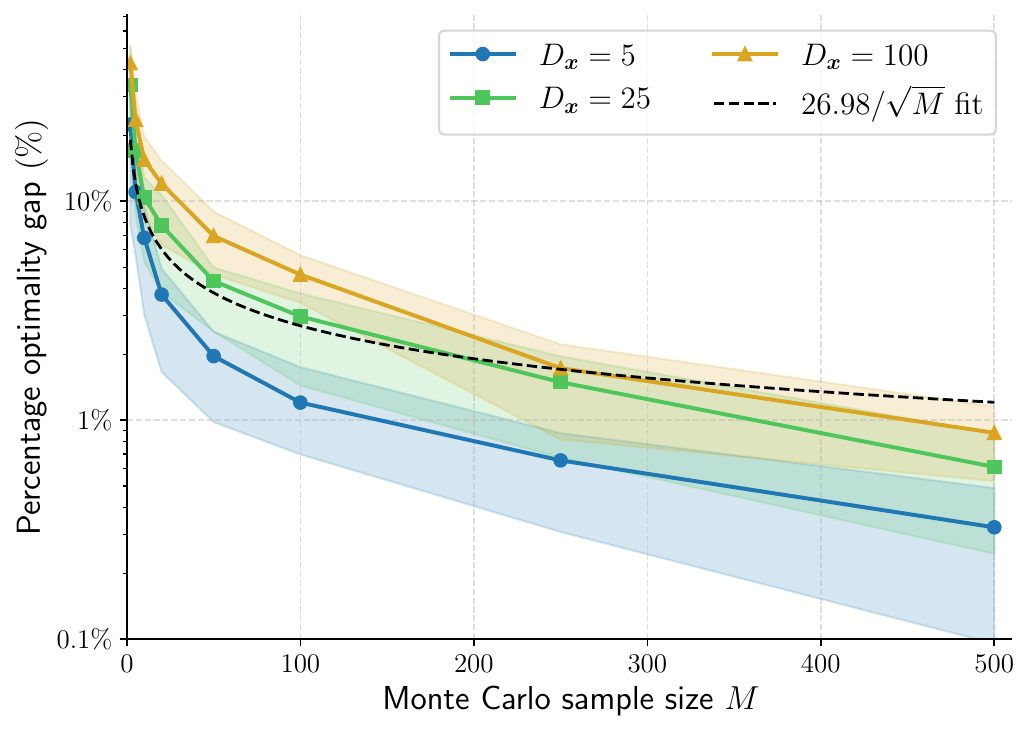}
    \caption{Percentage optimality gap of the sampled-dual solution as a function of the Monte Carlo sample size $M$.  The dashed curve is a fitted $C/\sqrt{M}$ reference line.}
    \label{fig:mc-suboptimality-gap}
\end{figure}

\begin{remark}[Numerical illustration of the Monte Carlo approximation]\label{rem:mc_numerical_illustration}
We provide a simple experiment to illustrate the sampling behavior in Theorem~\ref{thm:approximation_convergence}. The detailed experimental design is reported in Appendix~\ref{appendix:mc_numerical_illustration}. We use unit-ball decision and uncertainty sets, together with a shifted piecewise-affine loss and a symmetric empirical distribution, so that the exact optimizer $\x_\star$ is known and nonzero. This allows us to evaluate the sampled-dual solution $\widehat \x_M$  directly, without computing a large-reference Monte Carlo solution. We report the relative optimality gap
\[
    \frac{
    \EE_\PPhat[\ell(\widehat \x_M,\tilde\bxi)]
    -
    \EE_\PPhat[\ell(\x_\star,\tilde\bxi)]
    }{
    \EE_\PPhat[\ell(\x_\star,\tilde\bxi)]
    }
    \times 100\%
\]
across different Monte Carlo sample sizes $M$. 
To quantify the observed decay, we fit a single \(C/\sqrt{M}\) reference curve to the empirical curves across the three dimensions, following the \(\widetilde{\mathcal O}(1/\sqrt{M})\) sampling rate established in Theorem~\ref{thm:approximation_convergence}. The fitted constant is \(C = 26.98\).  As Figure~\ref{fig:mc-suboptimality-gap} shows, the empirical gaps decay at least as fast as this reference curve; for larger Monte Carlo sample sizes, the observed decay is even faster than the theoretical rate benchmark. 
At the largest tested Monte Carlo sizes, the gaps are reduced to the one-percent level across all dimensions. 
This provides numerical evidence that the Monte Carlo approximation improves as \(M\) increases, and that our method does not suffer from the curse of dimensionality.
\end{remark}

\subsection{Conic Programming Reformulations}

In this subsection, we show that the approximate problem admits a tractable conic programming reformulation. Tractable approximations for two-stage problems are provided in Appendix \ref{appendix:conic_two_stage}.

\begin{theorem}\label{theorem:Conic_reformulation}
Suppose that $\X$ and $\Xi$ are convex. Then problem \eqref{eq:dual_SAA} is equivalent to the following finite convex program:
\begin{equation} 
\begin{aligned}
\min \;&  \alpha  + \beta\Omega + \frac{1}{M}\sum_{m\in[M]}\nu_m(\EE_\PPhat[\ell(\z_m,\txi)]+ \epsilon)  \\
\st & \x\in\X, \; \alpha\in\RR,\;\beta\in\RR_+,\;\bm\nu\in\RR_+^M\\
& \bm\lambda_{jm}\in\RR_+^J\;\;\forall m\in[M],\quad \bm\theta_j\in\RR^{D_\bxi},\zeta_j\in\RR_+\;\;\forall j\in[J]\\
&\mathbf e^\top\bm\lambda_{jm}=\nu_m \;\;\forall m\in[M]\;\;\forall j\in[J]\\
&  \zeta_j+\sigma_\Xi(\bm\theta_j)+b_j(\x)-\frac{1}{M}\sum_{m\in[M],k\in[J]} \lambda_{jm}^k b_k(\z_m)\leq \alpha \quad \forall j\in[J]\\
& \left\|\begin{bmatrix}\bm a_j(\x)-\frac{1}{M}\sum_{m\in[M],k\in[J]} \lambda_{jm}^k \bm a_k(\z_m)-\bm\theta_j\\\zeta_j-\beta\end{bmatrix}\right\|\leq\zeta_j+\beta\quad\forall j\in[J]. 
\end{aligned}
\end{equation}
\end{theorem}
In particular, if $\X$ and $\Xi$ are second-order cone representable, then the reformulation above is a second-order cone program. 

\section{Extensions to Unbounded Support}\label{sec:unbounded_support} 

We now extend the previous results to the setting where the distributional support is \(\Xi\coloneqq\RR^{D_\bxi}\). We assume that there exist constants \(K_{\mathrm{sw}},\vartheta>0\) such that \[ \PP^\star(\|\txi\|\le t)\ge 1-2\exp\!\left(-\left(t/K_{\mathrm{sw}}\right)^\vartheta\right) \qquad \forall t>0. \] This is a standard sub-Weibull-type tail bound \cite[Definition 2.2]{kuchibhotla2022moving}. For a given threshold \(t>0\), we incorporate this tail information into the ambiguity set: \begin{equation} \label{eq:ambiguity_set_subweibull} \mP'_\epsilon\coloneqq \left\{\QQ\in\mathscr P(\Xi): \begin{array}{l} \EE_\ZZ\left[\left[\EE_\QQ[\ell(\tz,\txi)]-\EE_\PPhat[\ell(\tz,\txi)]-\eps\right]_+\right] \leq 0\\ \EE_\QQ[\|\txi\|^2]\leq \Omega,\;\; \QQ(\|\txi\|\le t)\ge 1-2\exp\!\left(-\left(t/K_{\mathrm{sw}}\right)^\vartheta\right) \end{array} \right\}. \end{equation}
The tail constraint in \eqref{eq:ambiguity_set_subweibull} is satisfied by \(\PP^\star\) by assumption, and the second-moment constraint is satisfied by Assumption~\ref{as:second_moment}. Hence, calibrating \(\epsilon\) as in Proposition~\ref{prop:distribution_coverage} guarantees the high-confidence containment \(\PP^\star\in\mP'_\epsilon\).


We next show that the proposed TIPM-DRO framework extends beyond compact support. Under the single-threshold tail bound implied by sub-Weibull tails above, we derive an exact dual reformulation, establish that the Monte Carlo approximation remains conservative, and show that the same \(\widetilde{\mathcal O}(1/\sqrt{M})\) sampling principle continues to hold up to explicit tail-dependent terms. We also obtain a tractable conic reformulation.

The following proposition provides the dual formulation of the DRO problem under this ambiguity set. 
\begin{proposition}\label{thm:dual_exact_subweibull}
The DRO problem \eqref{eq:DRO} with the ambiguity set \eqref{eq:ambiguity_set_subweibull} is equivalent to the infinite linear program
\begin{equation} 
\label{eq:dual_exact_subweibull}
\begin{aligned}
\min \;&  \alpha  + \beta\Omega  -\kappa \left(1-2\exp\left(-\left(t/K_{\mathrm{sw}}\right)^\vartheta\right)\right) + \int_\X\left(\EE_\PPhat[\ell(\z,\txi)]+ \epsilon\right)\nu(\mathrm d\z)  \\
\st & \x\in\X, \; \alpha\in\RR,\;\beta,\kappa\in\RR_+,\;\nu\in\M_+(\X)\\
&\ell(\x,\bxi)+\kappa\I_{\{\|\bxi\|\leq t\}}\leq\alpha+\beta\|\bxi\|^2+\int_\X\ell(\z,\bxi)\nu(\mathrm d\z)\qquad\forall\bxi\in\Xi.
\end{aligned}
\end{equation}
Moreover, the infimum in \eqref{eq:dual_exact_subweibull} is attained by some $(\x,\alpha,\beta,\kappa,\nu)\in\X\times\RR\times\RR_+^2\times\M_+(\X)$. 
\end{proposition}

Drawing $M$ samples from $\ZZ$, we obtain an outer approximation of the ambiguity set:
\begin{equation}
\label{eq:sample_ambiguity_set_subweibull}
\mP'^M_\epsilon\coloneqq \left\{\QQ\in\mathscr P(\Xi):
\begin{array}{l}\displaystyle\frac{1}{M}\sum_{m\in[M]}\left[\EE_\QQ[\ell(\z_m,\txi)]-\EE_\PPhat[\ell(\z_m,\txi)]-\eps\right]_+ \leq 0\\
\EE_\QQ[\|\txi\|^2]\leq \Omega, \;\; \QQ(\|\txi\| \leq t ) \geq 1-2\exp\left(-\left(t/K_{\mathrm{sw}}\right)^\vartheta\right)
\end{array}
\right\}.
\end{equation}
The corresponding dual formulation is given by
\begin{equation} 
\label{eq:dual_SAA_subweibull}
\begin{aligned}
\min \;&  \alpha  + \beta\Omega  -\kappa \left(1-2\exp\left(-\left(t/K_{\mathrm{sw}}\right)^\vartheta\right)\right) + \frac{1}{M}\sum_{m\in[M]}\nu_m(\EE_\PPhat[\ell(\z_m,\txi)]+ \epsilon)  \\
\st & \x\in\X, \; \alpha\in\RR,\;\beta,\kappa\in\RR_+,\;\bm\nu\in\RR_+^M\\
&\ell(\x,\bxi)+\kappa\I_{\{\|\bxi\|\leq t\}}\leq\alpha+\beta\|\bxi\|^2+\frac{1}{M}\sum_{m\in[M]}\nu_m\ell(\z_m,\bxi) \qquad\forall\bxi\in\Xi.
\end{aligned}
\end{equation}



\subsection{Theoretical Guarantees}
We now establish the theoretical guarantees for the Monte Carlo approximation under the sub-Weibull tail condition. 
For notational clarity, given a threshold sequence $t_M>0$, we denote by $\mP'_{\epsilon,M}$ the exact ambiguity set \eqref{eq:ambiguity_set_subweibull} with $t=t_M$. The corresponding relaxed moving-threshold set is
\begin{equation}
\label{eq:eta_ambiguity_set_subweibull}
\mP'_{\epsilon,M}(\eta)\coloneqq \left\{\QQ\in\mathscr P(\Xi):
\begin{array}{l}\EE_\ZZ\left[\left[\EE_\QQ[\ell(\tz,\txi)]-\EE_\PPhat[\ell(\tz,\txi)]-\eps\right]_+\right] \leq \eta\\
\EE_\QQ[\|\txi\|^2]\leq \Omega, \;\; \QQ(\|\txi\| \leq t_M ) \geq 1-2\exp\left(-\left(t_M/K_{\mathrm{sw}}\right)^\vartheta\right)
\end{array}
\right\},
\end{equation}
parameterized by $\eta>0$. 
In this section, we set $t=t_M\coloneqq K_{\mathrm{sw}}(\log M)^{1/\vartheta}$.
The first theoretical result establishes a high-confidence guarantee for the containment of the sample-average-based ambiguity set $\mP'^M_\epsilon$ within this relaxed exact ambiguity set. 
\begin{theorem} \label{thm:containment_unbounded_Xi}  Assume $\sup_{\x\in \X}\|\x\|\leq R_\x$, and let 
\begin{equation*}
R(t)\coloneqq \left(\max_{j\in[J]}\left\|\begin{bmatrix}\bm A_j^\top & \bm b_j\\ \overline {\bm a}_j^\top & \overline b_j\end{bmatrix}\right\|_{\textup{op}}\sqrt{ t^2+1}\right)\sqrt{R_\x^2+1},\qquad t>0. 
\end{equation*}
Then, setting
\begin{equation*}
\begin{aligned}
\eta_M\coloneqq\mathcal O\left( R(t_M)\sqrt{\frac{J\log J(\log M)^3}{M}}\right)+ 2(c_1 + c_2\sqrt{\Omega})\sqrt{\frac{1}{2M}\log\left(\frac{1}{\tau}\right)} + \frac{4c_1}{M} + 2c_2 \sqrt{ \frac{2\Omega}{M}},
\end{aligned}
\end{equation*}
we can ensure with probability at least $1-\tau$:
\begin{equation*}
 \mP'^M_\epsilon \subseteq \mP'_{\epsilon,M}(\eta_M). 
\end{equation*}
\end{theorem}
The argument mirrors the bounded-support case. The only additional ingredient is the technical generalization bound in Proposition~\ref{prop:hinge_SAA_guarantee_subweibull} in Appendix~\ref{appendix:proofs_sec4}, which controls the contribution of the tail event after truncation at \(t_M\). With this tail correction, the sampled ambiguity set retains the same conservative outer-approximation interpretation. Since \(R(t_M)=\mathcal O((\log M)^{1/\vartheta})\), the relaxation level \(\eta_M\) vanishes at the order \(\widetilde{\mathcal O}(M^{-1/2})\).

We next establish the convergence of the optimal value and optimal solution for the unbounded-support setting. Recall that $\hat v$ and $\mathcal S$ denote the optimal value and the set of minimizers of the exact problem with ambiguity set $\mP_\epsilon$.
\begin{theorem}\label{thm:unbounded_value_solution_convergence}
Fix $\epsilon>0$, and let $\hat v'_M$ and $\hat\x'_M$ be the optimal value and an optimal solution of the sampled problem \eqref{eq:dual_SAA_subweibull} with $t=t_M$. Under the assumptions of Theorem~\ref{thm:containment_unbounded_Xi}, as $M\to\infty$,
\begin{equation*}
\hat v'_M\overset{p}{\to}\hat v\quad\textup{and}\quad
\textup{dist}(\hat\x'_M,\mathcal S)\overset{p}{\to}0.
\end{equation*}
Furthermore, let $\bar \x_M$ be an exact optimizer of the threshold-$t_M$ problem, and assume there exists an optimal dual measure $\nu_M^\star$ in Proposition~\ref{thm:dual_exact_subweibull} such that $\gamma\ZZ-\nu_M^\star\in\M_+(\X)$ for some $\gamma\in\RR_+$. Define $\gamma^\star_M\coloneqq\inf\{\gamma\in\RR_+:\gamma\ZZ-\nu_M^\star\in\M_+(\X)\}$. Then, the following suboptimality bound holds with probability at least $1-\tau$:
\begin{equation*}
\begin{aligned}
\sup_{\QQ\in\mP'_{\epsilon,M}}  \EE_{\QQ}[\ell(\hat\x'_M,\txi)] &\leq \min_{\x\in\X} \sup_{\QQ\in\mP'_{\epsilon,M}}  \EE_{\QQ}[\ell(\x,\txi)]\\
&\qquad+ \gamma^\star_M\left(\mathcal O\left( R(t_M)\sqrt{\frac{J\log J(\log M)^3}{M}}\right)+ 2(c_1 + c_2\sqrt{\Omega}) \sqrt{\frac{1}{2M}\log\left(\frac{1}{\tau}\right)}\right.\\
&\qquad\qquad\left. + \frac{4c_1}{M} + 2c_2 \sqrt{ \frac{2\Omega}{M}}\right).
\end{aligned}
\end{equation*}
\end{theorem}

While the suboptimality bound above is measured against $\mP'_{\epsilon,M}$, the limiting convergence is formulated relative to the original ambiguity set $\mP_\epsilon$, as the moving tail constraint vanishes asymptotically. 
We defer a finite-sample suboptimality bound measured directly against $\mP_\epsilon$ to Theorem~\ref{thm:tuned_unbounded_original} in Appendix~\ref{appendix:proofs_sec4}. The appendix result shows that, when the corresponding dual sensitivity factors remain bounded, the error is $\widetilde{\mathcal O}(M^{-1/4})$. Thus, the threshold is used only to control tail behavior in the analysis and is sent to infinity as $M$ increases; it does not change the target DRO problem. Consequently, in implementations, we use the sampled ambiguity set from Section~\ref{sec:mc_sampling} even when the support is unbounded. Appendix~\ref{appendix:mc_numerical_illustration_student_t} provides a numerical illustration with a Student-$t_5$ heavy-tail distribution showing that the sampled-dual solution improves rapidly consistently with $M$ in different dimensions and does not suffer from the curse of dimensionality.

\subsection{Conic Programming Reformulations}
We now show that the approximate problem admits a tractable conic programming reformulation.
\begin{theorem}\label{theorem:Conic_reformulation_subweibull}
Suppose that $\X$ is convex.
The problem \eqref{eq:dual_SAA_subweibull} is equivalent to the following finite convex conic program:
\begin{equation} 
\begin{aligned}
\min \;&  \alpha  + \beta\Omega -\kappa \left(1-2\exp\left(-\left(t/K_{\mathrm{sw}}\right)^\vartheta\right)\right)+ \frac{1}{M}\sum_{m\in[M]}\nu_m(\EE_\PPhat[\ell(\z_m,\txi)]+ \epsilon)  \\
\st & \x\in\X, \; \alpha\in\RR,\;\beta,\kappa\in\RR_+,\;\bm\nu\in\RR_+^M\\
& \bm\lambda_{jm},\bm\lambda'_{jm}\in\RR_+^J\;\;\forall j\in[J]\;\forall m\in[M]\\
& \bm\theta_j\in\RR^{D_\bxi},\; \zeta_j,\zeta'_j\in\RR_+\;\;\forall j\in[J]\\
&\mathbf e^\top\bm\lambda_{jm}=\nu_m\;\;\forall m\in[M]\;\forall j\in[J]\\
&\mathbf e^\top\bm\lambda'_{jm}=\nu_m \;\;\forall m\in[M]\;\forall j\in[J]\\
&  \zeta_j+b_j(\x)-\frac{1}{M}\sum_{m\in[M],k\in[J]} \lambda_{jm}^k b_k(\z_m)\leq \alpha \quad \forall j\in[J]\\
& \left\|\begin{bmatrix}\bm a_j(\x)-\frac{1}{M}\sum_{m\in[M],k\in[J]} \lambda_{jm}^k \bm a_k(\z_m)\\\zeta_j-\beta\end{bmatrix}\right\|\leq\zeta_j+\beta\quad\forall j\in[J]\\
&  \zeta'_j+t \|\bm\theta_j\|+b_j(\x)+\kappa-\frac{1}{M}\sum_{m\in[M],k\in[J]} \lambda_{jm}^{'k} b_k(\z_m)\leq \alpha \quad \forall j\in[J]\\
& \left\|\begin{bmatrix}\bm a_j(\x)-\frac{1}{M}\sum_{m\in[M],k\in[J]} \lambda_{jm}^{'k} \bm a_k(\z_m)-\bm\theta_j\\\zeta'_j-\beta\end{bmatrix}\right\|\leq\zeta'_j+\beta\quad\forall j\in[J].
\end{aligned}
\end{equation}
\end{theorem}
In particular, if $\X$ is second-order cone representable, then the reformulation above is a second-order cone program.

\section{Numerical Experiments}\label{sec:experiments}

In this section, we conduct numerical experiments to validate the practical performance of our TIPM-DRO framework across two problem classes: a multi-item newsvendor problem and high-dimensional linear regression under adversarial outlier corruption. These experiments examine two complementary challenges: heavy-tailed operational uncertainty in inventory management, and high-dimensional geometry under high-leverage outlier contamination. In each experiment, we describe the problem setting, the competing methods, and the out-of-sample (OOS) evaluation results.
\subsection{Multi-Item Newsvendor Problem}

We consider a multi-item newsvendor problem with order vector $\x\in\R_+^K$ and random demand vector $\txi\in\R_+^K$. The loss is
\[
\ell(\x,\bxi)=\sum_{j=1}^K\left[h_j(x_j-\xi_j)_+ + b_j(\xi_j-x_j)_+\right].
\]
We optimize the $\PP$-CVaR$_\rho$ objective of this loss. The experiment pursues two objectives. First, we fix $K=3$ and compare four coordinatewise demand families: truncated Gaussian, rescaled $\chi^2$, Lognormal, and Pareto, to evaluate performance across lighter- and heavier-tailed demand regimes, since heavy tails can materially change stocking decisions~\citep{bimpikis2016inventory}. Second, we fix Lognormal demand and repeat the experiment for $K\in\{1,3,5\}$ products to assess whether TIPM-DRO remains effective as the number of products grows.


We compare TIPM-DRO against the strongest relevant benchmarks under the same CVaR objective. SAA provides the standard non-robust benchmark. For transport-based robustness, we use 2-Wasserstein DRO (2-WDRO), which has recently been shown to deliver strong out-of-sample performance in comparative studies~\citep{byeon2025comparative} and serves as our state-of-the-art Wasserstein baseline. We do not separately report 1-WDRO because, for newsvendor losses of this form, it shares the same optimal minimizers as SAA~\citep[Remark~6.7]{mohajerin2018data}. In each trial, the in-sample data are split into training and validation subsets; the ambiguity radius of each robust method is selected from the same grid by validation. To ensure a fair comparison, every method is evaluated on the same OOS cost using an independent test sample. Additional implementation details are reported in Appendix~\ref{appendix:newsvendor}.

\begin{figure}[t]
\centering
\begin{subfigure}[t]{0.48\textwidth}
    \centering
    \includegraphics[width=\linewidth]{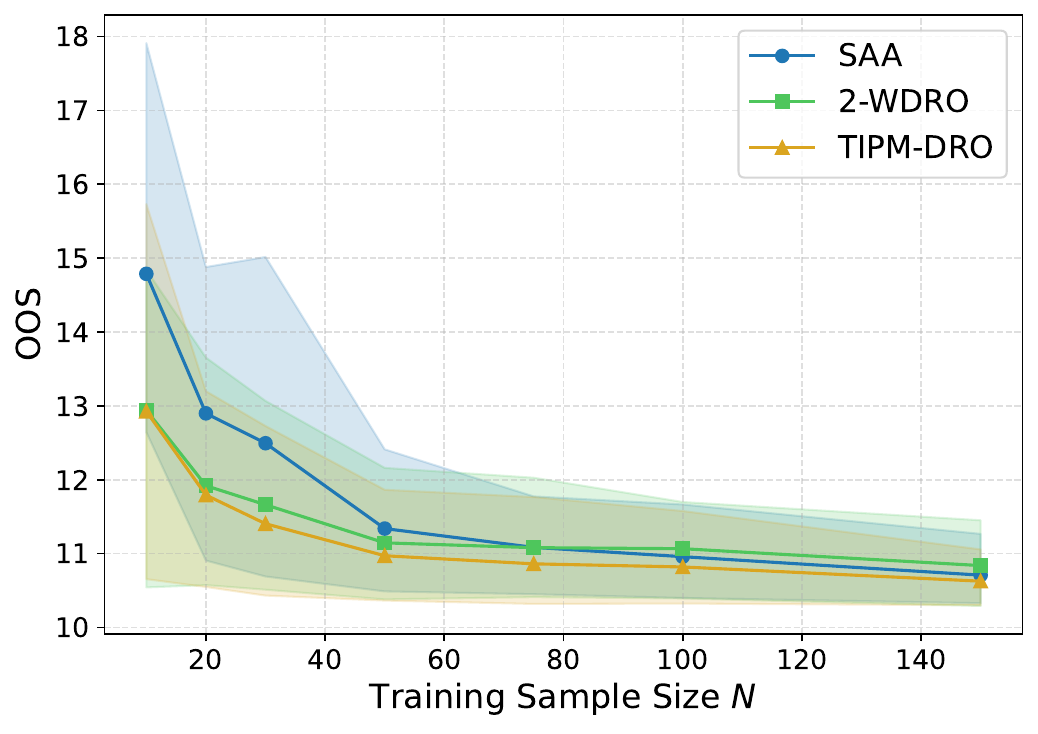}
    \caption{Truncated Gaussian}
\end{subfigure}
\hfill
\begin{subfigure}[t]{0.48\textwidth}
    \centering
    \includegraphics[width=\linewidth]{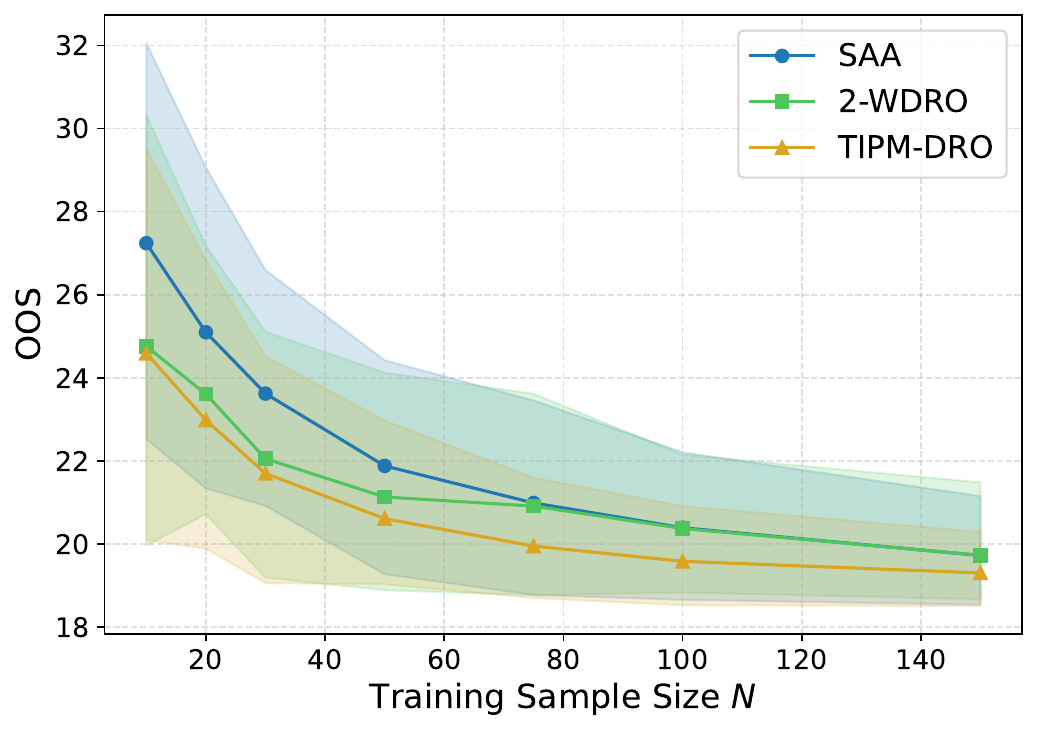}
    \caption{Rescaled $\chi^2$}
\end{subfigure}

\vspace{0.5em}

\begin{subfigure}[t]{0.48\textwidth}
    \centering
    \includegraphics[width=\linewidth]{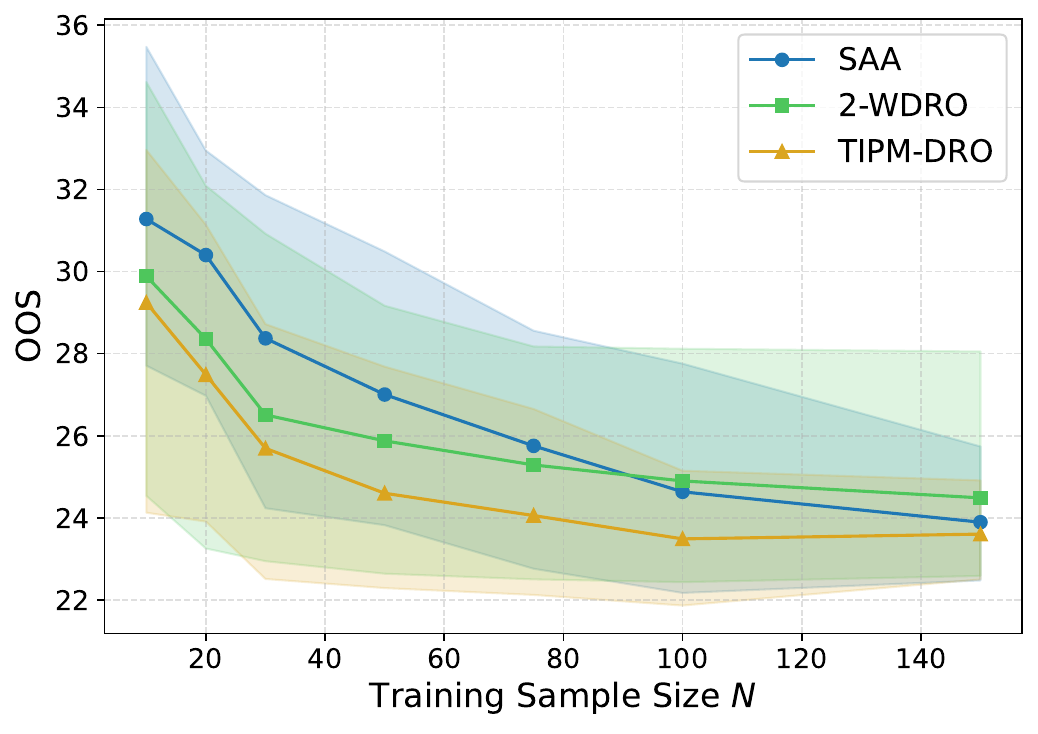}
    \caption{Lognormal}
\end{subfigure}
\hfill
\begin{subfigure}[t]{0.48\textwidth}
    \centering
    \includegraphics[width=\linewidth]{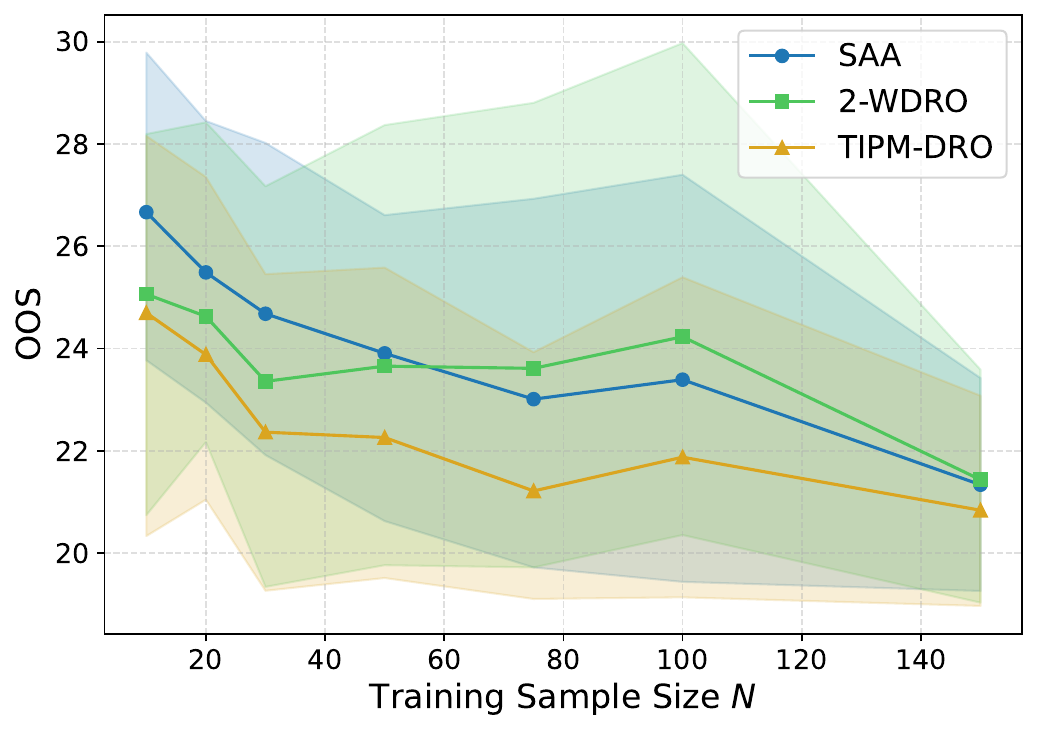}
    \caption{Pareto}
\end{subfigure}

\caption{Mean OOS cost versus the training sample size $N$ with 10th--90th percentile bands in the multi-item newsvendor experiment with $K=3$.}
\label{fig:nv_oos_by_n}
\end{figure}
Figure~\ref{fig:nv_oos_by_n} reports mean OOS cost with 10th--90th percentile bands over 50 trials, where lower values are better. TIPM-DRO achieves the best mean OOS performance in all four demand settings. Under truncated Gaussian demand, both robust methods improve on SAA and TIPM-DRO still performs slightly better than 2-WDRO, although the gap narrows as $N$ increases. Under rescaled $\chi^2$, Lognormal, and Pareto demand, the advantage of TIPM-DRO is more persistent, with the clearest separation appearing in the heavier-tailed settings. The truncated Gaussian panel shows that TIPM-DRO is already competitive in a lighter-tailed regime, while the heavier-tailed panels show that its advantage over both SAA and 2-WDRO becomes more pronounced.

\begin{figure}[t]
\centering
\begin{subfigure}[t]{0.328\textwidth}
    \centering
    \includegraphics[width=\linewidth]{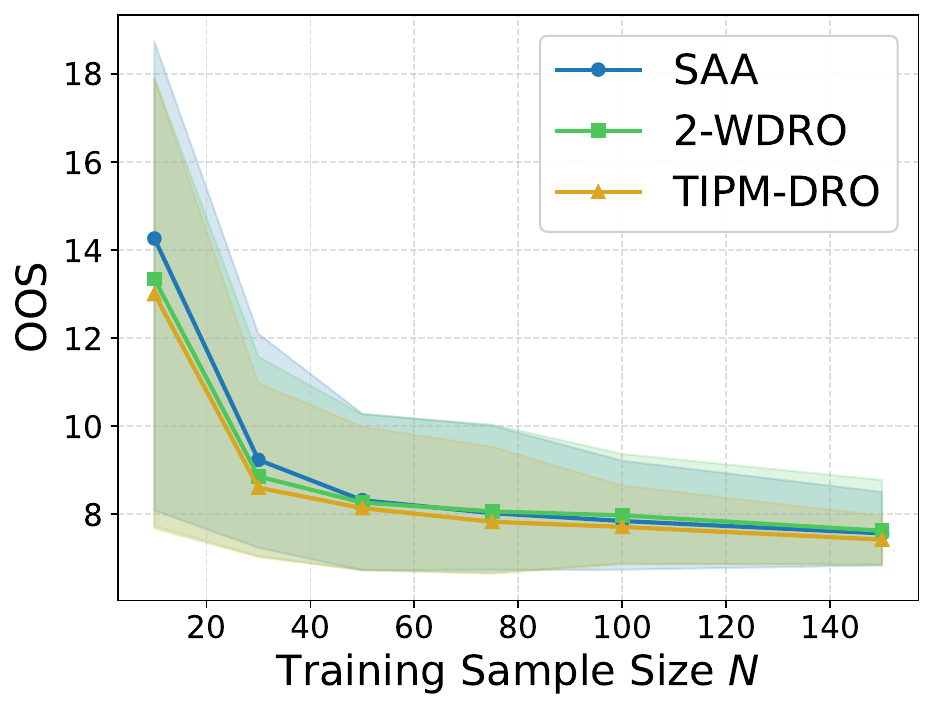}
    \caption{$K=1$}
\end{subfigure}
\hfill
\begin{subfigure}[t]{0.328\textwidth}
    \centering
    \includegraphics[width=\linewidth]{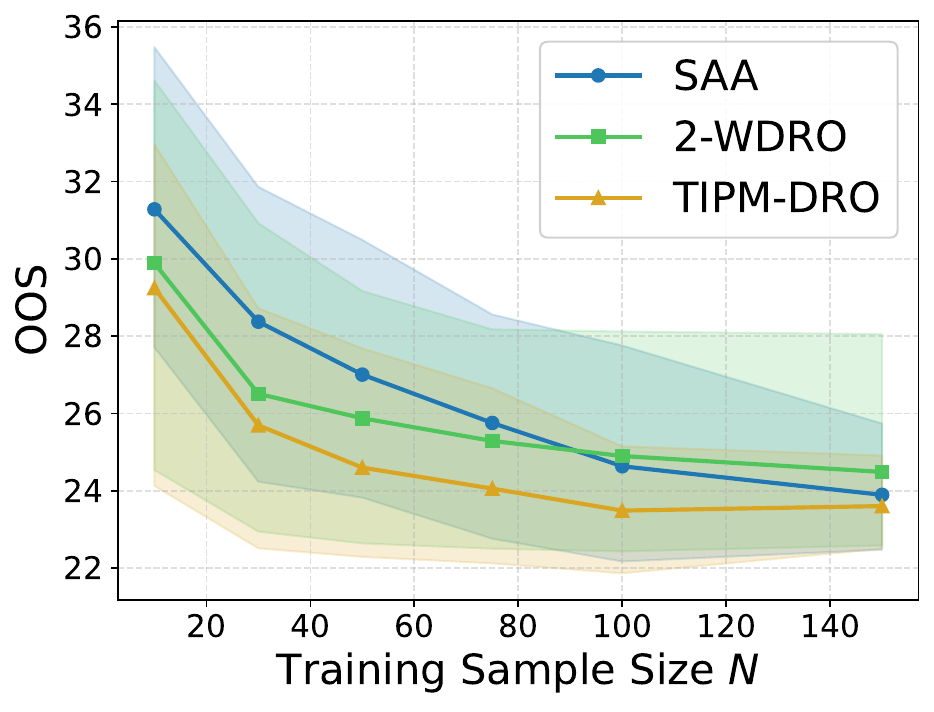}
    \caption{$K=3$}
\end{subfigure}
\hfill
\begin{subfigure}[t]{0.328\textwidth}
    \centering
    \includegraphics[width=\linewidth]{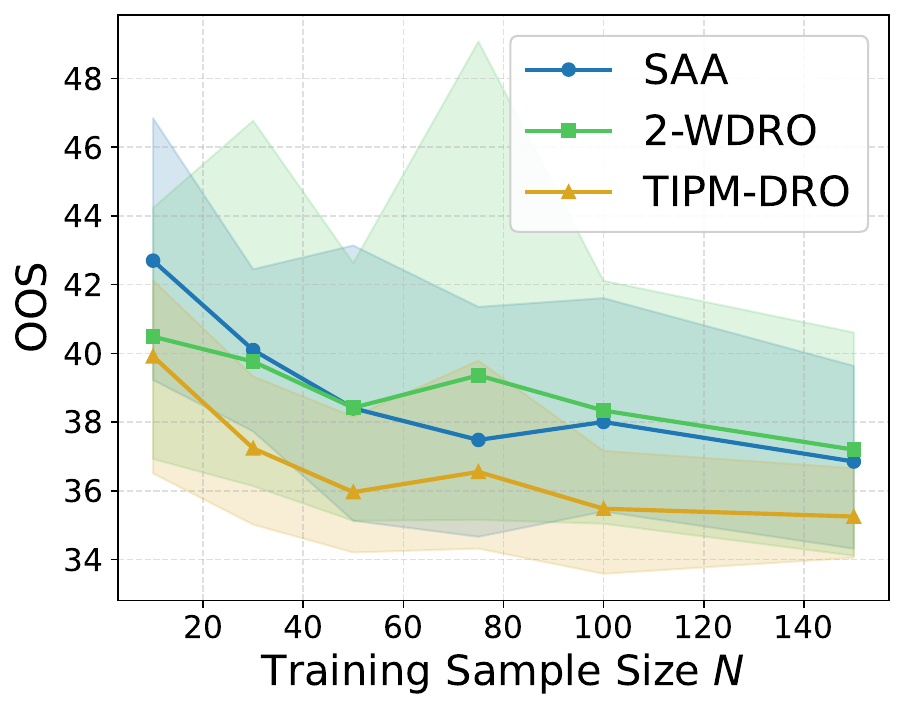}
    \caption{$K=5$}
\end{subfigure}
\caption{Mean OOS cost versus training sample size $N$ with 10th--90th percentile bands under Lognormal demand for $K=1$, $3$, and $5$ products, based on 50 independent trials.}
\label{fig:nv_K_sweep}
\end{figure}
Figure~\ref{fig:nv_K_sweep} reports the Lognormal results across $K\in\{1,3,5\}$ products over 50 independent trials. At $K=1$, all three methods perform similarly, with TIPM-DRO offering a modest improvement at small $N$. As $K$ increases to $3$ and $5$, the OOS advantage of TIPM-DRO over SAA and 2-WDRO grows steadily. At $K=5$, TIPM-DRO maintains a consistent gap below both baselines across all sample sizes, while 2-WDRO exhibits substantially wider variance bands, consistent with the difficulty of calibrating a high-dimensional Wasserstein ball under heavy-tailed demand. These results confirm that the advantage of TIPM-DRO scales with the number of products, in line with the dimension-free guarantees in Section~\ref{sec:IPMDRO}.

\subsection{Outlier-Corrupted Regression}\label{sec:exp_outlier}
\begin{figure}[t]
\centering
\begin{subfigure}[t]{0.48\textwidth}
    \centering
    \includegraphics[width=\linewidth]{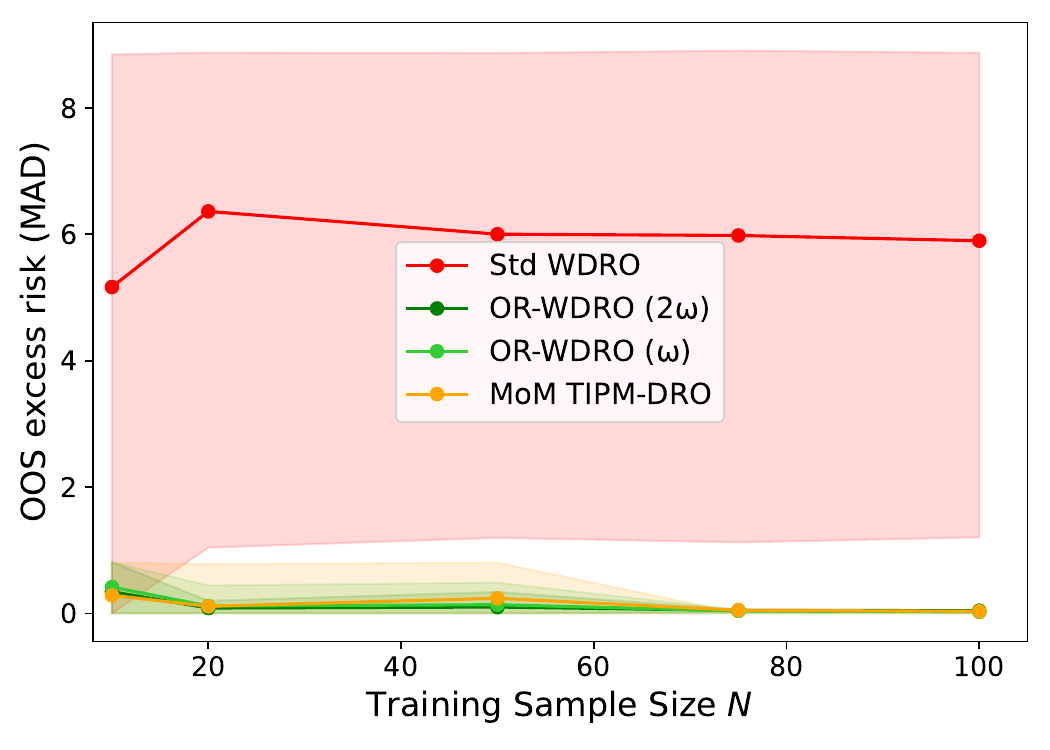}
    \caption{$D_\bxi=1$}
\end{subfigure}
\hfill
\begin{subfigure}[t]{0.48\textwidth}
    \centering
    \includegraphics[width=\linewidth]{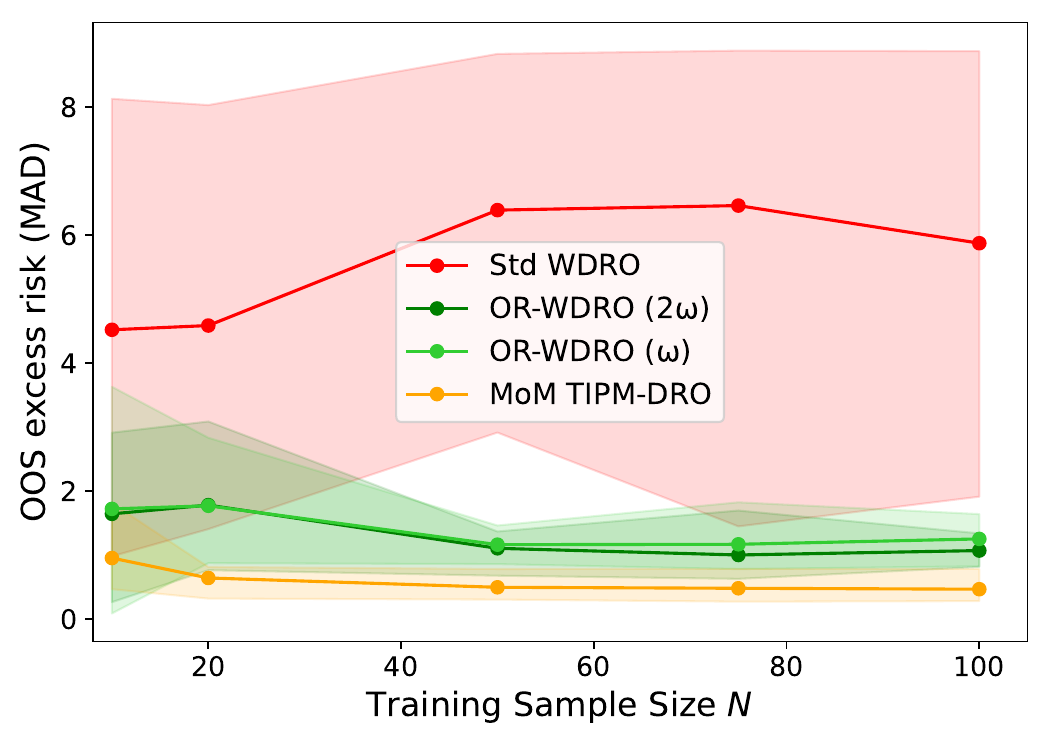}
    \caption{$D_\bxi=5$}
\end{subfigure}

\vspace{0.5em}

\begin{subfigure}[t]{0.48\textwidth}
    \centering
    \includegraphics[width=\linewidth]{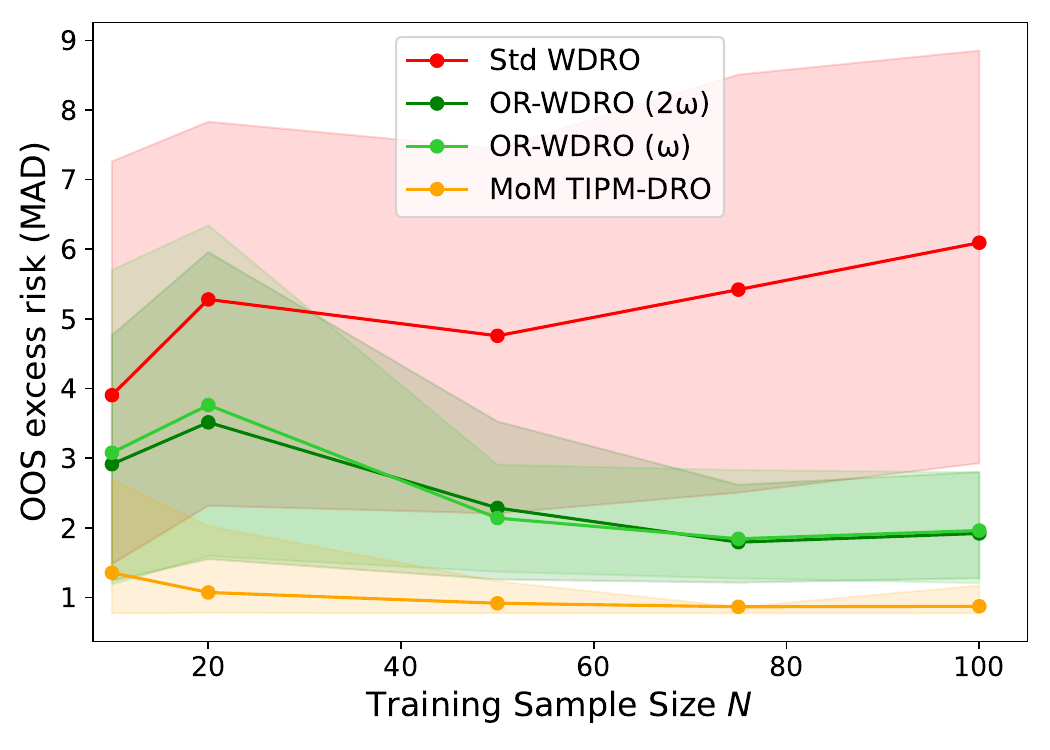}
    \caption{$D_\bxi=10$}
\end{subfigure}
\hfill
\begin{subfigure}[t]{0.48\textwidth}
    \centering
    \includegraphics[width=\linewidth]{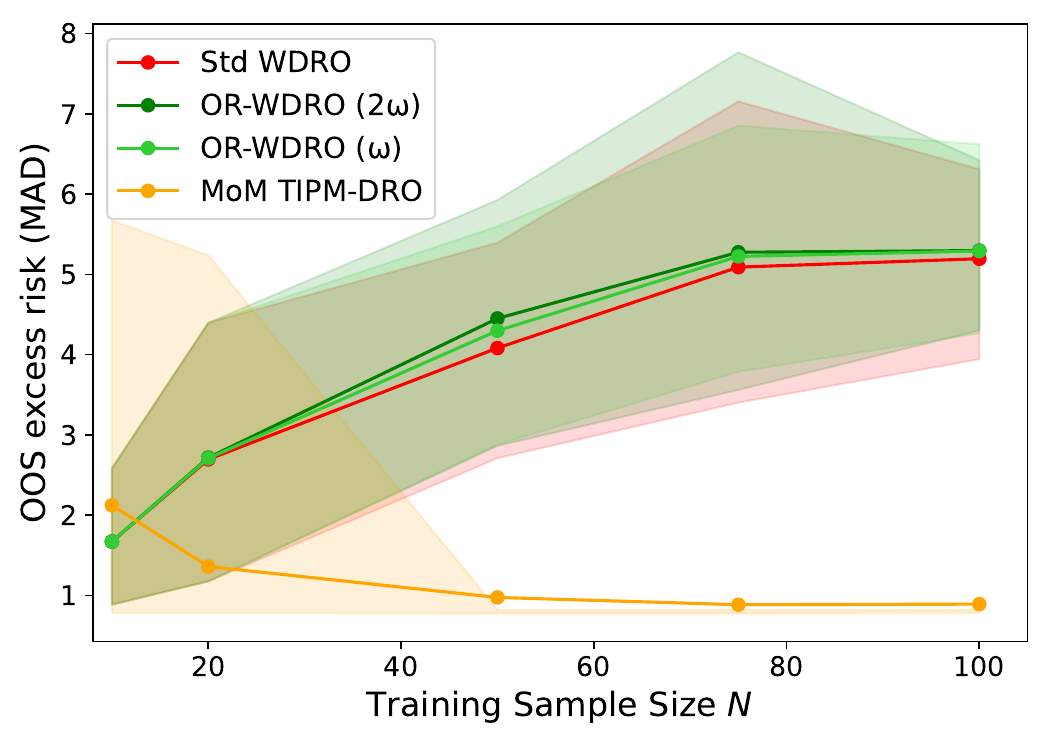}
    \caption{$D_\bxi=50$}
\end{subfigure}

\caption{OOS excess risk under the MAD loss versus the training sample size $N$, with panels ordered by increasing feature dimension $D_\bxi$.}
\label{fig:dim_sweep}
\end{figure}

We consider linear regression with the mean absolute deviation (MAD) loss $\ell(\x, \bxi) = |\x^\top \bm{u} - v|$, where $\bxi = (\bm{u}, v) \in \RR^{D_{\bxi}+1}$ denotes a data point with feature vector $\bm{u} \in \RR^{D_{\bxi}}$ and response $v \in \RR$, following the setup of~\citet{nietert2023outlier}. The true regression coefficient $\x^\star$ is drawn uniformly from the unit sphere in $\RR^{D_{\bxi}}$. Clean features satisfy $\bm{u}_i \sim \mathcal{N}(\bm{0}, \I_{D_{\bxi}})$, and clean responses are $v_i = \x^{\star\top}\bm{u}_i$. An adversary corrupts a fraction $\omega = 0.2$ of the training samples by replacing each corrupted pair $(\bm{u}_i,v_i)$ with $(C\bm{u}_i,-C^2v_i+\rho')$, where $\rho' = 0.1$. This creates high-leverage outliers whose severity is controlled by the scale parameter $C$. The experiment is designed to probe two questions. First, what happens as the ambient dimension $D_\bxi$ grows and Wasserstein geometry becomes less informative? Second, does robust performance depend on the outliers being geometrically well separated from the clean sample?


We benchmark TIPM-DRO against state-of-the-art transport-based competitors tailored to this outlier setting. Standard Wasserstein DRO~\citep{mohajerin2018data} (Std WDRO) provides the canonical baseline centered at the empirical distribution. We also compare against the outlier-robust Wasserstein DRO (OR-WDRO) methods of~\citet{nietert2023outlier}, which are designed for the same adversarial contamination model and represent the strongest direct transport-based alternative; we report both OR-WDRO$(\omega)$ and OR-WDRO$(2\omega)$ so that the baseline is not disadvantaged by a single conservative contamination radius. All robust models are tuned by a trimmed validation loss on a held-out split. For the three Wasserstein-based models, we cross-validate the Wasserstein radius over a common grid. For MoM TIPM-DRO, we cross-validate the number of blocks $K'$ and the TIPM radius $\epsilon$.  To ensure a fair comparison, all methods are evaluated on the same MAD-based OOS excess risk under the same test dataset. Additional implementation details are reported in Appendix~\ref{appendix:outlier_regression}.

Figure~\ref{fig:dim_sweep} reports the OOS excess risk under the MAD loss as a function of the training sample size $N$
for four representative dimensions $D_\bxi$, based on 50 trials and a clean test set of size $1000$. Standard WDRO barely improves with $N$, because its ambiguity set remains centered at the contaminated empirical distribution and therefore continues to fit the outliers as more corrupted data are observed. OR-WDRO works reasonably well in low dimension, where transport costs can still separate inliers from outliers, but this advantage fades as $D_\bxi$ grows and Wasserstein distances concentrate. MoM TIPM-DRO, in contrast, stays stable across dimensions and becomes clearly best at $D_\bxi=50$. This suggests that transport-based outlier handling depends on informative geometry, while MoM TIPM-DRO remains stable across dimensions, in line with our theory, which requires only $\omega<1/2$ and does not worsen with ambient dimension.

\begin{figure}[t]
  \centering
  \includegraphics[width=0.5\textwidth]{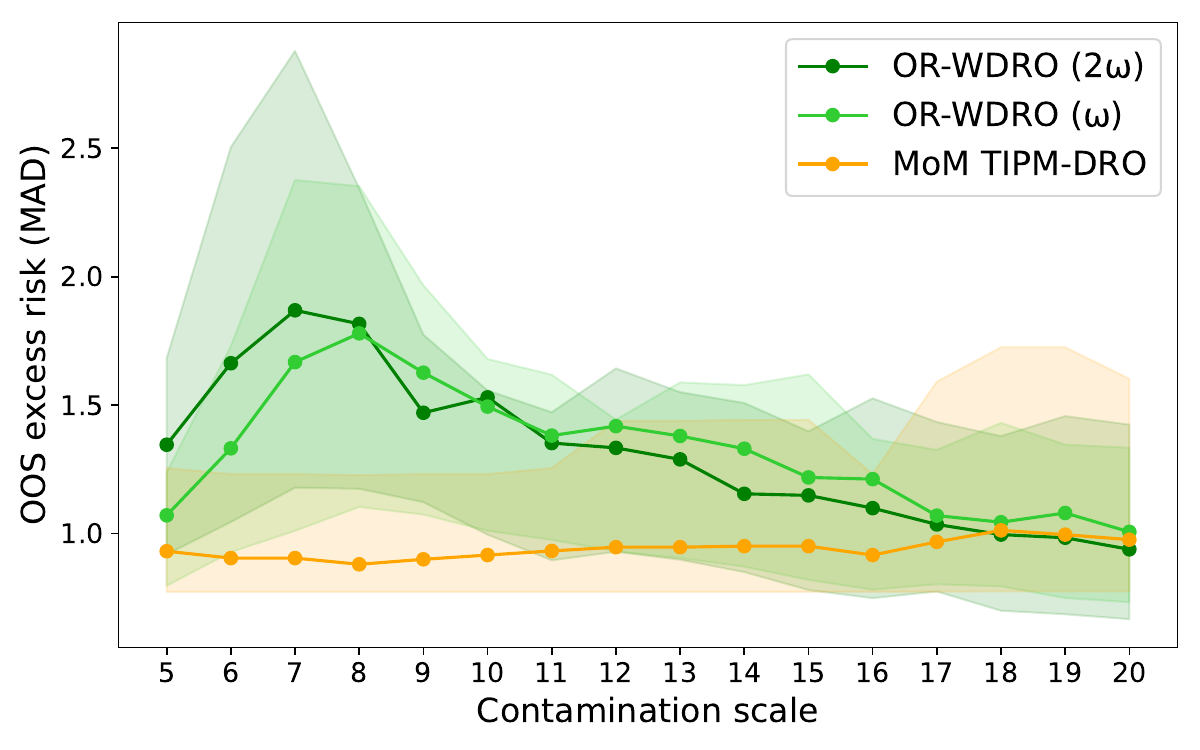}
  \caption{OOS excess risk under the MAD loss versus the contamination scale $C$, with 10th--90th percentile bands over the trials.}
  \label{fig:cont_scale}
\end{figure}

To isolate the effect of outlier severity, Figure~\ref{fig:cont_scale} fixes $N=50$ and $D_\bxi=10$ and sweeps $C$ from $5$ to $20$. For visual clarity, this figure displays only OR-WDRO$(2\omega)$, OR-WDRO$(\omega)$, and MoM TIPM-DRO. OR-WDRO shows a clear hump at moderate contamination scales: the outliers are already harmful enough to distort the empirical reference, but not yet far enough away in transport cost to be trimmed reliably. Only when $C$ becomes very large does OR-WDRO start to recover. MoM TIPM-DRO stays nearly flat throughout, showing that its performance is much less sensitive to how severe the leverage contamination is. Taken together, the two figures show that MoM TIPM-DRO does not just improve over standard WDRO. It is also more stable than outlier-aware WDRO baselines exactly when transport geometry is least reliable: at higher ambient dimension and at moderate contamination scales.

\section{Concluding Remarks}

We introduced a distributionally robust optimization framework based on targeted integral probability metrics. By defining distributional discrepancy directly through the decision-induced loss class, the proposed expected hinge formulation yields a task-aware ambiguity set that is equivalent to a semi-infinite family of loss-discrepancy constraints. This construction enables finite-sample guarantees that bypass the ambient curse of dimensionality whenever a scalar pointwise concentration inequality is available, thereby extending dimension-free rate guarantees for DRO to a broad range of nonstandard data regimes without requiring geometry-driven ambiguity-set design. The framework is modular: once a suitable pointwise concentration inequality is established for a new data process or reference estimator, the same TIPM ambiguity-set structure and calibration principle apply.

Our numerical results highlight the practical value of aligning distributional robustness with the downstream loss. In both inventory management and outlier-corrupted regression, the proposed framework delivers favorable out-of-sample performance relative to transport-based baselines. The experiments also suggest that the method may remain effective for Pareto-type heavy-tailed regimes beyond the current sub-Weibull theory, motivating further investigation. Several directions remain for future research, including extending the theoretical guarantees to Pareto-type heavy-tailed models, developing adaptive sampling schemes to improve the Monte Carlo approximation rate, and designing scalable algorithms for richer multistage decision problems.

\bibliographystyle{informs2014}
\bibliography{bibliography}

\clearpage
\begin{APPENDICES}
\renewcommand{\theHsection}{appendix.\Alph{section}}
\renewcommand{\theHsubsection}{appendix.\Alph{section}.\arabic{subsection}}
\linespread{1.34}\selectfont
\setlength{\parskip}{1em}

\section{Proofs of Section~\ref{sec:IPMDRO}}
\label{appendix:proofs_sec2}

\subsection{Ambiguity Set Characterization}

\begin{proof}{Proof of Proposition~\ref{prop:semiinf_ambiguity_set}.}
It suffices to show that the semi-infinite constraint
\begin{equation}
\label{eq:semiinf_constraint}
\EE_\QQ[\ell(\z,\txi)] - \EE_\PPhat[\ell(\z,\txi)]\leq \epsilon \quad\forall\z\in\X
\end{equation}
is equivalent to the expectation constraint 
\begin{equation}
\label{eq:hinge_constraint}
\EE_\ZZ\left[\left[\EE_\QQ[\ell(\tz,\txi)]-\EE_\PPhat[\ell(\tz,\txi)]-\eps\right]_+\right] \leq 0.
\end{equation}

If $\QQ$ is feasible to \eqref{eq:semiinf_constraint}, then $[\EE_\QQ[\ell(\z,\txi)] - \EE_\PPhat[\ell(\z,\txi)]- \epsilon]_+\leq 0$ for all $\z\in\X$. Taking expectation with respect to $\ZZ$ yields \eqref{eq:hinge_constraint}.

Conversely, suppose $\QQ$ is infeasible to \eqref{eq:semiinf_constraint}; that is, there exists $\z'\in\X$ such that $[\EE_\QQ[\ell(\z',\txi)] - \EE_\PPhat[\ell(\z',\txi)]- \epsilon]_+> 0$. By Assumption~\ref{as:linear_growth}, for any $\QQ$ feasible to \eqref{eq:ambiguity_set},
\begin{equation*}
\EE_\QQ[\ell(\z,\txi)]\leq c_1 + c_2\EE_\QQ[\|\txi\|]\leq c_1 + c_2\sqrt{\EE_\QQ[\|\txi\|^2]}\leq c_1 + c_2\sqrt{\Omega}<+\infty\qquad\forall\z\in\X,
\end{equation*}
where the second inequality follows from Jensen's inequality. By Assumption~\ref{as:linear_growth}, $|\ell(\z,\bxi)|\leq c_1+c_2\|\bxi\|$ provides a dominator integrable under $\QQ$ (via the second-moment constraint in~\eqref{eq:ambiguity_set}), so the dominated convergence theorem implies that $\EE_\QQ[\ell(\z,\txi)]$ is continuous in $\z$. Since $\EE_\PPhat[\ell(\z,\txi)]=\frac{1}{N}\sum_{i=1}^N\ell(\z,\hat{\bxi}_i)$ is a finite sum of functions continuous in $\z$, it is likewise continuous. Consequently, $[\EE_\QQ[\ell(\z,\txi)] - \EE_\PPhat[\ell(\z,\txi)]- \epsilon]_+$ is continuous in $\z$. Hence there exists a neighborhood $\mathcal Z_\tau=\{\z\in\X:\|\z-\z'\|<\tau\}$ in which $[\EE_\QQ[\ell(\z,\txi)] - \EE_\PPhat[\ell(\z,\txi)]- \epsilon]_+> 0$ for all $\z\in\mathcal Z_\tau$. We therefore have
\begin{equation*}
\begin{aligned}
\EE_\ZZ\left[\left[\EE_\QQ[\ell(\tz,\txi)]-\EE_\PPhat[\ell(\tz,\txi)]-\eps\right]_+\right]
\geq \ZZ(\z\in\mathcal Z_\tau)\EE_\ZZ\left[\left[\EE_\QQ[\ell(\tz,\txi)]-\EE_\PPhat[\ell(\tz,\txi)]-\eps\right]_+\Bigm|\z\in\mathcal Z_\tau\right]>0,
\end{aligned}
\end{equation*}
because $\ZZ$ has full support on $\X$. Hence $\QQ$ is infeasible to \eqref{eq:hinge_constraint}. Replacing the semi-infinite constraint in \eqref{eq:ambiguity_set} with the equivalent expectation constraint \eqref{eq:hinge_constraint} yields the claim.\qed
\end{proof}

\subsection{Performance Guarantees}

\begin{lemma}\label{lem:Lipschitz_condition}
For any $\QQ$ such that $\EE_\QQ[\|\bxi\|^2]\leq\Omega$, the following Lipschitz condition holds:
\begin{align*}
\EE_\QQ\left[|\ell(\z,\txi)-\ell(\z',\txi)|\right]
\leq L\|\z-\z'\|\quad\forall\z,\z'\in\X,
\end{align*}
where $L\coloneqq \max_{j\in[J]}\|\bm A_j\|_{\textup{op}}\sqrt{\Omega}+\max_{j\in[J]} \|\bm b_j\|$.
\end{lemma}
\begin{proof}{Proof of Lemma~\ref{lem:Lipschitz_condition}.}
For any $\z,\z'\in\X$, since $\bm a_j(\z)-\bm a_j(\z')=\bm A_j(\z-\z')$ and $b_j(\z)-b_j(\z')=\bm b_j^\top(\z-\z')$, we have
\begin{align*}
|\ell(\z,\bxi)-\ell(\z',\bxi)|
&\leq \max_{j\in[J]}\bigl|(\bm A_j(\z-\z'))^\top\bxi + \bm b_j^\top(\z-\z')\bigr| \\
&\leq \max_{j\in[J]}\bigl(\|\bm A_j\|_{\textup{op}}\|\bxi\| + \|\bm b_j\|\bigr)\cdot\|\z-\z'\|\\
&\leq \left(\max_{j\in[J]}\|\bm A_j\|_{\textup{op}}\|\bxi\|+\max_{j\in[J]} \|\bm b_j\| \right)\cdot\|\z-\z'\|.
\end{align*}
Taking expectation under $\QQ$ yields
\begin{align*}
\EE_\QQ\left[|\ell(\z,\txi)-\ell(\z',\txi)|\right]
&\leq \left( \max_{j\in[J]}\|\bm A_j\|_{\textup{op}}\EE_\QQ[\|\txi\|]+\max_{j\in[J]} \|\bm b_j\| \right)\|\z-\z'\|\\
&\leq \left( \max_{j\in[J]}\|\bm A_j\|_{\textup{op}}\sqrt{\Omega}+\max_{j\in[J]} \|\bm b_j\| \right)\|\z-\z'\|.
\end{align*}
This is the desired bound.\qed
\end{proof}

\begin{proof}{Proof of Proposition~\ref{prop:distribution_coverage}.}
Let $\eta_N\coloneqq \frac{1}{L\sqrt{N}}$,
and let $\X_{\eta_N}$ be an $\eta_N$-net of $\X$ under $\|\cdot\|$. Since $\X\subseteq \mathbb B(0,R_\x)$ is compact, the standard volumetric bound gives
\[
|\X_{\eta_N}|\leq \left(1+\frac{2R_\x}{\eta_N}\right)^{D_\x}
=\left(1+2R_\x L\sqrt{N}\right)^{D_\x}.
\]
For each $\z_0\in\X_{\eta_N}$, applying \eqref{eq:pointwise_conc} with confidence level $\delta' \coloneqq \delta/|\X_{\eta_N}|$ 
and then taking a union bound over $\X_{\eta_N}$, we obtain with probability at least $1-\delta$,
\[
\EE_{\PP^\star}[\ell(\z_0,\txi)]-\EE_\PPhat[\ell(\z_0,\txi)]
\leq \frac{1}{\sqrt N}\varepsilon\!\left(\log\frac{1}{\delta'}\right)
\qquad\forall \z_0\in\X_{\eta_N}.
\]
Because $\varepsilon$ is nondecreasing and
\[
\log\frac{1}{\delta'}
=\log|\X_{\eta_N}|+\log\frac{1}{\delta}
\leq D_\x\log\left(1+2R_\x L\sqrt N\right)+\log\frac{1}{\delta},
\]
the same event implies
\begin{equation}
\label{eq:net_bound_appendix}
\EE_{\PP^\star}[\ell(\z_0,\txi)]-\EE_\PPhat[\ell(\z_0,\txi)]
\leq
\frac{1}{\sqrt N}\varepsilon\!\left(D_\x\log\left(1+2R_\x L\sqrt N\right)+\log\frac{1}{\delta}\right)
\qquad\forall \z_0\in\X_{\eta_N}.
\end{equation}

Fix any $\z\in\X$ and choose $\z_0\in\X_{\eta_N}$ with $\|\z-\z_0\|\leq \eta_N$. By Lemma~\ref{lem:Lipschitz_condition}, applied once under $\PP^\star$ and once under $\PPhat$, we have
\begin{equation*}
\begin{aligned}
\EE_{\PP^\star}[\ell(\z,\txi)]-\EE_\PPhat[\ell(\z,\txi)]
&\leq \EE_{\PP^\star}[\ell(\z_0,\txi)]-\EE_\PPhat[\ell(\z_0,\txi)]\\
&\qquad+\EE_{\PP^\star}[|\ell(\z,\txi)-\ell(\z_0,\txi)|]+\EE_\PPhat[|\ell(\z,\txi)-\ell(\z_0,\txi)|]\\
&\leq \EE_{\PP^\star}[\ell(\z_0,\txi)]-\EE_\PPhat[\ell(\z_0,\txi)] + 2L\eta_N.
\end{aligned}
\end{equation*}
Combining this bound with \eqref{eq:net_bound_appendix} and using $2L\eta_N=2/\sqrt N$ yields
\[
\EE_{\PP^\star}[\ell(\z,\txi)]-\EE_\PPhat[\ell(\z,\txi)]
\leq \epsilon
\qquad\forall \z\in\X
\]
with probability at least $1-\delta$, where $\epsilon$ is defined in \eqref{eq:radius}. Proposition~\ref{prop:semiinf_ambiguity_set} then implies $\PP^\star\in\mP_\epsilon$.\qed
\end{proof}

\begin{proof}{Proof of Corollary~\ref{cor:OOS_guarantee}.}
On the event $\{\PP^\star\in\mP_\epsilon\}$, which occurs with probability at least $1-\delta$ by Proposition~\ref{prop:distribution_coverage}, we have
\[
\EE_{\PP^\star}[\ell(\hat\x,\txi)]
\leq \sup_{\QQ\in\mP_\epsilon}\EE_\QQ[\ell(\hat\x,\txi)]
= \hat J.
\]
This is precisely the claimed out-of-sample guarantee.\qed
\end{proof}

\begin{proof}{Proof of Theorem~\ref{thm:excess_risk}.}
Let
\[
\mathcal E_1\coloneqq\{\PP^\star\in\mP_\epsilon\}
\qquad\text{and}\qquad
\mathcal E_2\coloneqq\left\{\left|\EE_{\PP^\star}[\ell(\x^\star,\txi)]-\EE_\PPhat[\ell(\x^\star,\txi)]\right|\leq \frac{1}{\sqrt N}\varepsilon\!\left(\log\frac{1}{\delta}\right)\right\}.
\]
By Proposition~\ref{prop:distribution_coverage}, $\Prob(\mathcal E_1)\ge 1-\delta$, and by \eqref{eq:twosided_pointwise_conc} applied at the fixed decision $\x^\star$, $\Prob(\mathcal E_2)\ge 1-\delta$. Hence $\Prob(\mathcal E_1\cap \mathcal E_2)\ge 1-2\delta$. On $\mathcal E_1\cap \mathcal E_2$, we have
\begin{equation*}
\begin{aligned}
\EE_{\PP^\star}[\ell(\hat\x,\txi)]
&\leq \sup_{\QQ\in\mP_\epsilon}\EE_\QQ[\ell(\hat\x,\txi)]\\
&\leq \sup_{\QQ\in\mP_\epsilon}\EE_\QQ[\ell(\x^\star,\txi)]\\
&\leq \EE_\PPhat[\ell(\x^\star,\txi)]+\epsilon\\
&\leq \EE_{\PP^\star}[\ell(\x^\star,\txi)] + \frac{1}{\sqrt N}\varepsilon\!\left(\log\frac{1}{\delta}\right)+\epsilon.
\end{aligned}
\end{equation*}
This completes the proof.\qed
\end{proof}

\subsection{Applications}

\begin{proof}{Proof of Proposition~\ref{prop:subweibull_coverage}.}
Fix any $\z\in\X$ and define the centered random variables
\[
X_i\coloneqq \EE_{\PP^\star}[\ell(\z,\txi)]-\ell(\z,\hat{\bxi}_i),\qquad i\in[N].
\]
Then $X_1,\ldots,X_N$ are independent, mean-zero, and satisfy $\|X_i\|_{\psi_\vartheta}\le K$. Applying \citet[Theorem~3.1]{kuchibhotla2022moving} with weights $a_i=1/N$, we obtain
\[
\Prob\!\left(\left|\EE_{\PP^\star}[\ell(\z,\txi)]-\EE_\PPhat[\ell(\z,\txi)]\right|\geq 2eC(\vartheta)\|b\|\sqrt{t}+2eL_N^*(\vartheta)\|b\|_{\beta(\vartheta)}t^{1/\vartheta}\right)\leq 2e^{-t}\qquad\forall t\geq 0,
\]
where
\[
b\coloneqq \left(\frac{\|X_1\|_{\psi_\vartheta}}{N},\ldots,\frac{\|X_N\|_{\psi_\vartheta}}{N}\right),
\qquad
\beta(\vartheta)\coloneqq
\begin{cases}
\infty,& \vartheta\leq 1,\\
\frac{\vartheta}{\vartheta-1},& \vartheta>1.
\end{cases}
\]
Since $\|X_i\|_{\psi_\vartheta}\le K$, we have $\|b\|\leq K/\sqrt{N}$.
Moreover,
\[
\|b\|_{\beta(\vartheta)}\leq
\begin{cases}
\|b\|_\infty\le \frac{K}{N},& \vartheta\leq 1,\\[2mm]
\left(N\left(\frac{K}{N}\right)^{\beta(\vartheta)}\right)^{1/\beta(\vartheta)}=K\,N^{-1/\vartheta},& \vartheta>1,
\end{cases}
\]
and therefore $\|b\|_{\beta(\vartheta)}\leq K/N^{\min\{1,1/\vartheta\}}$.
Finally, Theorem~3.1 gives
\[
L_N^*(\vartheta)=L_N(\vartheta)\frac{C(\vartheta)\|b\|}{\|b\|_{\beta(\vartheta)}}=
\begin{cases}
\frac{4^{1/\vartheta}}{\sqrt 2}C(\vartheta),& \vartheta<1,\\[2mm]
\frac{4^{1/\vartheta}}{\sqrt 2}\,4e,& \vartheta\geq 1,
\end{cases}
\]
which depends only on $\vartheta$. Hence there exist constants $c_{1,\vartheta},c_{2,\vartheta}>0$, depending only on $\vartheta$, such that
\[
\Prob\!\left(\EE_{\PP^\star}[\ell(\z,\txi)]-\EE_\PPhat[\ell(\z,\txi)]\geq c_{1,\vartheta}K\sqrt{\frac{t}{N}}+c_{2,\vartheta}K\frac{t^{1/\vartheta}}{N^{\min\{1,1/\vartheta\}}}\right)\leq 2e^{-t}\qquad\forall t\geq 0.
\]
We now absorb the heavy-tail correction term into the leading root-$N$ term.

If $\vartheta=2$, both terms are already of the same order, so there exists $C_\vartheta>0$ such that
\[
c_{1,\vartheta}K\sqrt{\frac{t}{N}}+c_{2,\vartheta}K\frac{t^{1/\vartheta}}{N^{1/\vartheta}}
\leq
C_\vartheta K\sqrt{\frac{t}{N}}
\qquad \forall t\geq 0.
\]

If $\vartheta\in[1,2)$, then $1/\vartheta-1/2>0$. Therefore, there exists $c_\vartheta>0$, depending only on $\vartheta$, such that whenever $t\le c_\vartheta N$,
\[
c_{2,\vartheta}K\frac{t^{1/\vartheta}}{N^{1/\vartheta}}
\leq
c_{1,\vartheta}K\sqrt{\frac{t}{N}}.
\]
Hence, for all $t\le c_\vartheta N$,
\[
c_{1,\vartheta}K\sqrt{\frac{t}{N}}+c_{2,\vartheta}K\frac{t^{1/\vartheta}}{N^{1/\vartheta}}
\leq
2c_{1,\vartheta}K\sqrt{\frac{t}{N}}.
\]

If $\vartheta\in(0,1)$, then there exists $c_\vartheta>0$, depending only on $\vartheta$, such that whenever $t\le c_\vartheta N^{\vartheta/(2-\vartheta)}$,
\[
c_{2,\vartheta}K\frac{t^{1/\vartheta}}{N}
\leq
c_{1,\vartheta}K\sqrt{\frac{t}{N}}.
\]
Hence, for all $t\le c_\vartheta N^{\vartheta/(2-\vartheta)}$,
\[
c_{1,\vartheta}K\sqrt{\frac{t}{N}}+c_{2,\vartheta}K\frac{t^{1/\vartheta}}{N}
\leq
2c_{1,\vartheta}K\sqrt{\frac{t}{N}}.
\]

Setting $t=\log(2/\delta)$, the preceding three cases imply that there exist constants $C_\vartheta,c_\vartheta>0$, depending only on $\vartheta$, such that
\[
\EE_{\PP^\star}[\ell(\z,\txi)]-\EE_\PPhat[\ell(\z,\txi)]
\leq
C_\vartheta K\sqrt{\frac{\log(2/\delta)}{N}}
\]
with probability at least $1-\delta$, provided that
\[
\delta\in
\begin{cases}
\left[2\exp\!\left(-c_\vartheta N^{\vartheta/(2-\vartheta)}\right),\,1\right), & \vartheta\in(0,1),\\[2mm]
\left[2\exp\!\left(-c_\vartheta N\right),\,1\right), & \vartheta\in[1,2),\\[2mm]
(0,1), & \vartheta=2.
\end{cases}
\]
Since $\log(2/\delta)\le 1+\log(1/\delta)$ for all $\delta\in(0,1)$, enlarging $C_\vartheta$ if necessary yields \eqref{eq:eps_subweibull}. \qed
\end{proof}

\begin{proof}{Proof of Proposition~\ref{prop:markov_coverage}.}
By \citet[Theorem~1]{fan2021hoeffding}, for any $s>0$,
\[
\Prob\!\left(\sum_{i=1}^N \ell(\z,\hat{\bxi}_i) - N\EE_{\PP^\star}[\ell(\z,\txi)] < -s\right)
\leq \exp\!\left(-\frac{1-\lambda}{1+\lambda}\cdot\frac{s^2}{N{(\lbar-\lu)}^2/2}\right).
\]
Dividing by $N$ and substituting $s=Nt$ gives, for any $t\geq 0$,
\[
\Prob\!\left(\EE_{\PP^\star}[\ell(\z,\txi)] - \EE_\PPhat[\ell(\z,\txi)] \geq t\right)
\leq \exp\!\left(-\frac{2(1-\lambda)Nt^2}{(1+\lambda){(\lbar-\lu)}^2}\right).
\]
Equating the right-hand side to $\delta$ and solving for $t$ yields \eqref{eq:eps_markov}. The extension to the time-inhomogeneous setting follows by applying Theorem~5 in \citet{fan2021hoeffding}. Applying the same argument to $-\ell(\z,\cdot)\in[-\lbar,-\lu]$ yields $\EE_\PPhat[\ell(\z,\txi)]-\EE_{\PP^\star}[\ell(\z,\txi)]\leq\frac{1}{\sqrt{N}}\varepsilon\!\left(\log\frac{1}{\delta}\right)$ with probability at least $1-\delta$, so Theorem~\ref{thm:excess_risk} also applies in this setting. \qed
\end{proof}

Following~\cite{jia2024learning}, we make the following standard assumptions regarding the logging policy:
\begin{enumerate}[(i)]
\item\label{as:conditional_exchangeability} \emph{Conditional exchangeability:} Conditional on observed covariates $\bchi$, the selection decision is independent of the outcomes $\bomega$:
\begin{equation*}
\tomega\indep\tilde S ~|~\bchi.
\end{equation*}
\item\label{as:positivity} \emph{Positivity:} The probability of a candidate being selected is strictly positive for any given covariate values:
\begin{equation*}
\pi(\bchi)\geq\underline\pi \qquad\forall\bchi\in\mathscr X,
\end{equation*}
for some positive constant $\underline\pi>0$.
\end{enumerate}

\begin{lemma}\label{lem:IPW_equivalence}
Suppose Assumptions~\ref{as:conditional_exchangeability} and~\ref{as:positivity} hold. Then the inverse probability weighting estimator is equivalent to the true expectation:
\begin{equation*}
\EE_{\PP^\star}[\ell(\z,\txi)]=\EE_{\overline\PP}\left[\ell(\z,\txi)\frac{\I[\tilde S=1]}{\overline\PP(\tilde S=1|\tchi=\bchi)}\right].
\end{equation*}
\end{lemma}
\begin{proof}{Proof of Lemma~\ref{lem:IPW_equivalence}.}
We have
\begin{equation*}
\begin{aligned}
\EE_{\overline\PP}\left[\ell(\z,\txi)\frac{\I[\tilde S=1]}{\pi(\tchi)}\right]
&=\EE_{\overline\PP}\left[\EE_{\overline\PP}\left[\ell(\z,\txi)\frac{\I[\tilde S=1]}{\pi(\tchi)}\Bigm|\tchi,\tomega\right]\right]\\
&=\EE_{\overline\PP}\left[\ell(\z,\txi)\frac{\overline\PP(\tilde S=1\mid\tchi,\tomega)}{\pi(\tchi)}\right]\\
&=\EE_{\overline\PP}\left[\ell(\z,\txi)\frac{\overline\PP(\tilde S=1\mid\tchi)}{\pi(\tchi)}\right]\\
&=\EE_{\overline\PP}\left[\ell(\z,\txi)\right]\\
&=\EE_{\PP^\star}\left[\ell(\z,\txi)\right].
\end{aligned}
\end{equation*}
The first equality is the tower property. The third equality follows from Assumption~\ref{as:conditional_exchangeability}, the fourth uses the definition of $\pi(\bchi)$, and the last holds because $\PP^\star$ is the marginal law of $(\tchi,\tomega)$ under $\overline\PP$. Thus, the claim follows. \qed
\end{proof}

The notation $\EE_{\PPhat}[\ell(\z,\txi)]$ in \eqref{eq:IPW_estimator} denotes an IPW reference functional rather than an expectation under a normalized empirical probability measure. This distinction does not affect the dual reformulations: the generalized moment problem only uses these reference values on the right-hand side of the loss constraints, and the same duality argument applies under the corresponding Slater condition.

\begin{proof}{Proof of Proposition~\ref{prop:incomplete}.}
By positivity,
\[
\left|\frac{\I[\tilde S=1]}{\overline\PP(\tilde S=1|\tchi=\bchi)}\right|\le \frac{1}{\underline\pi}.
\]
Hence, under the assumed bound $\|\ell(\z,\txi)\|_{\psi_2}\le\sigma$,
\[
\left\|\ell(\z,\txi)\frac{\I[\tilde S=1]}{\overline\PP(\tilde S=1|\tchi=\bchi)}\right\|_{\psi_2}
\leq \frac{\sigma}{\underline\pi}.
\]
Centering changes the sub-Gaussian norm by at most a factor of two, so
\[
\left\|\ell(\z,\txi)\frac{\I[\tilde S=1]}{\overline\PP(\tilde S=1|\tchi=\bchi)}
-\EE_{\overline\PP}\left[\ell(\z,\txi)\frac{\I[\tilde S=1]}{\overline\PP(\tilde S=1|\tchi=\bchi)}\right]\right\|_{\psi_2}
\leq \frac{2\sigma}{\underline\pi}.
\]
Let
\[
Y_z\coloneqq \ell(\z,\txi)\frac{\I[\tilde S=1]}{\overline\PP(\tilde S=1|\tchi=\bchi)},\qquad
Y_{z,n}\coloneqq \ell(\z,\bxi_n)\frac{\I[S_n=1]}{\pi(\bchi_n)}.
\]
Then \eqref{eq:IPW_estimator} equals $N^{-1}\sum_{n=1}^N Y_{z,n}$, and the preceding bound applies to the centered variable $Y_z-\EE_{\overline\PP}[Y_z]$. By Lemma~\ref{lem:IPW_equivalence}, $\EE_{\overline\PP}[Y_z]=\EE_{\PP^\star}[\ell(\z,\txi)]$. Therefore, the estimation error can be written as
\[
\EE_{\PP^\star}[\ell(\z,\txi)]-\EE_{\PPhat}[\ell(\z,\txi)]
=-\frac1N\sum_{n=1}^N\left(Y_{z,n}-\EE_{\overline\PP}[Y_z]\right).
\]
The summands in this display are independent, centered, and sub-Gaussian with $\psi_2$ norm bounded by $2\sigma/\underline\pi$. The negative average satisfies the same one-sided concentration bound, so taking the deviation level proportional to $\sqrt{\log(1/\delta)/N}$ and applying the standard concentration inequality for averages of independent centered sub-Gaussian variables \citep[Theorem~2.6.3]{vershynin2018high} yields the stated bound
\[
\EE_{\PP^\star}[\ell(\z,\txi)]-\EE_{\PPhat}[\ell(\z,\txi)]
\leq C_{\mathrm{IPW}}\frac{\sigma}{\underline\pi}\sqrt{\frac{\log(1/\delta)}{N}}
\]
 with the numerical constants absorbed into the universal constant $C_{\mathrm{IPW}}$.\qed
\end{proof}

\subsection{Contextual Optimization}\label{sec:contextual_appendix}

Our framework can also be applied to stochastic optimization with side information. Suppose that a realization $\bc\in\RR^{D_\bc}$ of some contextual covariates is observed, which changes the conditional distribution of $\txi$. To exploit the side information, the decision-maker replaces the expectation in the objective with the conditional expectation
\begin{equation*}
\EE_{\PP^\star}[\ell(\z,\txi)|\tc=\bc],
\end{equation*}
where $\PP^\star$ denotes the joint distribution of $(\tc,\txi)$.
In a data-driven setting, we observe historical data $\{(\bc_n,\bxi_n)\}_{n\in[N]}$ and approximate the conditional expectation using the weighted sum
\begin{equation*}
\EE_{\PPhat}[\ell(\z,\txi)|\tc=\bc]\coloneqq\sum_{n\in[N]}w(\bc,\bc_n)\ell(\z,\bxi_n),
\end{equation*}
where the weights $(w(\bc,\bc_n))_{n\in[N]}$ are designed such that they give higher values to data points that are close to the current realization $\bc$. A popular scheme is the Nadaraya-Watson kernel regression, in which the weights are defined through a kernel function $\mathcal K_h$ with bandwidth parameter $h>0$:
\begin{equation*}
w(\bc,\bc_n)\coloneqq \frac{\mathcal K_h(\bc-\bc_n)}{\sum_{m\in[N]} \mathcal K_h(\bc-\bc_{m})}.
\end{equation*}

The resulting data-driven approximation is guaranteed to provide asymptotic consistency of the optimal solutions. The following proposition establishes a pointwise concentration bound for this data-driven conditional estimator.

\begin{proposition}\label{prop:contextual_bound}
Assume that the loss function takes values in the interval $[0,1]$, the conditional expectation $\EE_{\PP^\star}[\ell(\z,\txi)|\tc=\bc]$ is $L_\bc$-Lipschitz continuous in $\bc$, and the marginal density function $f(\bc)$ satisfies $f(\bc)\ge \underline f>0$ for all $\bc$ in its support. Then, the following pointwise concentration inequality holds:
\begin{equation*}
\EE_{\PP^\star}[\ell(\z,\txi)|\tc=\bc] - \EE_\PPhat[\ell(\z,\txi)|\tc=\bc]\leq  \left(\frac{2\log(1/\delta)}{Nc\underline{f} }\right)^{\frac{1}{D_\bc+2}}L_\bc^{\frac{D_\bc}{D_\bc+2}}\frac{D_\bc+2}{2(D_\bc/2)^{\frac{D_\bc}{D_\bc+2}}}
\end{equation*}
with probability at least $1-\delta$.
\end{proposition}

\begin{remark}[{Contextual rates and dimensionality}]
Note that the bound suffers from the curse of dimensionality in $D_\bc$. This dependence in the dimension is unavoidable since learning the conditional expectation in $D_\bc$-dimensional nonparametric regression has the \emph{minimax} squared-error rate of $O(N^{-2/(D_\bc+2)})$ \citep{gyorfi2002distribution}, corresponding to the absolute-error rate $O(N^{-1/(D_\bc+2)})$ in Proposition~\ref{prop:contextual_bound}. Nevertheless, our DRO framework is applicable to provide a robustification with coverage guarantee $\PP^\star(\cdot\mid\tc=\bc)\in\mP_{\epsilon_N}$ where the radius can be set to
\begin{equation*}
\epsilon_N\coloneqq
\frac{D_\bc+2}{2(D_\bc/2)^{\frac{D_\bc}{D_\bc+2}}}
L_\bc^{\frac{D_\bc}{D_\bc+2}}
\left(
\frac{2\left(D_\x\log\left(1+2R_\x L\sqrt N\right)+\log(1/\delta)\right)}
{Nc\underline f}
\right)^{\frac{1}{D_\bc+2}}
+\frac{2}{\sqrt N}.
\end{equation*}
\end{remark}

\begin{proof}{Proof of Proposition~\ref{prop:contextual_bound}.}
Based on \cite[Proposition~1]{wang2024generalization}, for any $h>0$ and $t>0$ we have
\[
\EE_{\PP^\star}[\ell(\z,\txi)|\tc=\bc] - \EE_\PPhat[\ell(\z,\txi)|\tc=\bc]\leq L_\bc h + t
\]
with probability at least $1-\exp(-Nc\underline f h^{D_\bc}t^2/2)$. Setting $\delta\coloneqq\exp(-Nc\underline f h^{D_\bc}t^2/2)$ gives
\[
\EE_{\PP^\star}[\ell(\z,\txi)|\tc=\bc] - \EE_\PPhat[\ell(\z,\txi)|\tc=\bc]\leq L_\bc h + \sqrt{\frac{2}{Nc\underline{f}h^{D_\bc}}\log\left(\frac{1}{\delta}\right)}
\]
with probability at least $1-\delta$. The choice
\[
h\coloneqq \left(\frac{2\log(1/\delta)}{Nc\underline{f}}\right)^{\frac{1}{D_\bc+2}}\left(\frac{D_\bc}{2L_\bc}\right)^{\frac{2}{D_\bc+2}}
\]
minimizes the right-hand side. At this bandwidth, the two terms satisfy
\[
\sqrt{\frac{2}{Nc\underline{f}h^{D_\bc}}\log\left(\frac{1}{\delta}\right)}
=\frac{2L_\bc h}{D_\bc},
\]
and hence the right-hand side equals
\[
\frac{D_\bc+2}{2(D_\bc/2)^{\frac{D_\bc}{D_\bc+2}}}
L_\bc^{\frac{D_\bc}{D_\bc+2}}
\left(\frac{2\log(1/\delta)}{Nc\underline f}\right)^{\frac{1}{D_\bc+2}},
\]
which is the stated bound.\qed
\end{proof}
\subsection{Dual Formulations}

\begin{proof}{Proof of Theorem~\ref{thm:dual_exact}.}
Fix $\x\in\X$. By Proposition~\ref{prop:semiinf_ambiguity_set}, the inner maximization problem in \eqref{eq:DRO} can be written in terms of the nonnegative measure $\mu$ associated with $\QQ$ as
\begin{equation*}
\begin{aligned}
\sup_{\mu\in\M_+(\Xi)}\;\;&\int_\Xi \ell(\x,\bxi)\,\mu(\mathrm d\bxi)\\
\st\;\;&\int_\Xi \ell(\z,\bxi)\,\mu(\mathrm d\bxi)\leq \EE_\PPhat[\ell(\z,\txi)]+\epsilon\qquad\forall \z\in\X,\\
&\int_\Xi \|\bxi\|^2\,\mu(\mathrm d\bxi)\leq \Omega,\\
&\int_\Xi 1\,\mu(\mathrm d\bxi)=1.
\end{aligned}
\end{equation*}
This is a generalized moment problem over nonnegative measures. To make the dual construction explicit, introduce a nonnegative measure $\nu\in\M_+(\X)$ for the family of loss constraints, a scalar multiplier $\beta\in\RR_+$ for the second-moment constraint, and a free scalar $\alpha\in\RR$ for the normalization constraint. For any $\mu\in\M_+(\Xi)$, the Lagrangian upper bound is
\begin{equation*}
\begin{aligned}
\mathcal L(\mu,\alpha,\beta,\nu)
&\coloneqq \int_\Xi \ell(\x,\bxi)\,\mu(\mathrm d\bxi)
 + \alpha\left(1-\int_\Xi \mu(\mathrm d\bxi)\right)
 + \beta\left(\Omega-\int_\Xi \|\bxi\|^2\,\mu(\mathrm d\bxi)\right)\\
&\qquad
 + \int_\X\left(\EE_\PPhat[\ell(\z,\txi)]+\epsilon
 - \int_\Xi \ell(\z,\bxi)\,\mu(\mathrm d\bxi)\right)\nu(\mathrm d\z)\\
&= \alpha+\beta\Omega
 +\int_\X\left(\EE_\PPhat[\ell(\z,\txi)]+\epsilon\right)\nu(\mathrm d\z)
 +\int_\Xi\left(\ell(\x,\bxi)-\alpha-\beta\|\bxi\|^2
 -\int_\X\ell(\z,\bxi)\nu(\mathrm d\z)\right)\mu(\mathrm d\bxi).
\end{aligned}
\end{equation*}
The interchange of the $\X$- and $\Xi$-integrals is justified for every primal feasible $\mu$ by Assumptions~\ref{as:linear_growth} and~\ref{as:second_moment}:
\[
\int_\X\int_\Xi |\ell(\z,\bxi)|\,\mu(\mathrm d\bxi)\nu(\mathrm d\z)
\leq \nu(\X)\int_\Xi (c_1+c_2\|\bxi\|)\,\mu(\mathrm d\bxi)
\leq \nu(\X)(c_1+c_2\sqrt{\Omega})<+\infty.
\]
Thus weak duality gives
\begin{equation*}
\begin{aligned}
\sup_{\QQ\in\mP_\epsilon}\EE_\QQ[\ell(\x,\txi)]
\leq \inf_{\alpha\in\RR,\;\beta\in\RR_+,\;\nu\in\M_+(\X)}
\sup_{\mu\in\M_+(\Xi)} \mathcal L(\mu,\alpha,\beta,\nu).
\end{aligned}
\end{equation*}
For fixed $(\alpha,\beta,\nu)$, the inner supremum over $\mu\in\M_+(\Xi)$ is finite if and only if
\[
\ell(\x,\bxi)-\alpha-\beta\|\bxi\|^2-\int_\X\ell(\z,\bxi)\nu(\mathrm d\z)\leq 0
\qquad\forall \bxi\in\Xi.
\]
Indeed, if this inequality fails at some $\bxi\in\Xi$, assigning arbitrarily large mass to the Dirac measure at $\bxi$ makes the supremum equal to $+\infty$; if the inequality holds everywhere, the supremum equals $0$. Hence the dual upper bound is
\begin{equation*}
\begin{aligned}
\inf\;\;&\alpha+\beta\Omega+\int_\X\left(\EE_\PPhat[\ell(\z,\txi)]+\epsilon\right)\nu(\mathrm d\z)\\
\st\;\;&\ell(\x,\bxi)\leq \alpha+\beta\|\bxi\|^2+\int_\X \ell(\z,\bxi)\nu(\mathrm d\z)\qquad\forall \bxi\in\Xi.
\end{aligned}
\end{equation*}
Because $\epsilon>0$ and Assumption~\ref{as:second_moment} gives $\EE_\PPhat[\|\txi\|^2]<\Omega$, the measure $\PPhat$ is strictly feasible for the primal inner problem:
\[
\int_\Xi \ell(\z,\bxi)\,\PPhat(\mathrm d\bxi)
=\EE_\PPhat[\ell(\z,\txi)]
<\EE_\PPhat[\ell(\z,\txi)]+\epsilon\qquad\forall\z\in\X,
\]
and the moment constraint is also satisfied with strict slack. Standard strong duality for generalized moment problems therefore applies, and dual attainment holds for each fixed $\x$; see \citet[Proposition~5.2]{Shapiro01}.

Now let $F(\x)\coloneqq \sup_{\QQ\in\mP_\epsilon}\EE_\QQ[\ell(\x,\txi)]$. The mapping $\x\mapsto F(\x)$ is lower semicontinuous as a pointwise supremum of continuous functions, and $\X$ is compact, so the outer minimization attains its optimum. Combining this minimizer with the attained inner dual solution yields the stated infinite linear programming reformulation and the existence of an optimal tuple $(\x,\alpha,\beta,\nu)$.\qed
\end{proof}

\begin{proof}{Proof of Proposition~\ref{prop:convergence_eps_to_0}.}
Fix $\x\in\X$, and define $f_\epsilon(\x)\coloneqq \sup_{\QQ\in\mP_\epsilon}\EE_\QQ[\ell(\x,\txi)]$. Because $\mP_0\subseteq\mP_\epsilon$ for every $\epsilon\ge 0$, we have $f_0(\x)\le f_\epsilon(\x)$. Moreover, by Proposition~\ref{prop:semiinf_ambiguity_set},
\[
f_0(\x)=\sup_{\QQ\in\mP_0}\EE_\QQ[\ell(\x,\txi)] = \EE_\PPhat[\ell(\x,\txi)],
\]
since $\PPhat\in\mP_0$ and every $\QQ\in\mP_0$ satisfies $\EE_\QQ[\ell(\x,\txi)]\le \EE_\PPhat[\ell(\x,\txi)]$.

It remains to prove $\limsup_{\epsilon\downarrow 0}f_\epsilon(\x)\le f_0(\x)$. Choose a sequence $\epsilon_k\downarrow0$ such that $f_{\epsilon_k}(\x)\to\limsup_{\epsilon\downarrow0}f_\epsilon(\x)$, and choose $\QQ_k\in\mP_{\epsilon_k}$ such that
\[
f_{\epsilon_k}(\x)-\frac{1}{k}\le \EE_{\QQ_k}[\ell(\x,\txi)].
\]
Because every $\QQ_k$ satisfies the common second-moment constraint, Markov's inequality implies that the family $\{\QQ_k\}_{k\in\N}$ is tight; see \cite[Section~5]{billingsley2013convergence}. After passing to a subsequence if necessary, Prohorov's theorem \cite[Theorem~5.1]{billingsley2013convergence} yields $\QQ_k\Rightarrow\QQ^\star$ for some probability measure $\QQ^\star$. Assumption~\ref{as:linear_growth} and the uniform second-moment bound imply
\[
\sup_{k\in\N}\EE_{\QQ_k}\!\left[|\ell(\z,\txi)|^2\right]
\leq 2c_1^2+2c_2^2\Omega<+\infty\qquad\forall \z\in\X,
\]
so $\{\ell(\z,\cdot)\}$ is uniformly integrable under $\{\QQ_k\}_{k\in\N}$ for every fixed $\z\in\X$. Since $\ell(\x,\cdot)$ is continuous, the mapping theorem \cite[Theorem~2.7]{billingsley2013convergence} implies weak convergence of the pushforward laws, and \cite[Theorem~3.5]{billingsley2013convergence} therefore gives
\[
\lim_{k\to\infty}\EE_{\QQ_k}[\ell(\x,\txi)] = \EE_{\QQ^\star}[\ell(\x,\txi)].
\]
Applying the same weak-convergence-plus-uniform-integrability argument to any fixed $\z\in\X$, and using $\QQ_k\in\mP_{\epsilon_k}$, we have
\[
\EE_{\QQ_k}[\ell(\z,\txi)]\le \EE_\PPhat[\ell(\z,\txi)] + \epsilon_k
\qquad\forall \z\in\X.
\]
Letting $k\to\infty$ yields
\[
\EE_{\QQ^\star}[\ell(\z,\txi)]\le \EE_\PPhat[\ell(\z,\txi)]
\qquad\forall \z\in\X,
\]
and the second-moment constraint also passes to the limit: since $\|\bxi\|^2\geq 0$ is lower semicontinuous, the Portmanteau theorem gives $\EE_{\QQ^\star}[\|\txi\|^2]\leq\liminf_{k\to\infty}\EE_{\QQ_k}[\|\txi\|^2]\leq\Omega$. Hence $\QQ^\star\in\mP_0$, so
\[
\limsup_{\epsilon\downarrow0} f_\epsilon(\x)
\leq \lim_{k\to\infty} f_{\epsilon_k}(\x)
\le \lim_{k\to\infty}\EE_{\QQ_k}[\ell(\x,\txi)]
=\EE_{\QQ^\star}[\ell(\x,\txi)]
\le f_0(\x).
\]
Therefore $f_\epsilon(\x)\downarrow f_0(\x)$ pointwise on $\X$.

Let $v_\epsilon\coloneqq \min_{\x\in\X} f_\epsilon(\x)$ and $v_0\coloneqq \min_{\x\in\X} f_0(\x)$. Since $f_\epsilon\ge f_0$ pointwise, we have $v_\epsilon\ge v_0$. On the other hand, for any minimizer $\x^\star\in\arg\min_{\x\in\X}f_0(\x)$,
\[
\limsup_{\epsilon\downarrow 0} v_\epsilon \le \limsup_{\epsilon\downarrow 0} f_\epsilon(\x^\star)=f_0(\x^\star)=v_0.
\]
Hence $v_\epsilon\to v_0$.

Finally, let $\hat\x_\epsilon$ be a minimizer of \eqref{eq:dual_exact} and let $\hat\x^\star$ be any cluster point of a sequence $\{\hat\x_{\epsilon_k}\}_{k\in\N}$ with $\epsilon_k\downarrow 0$. By compactness of $\X$, after passing to a subsequence we may assume $\hat\x_{\epsilon_k}\to\hat\x^\star$. Since $f_0(\x)=\EE_\PPhat[\ell(\x,\txi)]$ is continuous in $\x$,
\[
f_0(\hat\x^\star)=\lim_{k\to\infty}f_0(\hat\x_{\epsilon_k})
\le \lim_{k\to\infty}f_{\epsilon_k}(\hat\x_{\epsilon_k})
= \lim_{k\to\infty}v_{\epsilon_k}
= v_0.
\]
Thus $\hat\x^\star$ is a minimizer of the empirical risk minimization problem \eqref{eq:ERM}.\qed
\end{proof}

\section{Proofs of Section~\ref{sec:mc_sampling}}
\label{appendix:proofs_sec3}

\subsection{Theoretical Guarantees}

\begin{proposition}\label{prop:hinge_SAA_guarantee}
With probability at least $1-\tau$, we have: 
\begin{equation*}
\begin{aligned}
&\EE_\ZZ\left[\left[\EE_\QQ[\ell(\tz,\txi)]-\EE_\PPhat[\ell(\tz,\txi)]-\eps\right]_+\right]\\
&\leq  \frac{1}{M}\sum_{m\in[M]}\left[\EE_\QQ[\ell(\z_m,\txi)]-\EE_\PPhat[\ell(\z_m,\txi)]-\eps\right]_++\mathcal O\left( R\sqrt{\frac{J\log J(\log M)^3}{M}}\right) \\
&\qquad+ 2(c_1 + c_2\sqrt{\Omega}) \sqrt{\frac{1}{2M}\log\left(\frac{1}{\tau}\right)}\qquad\qquad\forall \QQ\in\mathscr P(\Xi): \EE_\QQ[\|\txi\|^2]\leq \Omega.
\end{aligned}
\end{equation*}
\end{proposition}
\begin{proof}{Proof of Proposition~\ref{prop:hinge_SAA_guarantee}.}
We define 
$h_\QQ(\z)\coloneqq\left[\EE_\QQ[\ell(\z,\txi)]-\EE_\PPhat[\ell(\z,\txi)]-\eps\right]_+$ and let
\begin{equation*}\HH\coloneqq \{h_\QQ:\QQ\in\mathscr P(\Xi),\;\EE_\QQ[\|\txi\|^2]\leq \Omega\}.
\end{equation*}
Our Assumption \ref{as:linear_growth} implies that for any $h\in\HH$ and $\z\in\X$, we have $\bar h \coloneqq 2(c_1 + c_2\sqrt{\Omega})  \geq h(\z) \geq 0$.

For a given sample set $\{\tz_m\}_{m\in[M]}$, we define the uniform approximation error as:
\begin{equation*}
e(\tz_1,\ldots,\tz_M)\coloneqq \sup_{h\in\HH}\left(\EE_\ZZ[h(\tz)]- \frac{1}{M}\sum_{m\in[M]} h(\tz_m)\right).
\end{equation*}
Note that changing one data point $\tz_m$ changes $e$ by at most $\bar h/M$. McDiarmid's inequality~\citep[Appendix D]{mohri2018foundations} therefore yields
\begin{equation*}
e(\tz_1,\ldots,\tz_M)\leq \EE[e(\tz_1,\ldots,\tz_M)] + \bar h \sqrt{\frac{1}{2M}\log\left(\frac{1}{\tau}\right)}
\end{equation*}
with probability at least $1-\tau$. 

We next upper bound the expectation using Rademacher complexity. Introduce $M$ i.i.d. ghost samples $\tz_1',\ldots,\tz_M'$ drawn from $\ZZ$ and let $\ts_1,\ldots,\ts_M\in\{-1,1\}$ be i.i.d. Rademacher random variables. A standard symmetrization inequality \cite[Theorem~3.3]{mohri2018foundations} yields
\begin{equation*}
\begin{aligned}
\EE\left[\sup_{h\in\HH}\left(\EE_\ZZ[h(\tz)]- \frac{1}{M}\sum_{m\in[M]} h(\tz_m)\right)\right]
&\leq \EE\left[\sup_{h\in\HH}\left(\frac{1}{M}\sum_{m\in[M]} h(\tz_m')- \frac{1}{M}\sum_{m\in[M]} h(\tz_m)\right)\right]\\
&= \EE\left[\sup_{h\in\HH}\left(\frac{1}{M}\sum_{m\in[M]} \ts_m(h(\tz_m')-h(\tz_m))\right)\right]\\
&\leq 2 \EE\left[\sup_{h\in\HH}\left(\frac{1}{M}\sum_{m\in[M]} \ts_mh(\tz_m)\right)\right]
= 2\mathcal R_M(\HH).
\end{aligned}
\end{equation*}
Since the hinge map $u\mapsto[u]_+$ is $1$-Lipschitz, the contraction principle \cite[Lemma~26.9]{shalev2014understanding} gives
\begin{equation}
\label{eq:rademacher_bound}
\begin{aligned}
\mathcal R_M(\HH) &\leq \EE\left[\sup_{\substack{\QQ\in\mathscr P(\Xi)\\\EE_\QQ[\|\txi\|^2]\leq \Omega}}\left(\frac{1}{M}\sum_{m\in[M]} \ts_m\left(\EE_\QQ[\ell(\tz_m,\txi)]-\EE_\PPhat[\ell(\tz_m,\txi)]-\eps\right)\right)\right]\\
&=\EE\left[\sup_{\substack{\QQ\in\mathscr P(\Xi)\\\EE_\QQ[\|\txi\|^2]\leq \Omega}}\EE_\QQ\left[\frac{1}{M}\sum_{m\in[M]} \ts_m\ell(\tz_m,\txi)\right]\right]- \EE\left[\frac{1}{M}\sum_{m\in[M]}\ts_m\left(\EE_\PPhat[\ell(\tz_m,\txi)]+\eps\right)\right]\\
&\leq\EE\left[\sup_{\QQ\in\mathscr P(\Xi)}\EE_\QQ\left[\frac{1}{M}\sum_{m\in[M]} \ts_m\ell(\tz_m,\txi)\right]\right]\\
&= \EE\left[\sup_{\bxi\in\Xi}\frac{1}{M}\sum_{m\in[M]} \ts_m\ell(\tz_m,\bxi)\right].
\end{aligned}
\end{equation}
In the second line, we moved the term $\frac{1}{M}\sum_{m\in[M]}\ts_m\left(\EE_\PPhat[\ell(\tz_m,\txi)]+\eps\right)$ outside the supremum as it is independent of $\QQ$. Since each $\ts_m$ is centered and independent of $\tz_m$, this term has expectation zero.

Each affine piece in the loss function can be rewritten as
\begin{equation*}
\bm a_j(\bm x)^\top\bxi + b_j(\bm x)
=\left(\begin{bmatrix}\bm A_j^\top &\bm b_j\\ \overline {\bm a}_j^\top & \overline b_j\end{bmatrix}\begin{bmatrix}\bxi\\ 1\end{bmatrix}\right)^\top \begin{bmatrix}\x\\ 1\end{bmatrix}.
\end{equation*}
Let $\mathcal B^{D_\x+1}$ be the unit ball in $\RR^{D_\x+1}$. For each $\bxi\in\Xi$ and $j\in[J]$, define
\[
\bm w_j(\bxi)\coloneqq \left(\max_{i\in[J]}\left\|\begin{bmatrix}\bm A_i^\top & \bm b_i\\ \overline{\bm a}_i^\top & \overline b_i\end{bmatrix}\right\|_{\mathrm{op}}\sqrt{R_\bxi^2+1}\right)^{-1}
\begin{bmatrix}\bm A_j^\top & \bm b_j\\ \overline{\bm a}_j^\top & \overline b_j\end{bmatrix}
\begin{bmatrix}\bxi\\ 1\end{bmatrix},
\]
so that $\bm w_j(\bxi)\in\mathcal B^{D_\x+1}$ for all $j\in[J]$, and also $\bigl(\sqrt{R_\x^2+1}\,\bigr)^{-1}\bigl[\begin{smallmatrix}\x\\ 1\end{smallmatrix}\bigr]\in\mathcal B^{D_\x+1}$.
Hence
\begin{equation}
\label{eq:expected_wc_cost}
\begin{aligned}
\mathcal R_M(\HH)
&\leq \EE\left[\sup_{\bxi\in\Xi}\frac{1}{M}\sum_{m\in[M]} \ts_m\ell(\tz_m,\bxi)\right]\\
&= \EE\left[\sup_{\bxi\in\Xi}\frac{1}{M}\sum_{m\in[M]} \ts_m\max_{j\in[J]}\bigl(\bm a_j(\tz_m)^\top\bxi + b_j(\tz_m)\bigr) \right]\\
&\leq \EE\left[\sup_{\bxi_1,\ldots,\bxi_J\in\Xi}\frac{1}{M}\sum_{m\in[M]} \ts_m\max_{j\in[J]}\bigl(\bm a_j(\tz_m)^\top\bxi_j + b_j(\tz_m)\bigr) \right]\\
&\leq \max_{j\in[J]}\left\|\begin{bmatrix}\bm A_j^\top &\bm b_j\\ \overline {\bm a}_j^\top & \overline b_j\end{bmatrix}\right\|_{\textup{op}}\sqrt{R_\bxi^2+1}\sqrt{R_\x^2+1}\cdot \EE\left[\sup_{\bm w_1,\ldots,\bm w_J\in\mathcal B^{D_{\x}+1}}\frac{1}{M}\sum_{m\in[M]} \ts_m\max_{j\in[J]}\bm w_j^\top \left(\sqrt{R_\x^2+1}\right)^{-1}\begin{bmatrix}\tz_m\\ 1\end{bmatrix} \right]\\
&\leq \mathcal O \left(\max_{j\in[J]}\left\|\begin{bmatrix}\bm A_j^\top &\bm b_j\\ \overline {\bm a}_j^\top & \overline b_j\end{bmatrix}\right\|_{\textup{op}}\sqrt{R_\bxi^2+1}\sqrt{R_\x^2+1}\sqrt{\frac{J\log J(\log M)^3}{M}}\right),
\end{aligned}
\end{equation}
where the last step uses the Rademacher complexity of $J$-fold maxima of hyperplanes \cite[Corollary~5]{attias2024fat}. Combining the McDiarmid and Rademacher bounds yields the claim.\qed
\end{proof}

\begin{proof}{Proof of Theorem~\ref{coro:set_containment}.}
Consider the high-probability event of Proposition~\ref{prop:hinge_SAA_guarantee}, which occurs with probability at least $1-\tau$. Fix any $\QQ\in\mP_\epsilon^M$. By definition of $\mP_\epsilon^M$,
\[
\frac{1}{M}\sum_{m\in[M]}\left[\EE_\QQ[\ell(\z_m,\txi)]-\EE_\PPhat[\ell(\z_m,\txi)]-\eps\right]_+ \le 0
\quad\text{and}\quad
\EE_\QQ[\|\txi\|^2]\le \Omega.
\]
Applying Proposition~\ref{prop:hinge_SAA_guarantee} to this $\QQ$ shows that
\[
\EE_\ZZ\left[\left[\EE_\QQ[\ell(\tz,\txi)]-\EE_\PPhat[\ell(\tz,\txi)]-\eps\right]_+\right]\le \eta,
\]
where $\eta$ is the bound in \eqref{eq:eta}. Together with the shared second-moment constraint, this means $\QQ\in\mP_\epsilon(\eta)$. Since $\QQ\in\mP_\epsilon^M$ was arbitrary, we obtain $\mP_\epsilon^M\subseteq\mP_\epsilon(\eta)$.\qed
\end{proof}

\begin{lemma}\label{lem:convergence_relaxed_optimal_value}
As $\eta\downarrow 0$, the optimal value of the relaxed DRO problem
\begin{equation*}
\min_{\x\in\X} \sup_{\QQ\in\mP_\epsilon(\eta)}  \EE_{\QQ}[\ell(\x,\txi)]
\end{equation*}
converges to that of the exact DRO problem.
\end{lemma}
\begin{proof}{Proof of Lemma~\ref{lem:convergence_relaxed_optimal_value}.}
We follow the development of the proof of Proposition~\ref{prop:convergence_eps_to_0}. Fix $\x\in\X$, and define $f_\eta(\x)\coloneqq \sup_{\QQ\in\mP_\epsilon(\eta)}\EE_\QQ[\ell(\x,\txi)]$. Since $\mP_\epsilon=\mP_\epsilon(0)\subseteq\mP_\epsilon(\eta)$ for every $\eta\geq 0$, we have $f_0(\x)\leq f_\eta(\x)$. Thus, it remains to prove
\begin{equation}
\label{eq:limsup_convergence_2}
\lim_{\eta\downarrow 0} f_\eta(\x)\leq f_0(\x).
\end{equation}
Let $\eta_k\downarrow 0$ be any decreasing sequence converging to zero. For each $k\in\N$, choose $\QQ_k\in\mP_\epsilon(\eta_k)$ such that
\begin{equation}
\label{eq:near_optimal_Q_k}
f_{\eta_k}(\x)-\frac{1}{k}\leq  \EE_{\QQ_k}[\ell(\x,\txi)].
\end{equation}
Because each $\QQ_k$ satisfies the second-moment constraint in $\mP_\epsilon(\eta_k)$, the family $\{\QQ_k\}_{k\in\N}$ is tight. Hence, after passing to a subsequence if necessary, $\QQ_k\Rightarrow \QQ^\star$ for some probability measure $\QQ^\star$.

By the same argument as in Proposition~\ref{prop:convergence_eps_to_0}, weak convergence together with the uniform second-moment constraint implies $\lim_{k\rightarrow\infty}\EE_{\QQ_k}[\ell(\x,\txi)]= \EE_{\QQ^\star}[\ell(\x,\txi)]$, and $\QQ^\star$ satisfies the second-moment constraint in $\mP_\epsilon$. It remains to be verified that the expected hinge constraint holds. For each fixed $\z\in\X$, the convergence of expectations implies
\begin{equation*}
\lim_{k\rightarrow\infty}\left[\EE_{\QQ_k}[\ell(\z,\txi)]-\EE_\PPhat[\ell(\z,\txi)]-\eps\right]_+= \left[\EE_{\QQ^\star}[\ell(\z,\txi)]-\EE_\PPhat[\ell(\z,\txi)]-\eps\right]_+.
\end{equation*}
By Assumptions~\ref{as:linear_growth} and~\ref{as:second_moment}, the hinge term is bounded uniformly in $\z\in\X$ and $k\in\N$ by a finite constant, so dominated convergence yields
\begin{equation*}
\lim_{k\rightarrow\infty}\EE_\ZZ\left[\left[\EE_{\QQ_k}[\ell(\tz,\txi)]-\EE_\PPhat[\ell(\tz,\txi)]-\eps\right]_+\right] = \EE_\ZZ\left[\left[\EE_{\QQ^\star}[\ell(\tz,\txi)]-\EE_\PPhat[\ell(\tz,\txi)]-\eps\right]_+\right].
\end{equation*}
Since $\QQ_k\in\mP_\epsilon(\eta_k)$, we have
\[
\EE_\ZZ\left[\left[\EE_{\QQ_k}[\ell(\tz,\txi)]-\EE_\PPhat[\ell(\tz,\txi)]-\eps\right]_+\right]\leq \eta_k.
\]
Letting $k\rightarrow\infty$, we obtain $\QQ^\star\in\mP_\epsilon$. Taking the limit in \eqref{eq:near_optimal_Q_k} gives
\[
\lim_{k\rightarrow\infty}f_{\eta_k}(\x)\leq \lim_{k\rightarrow\infty} \EE_{\QQ_k}[\ell(\x,\txi)]=\EE_{\QQ^\star}[\ell(\x,\txi)]\leq f_0(\x).
\]
Since $\eta_k\downarrow 0$ was arbitrary, \eqref{eq:limsup_convergence_2} holds. Thus, for any $\x\in\X$,
\begin{equation}
\label{eq:pointwise_convergence_eta}
\lim_{\eta\downarrow 0} \sup_{\QQ\in\mP_\epsilon(\eta)} \EE_{\QQ}[\ell(\x,\txi)]= \sup_{\QQ\in\mP_\epsilon}\EE_{\QQ}[\ell(\x,\txi)].
\end{equation}
The optimal value convergence then follows from the monotonicity of $f_\eta(\x)$ in $\eta$.\qed
\end{proof}

\begin{proof}{Proof of Theorem~\ref{thm:approximation_convergence}.}
For $\eta\ge 0$, define
\[
\hat v(\eta)\coloneqq \min_{\x\in\X}\sup_{\QQ\in\mP_\epsilon(\eta)}\EE_\QQ[\ell(\x,\txi)].
\]
By Theorem~\ref{coro:set_containment}, setting $\eta_M$ to \eqref{eq:eta}, we have for any $\tau>0$,
\[
\Prob\!\left(
\inf_{\x\in\X}\sup_{\QQ\in\mP_\epsilon^M}\EE_\QQ[\ell(\x,\txi)]
\le
\inf_{\x\in\X}\sup_{\QQ\in\mP_\epsilon(\eta_M)}\EE_\QQ[\ell(\x,\txi)]
\right)\ge 1-\tau.
\]
Equivalently,
\[
\Prob(\hat v_M\le \hat v(\eta_M))\ge 1-\tau.
\]
Since $\mP_\epsilon\subseteq\mP_\epsilon^M$, we also have $\hat v\le \hat v_M$. Therefore,
\[
\Prob(\hat v\le \hat v_M\le \hat v(\eta_M))\ge 1-\tau.
\]
By Lemma~\ref{lem:convergence_relaxed_optimal_value}, $\hat v(\eta_M)\to \hat v$ as $M\to\infty$. Hence, for any $\rho>0$, there exists $M'\in\N$ such that
\[
\hat v(\eta_M)-\hat v\le \rho \qquad \forall M\ge M'.
\]
Combining this bound with the preceding high-probability event yields
\[
\Prob(\hat v_M-\hat v\le \rho)\ge 1-\tau \qquad \forall M\ge M'.
\]
Since also $\hat v_M-\hat v\ge 0$, we conclude that $\hat v_M\overset{p}{\to}\hat v$.

We next establish convergence of optimal solutions. For $\delta>0$, define$\mathcal A_\delta\coloneqq \{\x\in\X:\textup{dist}(\x,\mathcal S)\ge \delta\}$.
This set is closed because the distance function is continuous. Since $\X$ is compact, $\mathcal A_\delta$ is compact. Moreover, by definition of $\mathcal S$, no point in $\mathcal A_\delta$ is optimal. Hence,
\[
\rho_\delta\coloneqq \inf_{\x\in\mathcal A_\delta}\sup_{\QQ\in\mP_\epsilon}\EE_\QQ[\ell(\x,\txi)]-\inf_{\x\in\X}\sup_{\QQ\in\mP_\epsilon}\EE_\QQ[\ell(\x,\txi)]>0.
\]
Now, if $\textup{dist}(\hat\x_M,\mathcal S)\ge \delta$, then $\hat\x_M\in\mathcal A_\delta$, and therefore
\[
\sup_{\QQ\in\mP_\epsilon}\EE_\QQ[\ell(\hat\x_M,\txi)]-\inf_{\x\in\X}\sup_{\QQ\in\mP_\epsilon}\EE_\QQ[\ell(\x,\txi)]\ge \rho_\delta.
\]
Since $\mP_\epsilon\subseteq\mP_\epsilon^M$, we also have
\[
\sup_{\QQ\in\mP_\epsilon}\EE_\QQ[\ell(\hat\x_M,\txi)]\le \sup_{\QQ\in\mP_\epsilon^M}\EE_\QQ[\ell(\hat\x_M,\txi)]=\hat v_M.
\]
Thus, $\{\textup{dist}(\hat\x_M,\mathcal S)\ge \delta\}\subseteq\{\hat v_M-\hat v\ge \rho_\delta\}$, and consequently
\[
\Prob(\textup{dist}(\hat\x_M,\mathcal S)\ge \delta)\le \Prob(\hat v_M-\hat v\ge \rho_\delta).
\]
Since $\hat v_M\overset{p}{\to}\hat v$, the right-hand side converges to zero. Therefore, $\textup{dist}(\hat\x_M,\mathcal S)\overset{p}{\to}0$.

Lastly, we prove the suboptimality bound. We have
\[
\begin{aligned}
\sup_{\QQ\in\mP_\epsilon}\EE_\QQ[\ell(\hat\x_M,\txi)]-\sup_{\QQ\in\mP_\epsilon}\EE_\QQ[\ell(\hat\x,\txi)]
=\;&\sup_{\QQ\in\mP_\epsilon}\EE_\QQ[\ell(\hat\x_M,\txi)]-\sup_{\QQ\in\mP_\epsilon^M}\EE_\QQ[\ell(\hat\x_M,\txi)]\\
&+\sup_{\QQ\in\mP_\epsilon^M}\EE_\QQ[\ell(\hat\x_M,\txi)]-\sup_{\QQ\in\mP_\epsilon}\EE_\QQ[\ell(\hat\x,\txi)]\\
\le\;&\sup_{\QQ\in\mP_\epsilon^M}\EE_\QQ[\ell(\hat\x,\txi)]-\sup_{\QQ\in\mP_\epsilon}\EE_\QQ[\ell(\hat\x,\txi)].
\end{aligned}
\]
Here, the first difference is nonpositive because $\mP_\epsilon\subseteq\mP_\epsilon^M$, and the second difference is bounded above by replacing $\hat\x_M$ with $\hat\x$, since $\hat\x$ is suboptimal for the approximate problem.

Next, by Theorem~\ref{coro:set_containment}, setting $\eta$ to \eqref{eq:eta}, we obtain the high-probability bound
\[
\sup_{\QQ\in\mP_\epsilon^M}\EE_\QQ[\ell(\hat\x,\txi)]-\sup_{\QQ\in\mP_\epsilon}\EE_\QQ[\ell(\hat\x,\txi)]
\le
\sup_{\QQ\in\mP_\epsilon(\eta)}\EE_\QQ[\ell(\hat\x,\txi)]-\sup_{\QQ\in\mP_\epsilon}\EE_\QQ[\ell(\hat\x,\txi)]
\]
with probability at least $1-\tau$. For any fixed $\x$ and $\eta\ge 0$, the relaxed problem \eqref{eq:eta_ambiguity_set} has the dual reformulation
\begin{equation*}
\begin{aligned}
\min \;&  \alpha+\beta\Omega+\gamma\eta+\int_\X\left(\EE_\PPhat[\ell(\z,\txi)]+\epsilon\right)\nu(\mathrm d\z)\\
\st & \x\in\X,\;\alpha\in\RR,\;\beta,\gamma\in\RR_+,\;\nu\in\M_+(\X),\\
& \ell(\x,\bxi)\le \alpha+\beta\|\bxi\|^2+\int_\X \ell(\z,\bxi)\nu(\mathrm d\z)\qquad\forall \bxi\in\Xi,\\
& \gamma\ZZ-\nu\in\M_+(\X),
\end{aligned}
\end{equation*}
which follows from the same generalized-moment duality argument as in Theorem~\ref{thm:dual_exact}. Let $(\hat\x,\alpha^\star,\beta^\star,\nu^\star)$ be an optimal solution to the exact dual problem, and let $\gamma^\star$ be as defined in the theorem statement. Then $(\hat\x,\alpha^\star,\beta^\star,\gamma^\star,\nu^\star)$ is feasible, though generally suboptimal, for the relaxed dual problem. We can therefore further upper bound the preceding display by
\[
\begin{aligned}
&\sup_{\QQ\in\mP_\epsilon(\eta)}\EE_\QQ[\ell(\hat\x,\txi)]-\sup_{\QQ\in\mP_\epsilon}\EE_\QQ[\ell(\hat\x,\txi)]\\
\le\;& \alpha^\star+\beta^\star\Omega+\int_\X \left(\EE_\PPhat[\ell(\z,\txi)]+\epsilon\right)\nu^\star(\mathrm d\z)+\gamma^\star\eta\\
&\qquad-\left(\alpha^\star+\beta^\star\Omega+\int_\X \left(\EE_\PPhat[\ell(\z,\txi)]+\epsilon\right)\nu^\star(\mathrm d\z)\right)\\
=\;&\gamma^\star\eta.
\end{aligned}
\]
Substituting the expression for $\eta$ in \eqref{eq:eta} completes the proof.\qed
\end{proof}

\subsection{Conic Programming Reformulations}

\begin{proof}{Proof of Theorem~\ref{theorem:Conic_reformulation}.}
For each sampled point $\z_m$, define
\[
g_m(\bxi)\coloneqq \max_{k\in[J]}\left(\bm a_k(\z_m)^\top\bxi+b_k(\z_m)\right).
\]
The semi-infinite constraint in \eqref{eq:dual_SAA} is equivalent to
\[
\bm a_j(\x)^\top\bxi+b_j(\x)\le \alpha+\beta\|\bxi\|^2+\frac{1}{M}\sum_{m\in[M]}\nu_m g_m(\bxi)
\qquad \forall \bxi\in\Xi,\;\forall j\in[J].
\]
For each pair $(j,m)$, the term $\nu_m g_m(\bxi)$ admits the simplex representation
\[
\nu_m g_m(\bxi)=\max_{\substack{\bm\lambda_{jm}\in\RR_+^J\\ \mathbf e^\top\bm\lambda_{jm}=\nu_m}}\sum_{k\in[J]}\lambda_{jm}^k\bigl(\bm a_k(\z_m)^\top\bxi+b_k(\z_m)\bigr).
\]
For fixed $j$, let $\Lambda_j\coloneqq\prod_{m\in[M]}\{\bm\lambda_{jm}\in\RR_+^J:\mathbf e^\top\bm\lambda_{jm}=\nu_m\}$.

The constraint for this $j$ can then be written as
\[
\sup_{\bxi\in\Xi}\min_{\{\bm\lambda_{jm}\}_{m\in[M]}\in\Lambda_j}\Phi_j(\bxi,\{\bm\lambda_{jm}\}_{m\in[M]})\leq \alpha,
\]
where
\[
\begin{aligned}
\Phi_j(\bxi,\{\bm\lambda_{jm}\}_{m\in[M]})
\coloneqq\;&
\bm a_j(\x)^\top\bxi+b_j(\x)-\beta\|\bxi\|^2-\frac{1}{M}\sum_{m\in[M],k\in[J]}\lambda_{jm}^k\bigl(\bm a_k(\z_m)^\top\bxi+b_k(\z_m)\bigr).
\end{aligned}
\]
The set $\Lambda_j$ is compact and convex, and $\Xi$ is compact and convex in the present compact-support setting. Moreover, $\Phi_j$ is affine in $\{\bm\lambda_{jm}\}_{m\in[M]}$ and concave upper semicontinuous in $\bxi$ because $\beta\ge0$. Sion's minimax theorem therefore yields
\[
\sup_{\bxi\in\Xi}\min_{\{\bm\lambda_{jm}\}_{m\in[M]}\in\Lambda_j}\Phi_j(\bxi,\{\bm\lambda_{jm}\}_{m\in[M]})
=
\min_{\{\bm\lambda_{jm}\}_{m\in[M]}\in\Lambda_j}\sup_{\bxi\in\Xi}\Phi_j(\bxi,\{\bm\lambda_{jm}\}_{m\in[M]}).
\]
Hence the semi-infinite constraint is equivalent to the existence of multipliers $\bm\lambda_{jm}\in\RR_+^J$ with $\mathbf e^\top\bm\lambda_{jm}=\nu_m$ such that, for every $j\in[J]$,
\[
\sup_{\bxi\in\Xi}\left\{\left(\bm a_j(\x)-\frac{1}{M}\sum_{m\in[M],k\in[J]}\lambda_{jm}^k\bm a_k(\z_m)\right)^\top\bxi-\beta\|\bxi\|^2\right\}
+b_j(\x)-\frac{1}{M}\sum_{m\in[M],k\in[J]}\lambda_{jm}^k b_k(\z_m)\le \alpha.
\]

Let $\bm c_j\coloneqq \bm a_j(\x)-\frac{1}{M}\sum_{m\in[M],k\in[J]}\lambda_{jm}^k\bm a_k(\z_m)$. Using Fenchel duality with the support function of $\Xi$ and the infimal convolution identity for Fenchel conjugates \citep[Theorem~16.4]{Rockafellar70},
\[
\sup_{\bxi\in\Xi}\bigl\{\bm c_j^\top\bxi-\beta\|\bxi\|^2\bigr\}
=\inf_{\bm\theta_j\in\RR^{D_\bxi}}\left\{\sigma_\Xi(\bm\theta_j)+\sup_{\bxi\in\RR^{D_\bxi}}\bigl[(\bm c_j-\bm\theta_j)^\top\bxi-\beta\|\bxi\|^2\bigr]\right\}.
\]
The remaining quadratic supremum equals $\|\bm c_j-\bm\theta_j\|^2/(4\beta)$ when $\beta>0$, and its epigraph admits the standard second-order-cone representation
\[
\frac{\|\bm c_j-\bm\theta_j\|^2}{4\beta}\le \zeta_j
\iff
\left\|\begin{bmatrix}\bm c_j-\bm\theta_j\\ \zeta_j-\beta\end{bmatrix}\right\|\le \zeta_j+\beta,\qquad \zeta_j\ge 0.
\]
When $\beta=0$, the semi-infinite constraint is satisfied if and only if
\[
\sigma_\Xi(\bm c_j)+b_j(\x)-\frac{1}{M}\sum_{m\in[M],k\in[J]}\lambda_{jm}^k b_k(\z_m)\le \alpha
\qquad \forall j\in[J].
\]
Moreover, when $\beta=0$, the second-order conic constraint above reduces to
\[
\left\|\begin{bmatrix}\bm c_j-\bm\theta_j\\ \zeta_j\end{bmatrix}\right\|\le \zeta_j,
\]
which implies $\bm\theta_j=\bm c_j$. Substituting this identity into the scalar inequality of the conic reformulation gives
\[
\zeta_j+\sigma_\Xi(\bm c_j)+b_j(\x)-\frac{1}{M}\sum_{m\in[M],k\in[J]}\lambda_{jm}^k b_k(\z_m)\le \alpha.
\]
In this scalar inequality, smaller values of $\zeta_j$ weakly enlarge the feasible region. Since the reduced conic constraint imposes no positive lower bound on $\zeta_j$ beyond $\zeta_j\ge 0$, the least restrictive feasible choice is $\zeta_j=0$, which recovers exactly the support-function constraint above. Hence, the conic reformulation in the theorem also remains valid in the case $\beta=0$. Substituting $\bm c_j$ back into the inequality therefore produces exactly the conic reformulation stated in the theorem.

It remains to verify convexity of the displayed reformulation. The objective is affine in the decision variables, and the sign, simplex, and equality constraints are affine. Since $\bm a_j(\x)$ and $b_j(\x)$ are affine in $\x$, the scalar inequality
\[
\zeta_j+\sigma_\Xi(\bm\theta_j)+b_j(\x)-\frac{1}{M}\sum_{m\in[M],k\in[J]}\lambda_{jm}^k b_k(\z_m)\leq \alpha
\]
has a convex left-hand side, because $\sigma_\Xi$ is convex as the support function of a closed convex set. The remaining nonlinear constraint is the second-order-cone membership
\[
\left(\bm a_j(\x)-\frac{1}{M}\sum_{m\in[M],k\in[J]}\lambda_{jm}^k\bm a_k(\z_m)-\bm\theta_j,\;\zeta_j-\beta,\;\zeta_j+\beta\right)\in\mathcal Q,
\]
where $\mathcal Q\coloneqq\{(\bm u,t):\|\bm u\|\leq t\}$ denotes the standard second-order cone. Its argument is affine in the decision variables, so its inverse image is convex. Together with the convexity of the constraint $\x\in\X$, this shows that the finite reformulation is a convex conic program. The final claims follow because the support function of a second-order-cone representable set is second-order-cone representable.\qed
\end{proof}

\section{Proofs of Section~\ref{sec:unbounded_support}}
\label{appendix:proofs_sec4}

\subsection{Theoretical Guarantees}

\begin{proof}{Proof of Proposition~\ref{thm:dual_exact_subweibull}.}
Fix $\x\in\X$. The inner maximization problem associated with \eqref{eq:ambiguity_set_subweibull} is again a generalized moment problem, now with the additional constraint
\[
\QQ(\|\txi\|\le t)\ge 1-2\exp\left(-\left(t/K_{\mathrm{sw}}\right)^\vartheta\right).
\]
Introducing a multiplier $\kappa\in\RR_+$ for this constraint, together with the multipliers $\nu\in\M_+(\X)$, $\beta\in\RR_+$, and $\alpha\in\RR$ used in the proof of Theorem~\ref{thm:dual_exact}, yields the dual problem
\[
\begin{aligned}
\min\;\;&\alpha+\beta\Omega-\kappa\left(1-2\exp\left(-\left(t/K_{\mathrm{sw}}\right)^\vartheta\right)\right)+\int_\X\left(\EE_\PPhat[\ell(\z,\txi)]+\epsilon\right)\nu(\mathrm d\z)\\
\st\;\;&\ell(\x,\bxi)+\kappa\I_{\{\|\bxi\|\le t\}}\le \alpha+\beta\|\bxi\|^2+\int_\X \ell(\z,\bxi)\nu(\mathrm d\z)\qquad\forall \bxi\in\Xi.
\end{aligned}
\]
The strict feasibility argument from Theorem~\ref{thm:dual_exact} still applies, so strong duality and dual attainment hold. Minimizing over $\x\in\X$ gives \eqref{eq:dual_exact_subweibull} and yields the existence of an optimal tuple.\qed
\end{proof}

\begin{proposition}\label{prop:hinge_SAA_guarantee_subweibull}
With probability at least $1-\tau$, the following bound holds for all $\QQ\in\mathscr P(\Xi)$ satisfying $\EE_\QQ[\|\txi\|^2]\leq \Omega$ and $\QQ(\|\txi\| \leq t ) \geq 1-2\exp\left(-\left(t/K_{\mathrm{sw}}\right)^\vartheta\right)$:
\begin{equation*}
\begin{aligned}
&\EE_\ZZ\left[\left[\EE_\QQ[\ell(\tz,\txi)]-\EE_\PPhat[\ell(\tz,\txi)]-\eps\right]_+\right]\\
&\leq  \frac{1}{M}\sum_{m\in[M]}\left[\EE_\QQ[\ell(\z_m,\txi)]-\EE_\PPhat[\ell(\z_m,\txi)]-\eps\right]_++\mathcal O\left( R(t)\sqrt{\frac{J\log J(\log M)^3}{M}}\right)\\
&\qquad + 2(c_1 + c_2\sqrt{\Omega})\sqrt{\frac{1}{2M}\log\left(\frac{1}{\tau}\right)} + 4c_1\exp\left(-\left(t/K_{\mathrm{sw}}\right)^\vartheta\right) + 2c_2 \sqrt{ 2\Omega \exp\left(-\left(t/K_{\mathrm{sw}}\right)^\vartheta\right)}.
\end{aligned}
\end{equation*}
Here $R(t)$ is defined as in Theorem~\ref{thm:containment_unbounded_Xi}.
\end{proposition}
\begin{proof}{Proof of Proposition~\ref{prop:hinge_SAA_guarantee_subweibull}.}
We define the ambiguity set
\begin{equation*}
\mP_{\rm sw}\coloneqq \left\{\QQ\in\mathscr P(\Xi):\begin{array}{l}
\EE_\QQ[\|\txi\|^2]\leq \Omega\\
\QQ(\|\txi\| \leq t ) \geq   1-2\exp\left(-\left(t/K_{\mathrm{sw}}\right)^\vartheta\right)
\end{array}\right\},
\end{equation*}
which contains both $\mP'_\epsilon$ and $\mP'^M_\epsilon$. Let $h_\QQ(\z)\coloneqq\left[\EE_\QQ[\ell(\z,\txi)]-\EE_\PPhat[\ell(\z,\txi)]-\eps\right]_+$ and $\HH\coloneqq \{h_\QQ:\QQ\in\mP_{\rm sw}\}$.
We now repeat the steps of the proof of Proposition~\ref{prop:hinge_SAA_guarantee}. Assumption~\ref{as:linear_growth} implies that $\bar h\coloneqq 2(c_1+c_2\sqrt{\Omega})$ satisfies $0\leq h(\z)\leq \bar h$ for all $h\in\HH$ and $\z\in\X$. 

By McDiarmid's inequality~\citep[Appendix D]{mohri2018foundations}, we have
\begin{equation*}
\sup_{h\in\HH}\left(\EE_\ZZ[h(\tz)]- \frac{1}{M}\sum_{m\in[M]} h(\tz_m)\right)\leq \EE\left[\sup_{h\in\HH}\left(\EE_\ZZ[h(\tz)]- \frac{1}{M}\sum_{m\in[M]} h(\tz_m)\right)\right] + \bar h \sqrt{\frac{1}{2M}\log\left(\frac{1}{\tau}\right)}
\end{equation*}
with probability at least $1-\tau$. The expectation term can be bounded by twice the Rademacher complexity of $\HH$. Let $F_M(\bxi)\coloneqq M^{-1}\sum_{m\in[M]}\ts_m\ell(\tz_m,\bxi)$. Then
\begin{equation*}
\begin{aligned}
\mathcal R_M(\HH)
&\leq\EE\left[\sup_{\QQ\in\mP_{\rm sw}}\EE_\QQ[F_M(\txi)]\right]\\
&=\EE\left[\sup_{\QQ\in\mP_{\rm sw}}\left\{\EE_\QQ\left[F_M(\txi)\I_{\{\|\txi\|\leq t\}}\right]+\EE_\QQ\left[F_M(\txi)\I_{\{\|\txi\|> t\}}\right]\right\}\right]\\
&\leq \EE\left[\sup_{\QQ\in\mP_{\rm sw}}\EE_\QQ\left[F_M(\txi)\I_{\{\|\txi\|\leq t\}}\right]\right] + \EE\left[\sup_{\QQ\in\mP_{\rm sw}}\EE_\QQ\left[F_M(\txi)\I_{\{\|\txi\|> t\}}\right]\right].
\end{aligned}
\end{equation*}
For the first expectation, since the indicator restricts the integrand to the ball $\{\|\bxi\|\le t\}$,
\[
\sup_{\QQ\in\mP_{\rm sw}}\EE_\QQ\left[F_M(\txi)\I_{\{\|\txi\|\le t\}}\right]
\le \sup_{\|\bxi\|\le t}|F_M(\bxi)|.
\]
By symmetry of the Rademacher signs, the expectation of the right-hand side is bounded, up to an absolute constant, by the localized version of the complexity bound in \eqref{eq:expected_wc_cost}:
\begin{equation*}
\begin{aligned}
 \EE\left[\sup_{\QQ\in\mP_{\rm sw}}\EE_\QQ\left[F_M(\txi)\I_{\{\|\txi\|\leq {t}\}}\right]\right]
&\leq \EE\left[\sup_{\bxi\in\RR^{ D_\bxi}:\|\bxi\|\leq t}\left|\frac{1}{M}\sum_{m\in[M]} \ts_m\ell(\tz_m,\bxi)\right|\right]\\
&\leq 2\EE\left[\sup_{\bxi\in\RR^{ D_\bxi}:\|\bxi\|\leq t}\frac{1}{M}\sum_{m\in[M]} \ts_m\ell(\tz_m,\bxi)\right]\\
&\leq \mathcal O\left( R(t)\sqrt{\frac{J\log J(\log M)^3}{M}}\right).
\end{aligned}
\end{equation*}
For the second expectation, Assumption~\ref{as:linear_growth}, Cauchy--Schwarz, and the defining constraints of $\mP_{\rm sw}$ yield
\begin{equation*}
\begin{aligned}
& \sup_{\QQ\in\mP_{\rm sw}}\EE_\QQ\left[\left|\frac{1}{M}\sum_{m\in[M]} \ts_m\ell(\tz_m,\txi)\right|\I_{\{\|\txi\|> t\}}\right]\\
&\le c_1\,\QQ(\|\txi\|>t)+c_2\sqrt{\EE_\QQ[\|\txi\|^2]\;\QQ(\|\txi\|>t)}\\
&\le 2c_1\exp\left(-\left(t/K_{\mathrm{sw}}\right)^\vartheta\right) + c_2 \sqrt{2\Omega\exp\left(-\left(t/K_{\mathrm{sw}}\right)^\vartheta\right)},
\end{aligned}
\end{equation*}
where the last step uses the tail-mass constraint $\QQ(\|\txi\|>t)\le 2\exp\left(-\left(t/K_{\mathrm{sw}}\right)^\vartheta\right)$. The factor $2$ in the tail correction of the statement comes from the symmetrization step. Combining these bounds with the same symmetrization argument as in Proposition~\ref{prop:hinge_SAA_guarantee} yields the claim.\qed
\end{proof}

\begin{proof}{Proof of Theorem~\ref{thm:containment_unbounded_Xi}.}
On the high-probability event of Proposition~\ref{prop:hinge_SAA_guarantee_subweibull} applied with $t=t_M$, fix any $\QQ\in\mP'^M_\epsilon$. By definition of \eqref{eq:sample_ambiguity_set_subweibull}, the sample-average hinge term is nonpositive, $\EE_\QQ[\|\txi\|^2]\le \Omega$, and $\QQ(\|\txi\|\le t_M)\ge 1-2/M$. Proposition~\ref{prop:hinge_SAA_guarantee_subweibull} therefore implies
\[
\EE_\ZZ\left[\left[\EE_\QQ[\ell(\tz,\txi)]-\EE_\PPhat[\ell(\tz,\txi)]-\eps\right]_+\right]\le \eta,
\]
where $\eta=\eta_M$ is the bound stated in Theorem~\ref{thm:containment_unbounded_Xi}. Thus $\QQ\in\mP'_{\epsilon,M}(\eta_M)$. Since $\QQ$ was arbitrary, we obtain $\mP'^M_\epsilon\subseteq\mP'_{\epsilon,M}(\eta_M)$. The inclusion $\mP'_{\epsilon,M}\subseteq\mP'^M_\epsilon$ is immediate from \eqref{eq:sample_ambiguity_set_subweibull}, because the empirical average of nonnegative terms is zero whenever the expected hinge constraint is satisfied exactly. Finally, $R(t_M)=\mathcal O((\log M)^{1/\vartheta})$, and hence $\eta_M\to0$.\qed
\end{proof}

\begin{lemma}\label{lem:moving_threshold_objective_convergence}
Let $\mP'_{\epsilon,M}$ denote \eqref{eq:ambiguity_set_subweibull} with $t=t_M$, and let $\mP'_{\epsilon,M}(\eta)$ denote \eqref{eq:eta_ambiguity_set_subweibull}. Fix $\epsilon>0$. For any sequence $\eta_M\ge0$ with $\eta_M\to0$ as $M\to\infty$, define
\begin{equation*}
\begin{aligned}
f_0(\x)\coloneqq \sup_{\QQ\in\mP_\epsilon}\EE_\QQ[\ell(\x,\txi)],\qquad f_M(\x)\coloneqq \sup_{\QQ\in\mP'_{\epsilon,M}}\EE_\QQ[\ell(\x,\txi)],\quad g_M(\x)&\coloneqq \sup_{\QQ\in\mP'_{\epsilon,M}(\eta_M)}\EE_\QQ[\ell(\x,\txi)].
\end{aligned}
\end{equation*}
Then
\begin{equation*}
\sup_{\x\in\X}|f_M(\x)-f_0(\x)|\to0,\qquad
\sup_{\x\in\X}|g_M(\x)-f_0(\x)|\to0\qquad\textup{as }M\to\infty.
\end{equation*}
\end{lemma}
\begin{proof}{Proof of Lemma~\ref{lem:moving_threshold_objective_convergence}.}
For a probability measure $\QQ$, define
\begin{equation*}
    H(\QQ)\coloneqq \EE_\ZZ\left[\left[\EE_\QQ[\ell(\tz,\txi)]-\EE_\PPhat[\ell(\tz,\txi)]-\eps\right]_+\right].
\end{equation*}
We first show that $f_M(\x)\rightarrow f_0(\x)$ and $g_M(\x)\rightarrow f_0(\x)$ for each fixed $\x\in \X$.

Since $\mP'_{\epsilon,M}\subseteq\mP_\epsilon$, we have $f_M(\x)\le f_0(\x)$. To prove the reverse limit inequality, fix $\delta>0$ and choose $\QQ^\delta\in\mP_\epsilon$ such that
\begin{equation*}
    \EE_{\QQ^\delta}[\ell(\x,\txi)]\ge f_0(\x)-\delta.
\end{equation*}
Let $\Pi_M(\bxi)=\bxi$ if $\|\bxi\|\le t_M$ and $\Pi_M(\bxi)=t_M\bxi/\|\bxi\|$ otherwise, and let $\widetilde\QQ_M=\QQ^\delta\circ\Pi_M^{-1}$ be the pushforward of $\QQ^\delta$ under $\Pi_M$; equivalently, if $\txi\sim\QQ^\delta$, then $\Pi_M(\txi)\sim\widetilde\QQ_M$. Then $\widetilde\QQ_M$ is supported on $\{\bxi\in\Xi:\|\bxi\|\le t_M\}$ and satisfies the second-moment constraint. Moreover, by Assumption~\ref{as:linear_growth}, as $M\to\infty$,
\begin{equation*}
\begin{aligned}
\Delta_M
&\coloneqq \sup_{\z\in\X}\left|\EE_{\widetilde\QQ_M}[\ell(\z,\txi)]-\EE_{\QQ^\delta}[\ell(\z,\txi)]\right|\\
&= \sup_{\z\in\X}\left|\EE_{\QQ^\delta}\left[\ell(\z,\Pi_M(\txi))-\ell(\z,\txi)\right]\right|\\
&\le \EE_{\QQ^\delta}\left[(2c_1+2c_2\|\txi\|)\I_{\{\|\txi\|>t_M\}}\right]\to0.
\end{aligned}
\end{equation*}
Since $\epsilon>0$, define
\begin{equation}
\label{eq:lambda_QM_def}
\lambda_M\coloneqq \frac{\Delta_M}{\epsilon+\Delta_M},
\qquad
\QQ_M\coloneqq (1-\lambda_M)\widetilde\QQ_M+\lambda_M\PPhat .
\end{equation}
Because $\PPhat$ is the fixed empirical reference distribution, it has finite support. Hence, for all sufficiently large $M$, $\PPhat$ is supported on $\{\bxi\in\Xi:\|\bxi\|\le t_M\}$, and therefore $\QQ_M(\{\bxi\in\Xi:\|\bxi\|\le t_M\})=1$. The second-moment constraint also holds, since
\[
\EE_{\QQ_M}[\|\txi\|^2]
=(1-\lambda_M)\EE_{\widetilde\QQ_M}[\|\txi\|^2]+\lambda_M\EE_{\PPhat}[\|\txi\|^2]
\le \Omega,
\]
where $\widetilde\QQ_M$ satisfies the second-moment constraint by construction and $\PPhat$ satisfies it by Assumption~\ref{as:second_moment}. Since $\QQ^\delta\in\mP_\epsilon$, feasibility gives $H(\QQ^\delta)\le0$. Since $H(\QQ^\delta)$ is the expectation of a nonnegative hinge term, $H(\QQ^\delta)\ge0$, and hence $H(\QQ^\delta)=0$. Therefore, for $\ZZ$-almost every realization $\z\in\X$,
\[
\begin{aligned}
\EE_{\QQ_M}[\ell(\z,\txi)]-\EE_\PPhat[\ell(\z,\txi)]-\epsilon
&=(1-\lambda_M)\left(\EE_{\widetilde\QQ_M}[\ell(\z,\txi)]-\EE_\PPhat[\ell(\z,\txi)]\right)-\epsilon\\
&\le (1-\lambda_M)\left(\EE_{\QQ^\delta}[\ell(\z,\txi)]-\EE_\PPhat[\ell(\z,\txi)]+\Delta_M\right)-\epsilon\\
&\le (1-\lambda_M)(\epsilon+\Delta_M)-\epsilon\\
&=(1-\lambda_M)\Delta_M-\lambda_M\epsilon=0.
\end{aligned}
\]
Here, the first equality uses the definition of $\QQ_M$ in \eqref{eq:lambda_QM_def}; the first inequality uses the definition of $\Delta_M$; the second inequality uses $H(\QQ^\delta)=0$; and the last equality follows from the definition of $\lambda_M$ in \eqref{eq:lambda_QM_def}.

Thus $\QQ_M\in\mP'_{\epsilon,M}$ for all sufficiently large $M$. Because $\lambda_M\to0$ and $\Delta_M\to0$,
\[
\liminf_{M\to\infty} f_M(\x)\ge \liminf_{M\to\infty}\EE_{\QQ_M}[\ell(\x,\txi)]
=\EE_{\QQ^\delta}[\ell(\x,\txi)]\ge f_0(\x)-\delta.
\]
Letting $\delta\downarrow0$ gives $f_M(\x)\to f_0(\x)$.

Next, consider $g_M$. The proof follows the same compactness argument as Lemma~\ref{lem:convergence_relaxed_optimal_value}. The only difference is that the sets $\mP'_{\epsilon,M}(\eta_M)$ also impose the moving tail constraint, but this constraint does not need to be inherited by the limit because the limiting set $\mP_\epsilon$ has no tail constraint.

The lower bound $\liminf_M g_M(\x)\ge f_0(\x)$ follows from $g_M(\x)\ge f_M(\x)$ and the convergence of $f_M(\x)$. For the upper bound, take any subsequence and choose $\QQ_M\in\mP'_{\epsilon,M}(\eta_M)$ such that $g_M(\x)-\frac1M\le \EE_{\QQ_M}[\ell(\x,\txi)]$. The common second-moment constraint gives tightness, so along a further subsequence $\QQ_M\Rightarrow\QQ^\star$. By uniform integrability and Assumption~\ref{as:linear_growth}, the loss expectations converge for every fixed $\z\in\X$:
$\EE_{\QQ_M}[\ell(\z,\txi)]\to \EE_{\QQ^\star}[\ell(\z,\txi)]$. The hinge terms are uniformly bounded by Assumptions~\ref{as:linear_growth} and~\ref{as:second_moment}; hence dominated convergence gives $H(\QQ_M)\to H(\QQ^\star)$. Since $H(\QQ_M)\le\eta_M$ and $\eta_M\to0$, we have $H(\QQ^\star)=0$. Lower semicontinuity preserves the second-moment constraint, and hence $\QQ^\star\in\mP_\epsilon$. Therefore, $\limsup_{M\to\infty}g_M(\x)\le \EE_{\QQ^\star}[\ell(\x,\txi)]\le f_0(\x)$. Together with the lower bound, this proves $g_M(\x)\to f_0(\x)$.

The convergence is uniform over $\X$. Indeed, the Lipschitz estimate in Lemma~\ref{lem:Lipschitz_condition} holds uniformly over all measures satisfying the second-moment constraint, and taking suprema over such measures preserves the same Lipschitz constant for $f_0$, $f_M$, and $g_M$. Since $\X$ is compact, a finite-net argument upgrades the pointwise convergence above to
\[
\sup_{\x\in\X}|f_M(\x)-f_0(\x)|\to0,\qquad
\sup_{\x\in\X}|g_M(\x)-f_0(\x)|\to0.
\]
Thus, the claim follows. 
\qed
\end{proof}

\begin{proof}{Proof of Theorem~\ref{thm:unbounded_value_solution_convergence}.}
Let $f_0$, $f_M$, and $g_M$ be as in Lemma~\ref{lem:moving_threshold_objective_convergence}, with $\eta_M$ chosen as in Theorem~\ref{thm:containment_unbounded_Xi}. Let
\[
h_M(\x)\coloneqq \sup_{\QQ\in\mP'^M_\epsilon}\EE_\QQ[\ell(\x,\txi)]
\]
be the sampled objective. The value convergence follows from the same sandwich argument used in Theorem~\ref{thm:approximation_convergence}, with Lemma~\ref{lem:moving_threshold_objective_convergence} replacing Lemma~\ref{lem:convergence_relaxed_optimal_value}; the solution convergence then follows from the same separation argument. Indeed, on the high-probability event of Theorem~\ref{thm:containment_unbounded_Xi}, we have
\[
f_M(\x)\le h_M(\x)\le g_M(\x)\qquad\forall \x\in\X.
\]
To obtain convergence in probability, fix $\rho>0$ and $\upsilon>0$, and apply Theorem~\ref{thm:containment_unbounded_Xi} with confidence parameter $\upsilon$. The corresponding sequence $\eta_M$ still satisfies $\eta_M\to0$. Hence Lemma~\ref{lem:moving_threshold_objective_convergence} implies that, for all sufficiently large $M$,
\[
\sup_{\x\in\X}|f_M(\x)-f_0(\x)|\le \rho
\quad\textup{and}\quad
\sup_{\x\in\X}|g_M(\x)-f_0(\x)|\le \rho.
\]
On the containment event, the sandwich then gives $\sup_{\x\in\X}|h_M(\x)-f_0(\x)|\le\rho$. Therefore,
\[
\Prob\left(\sup_{\x\in\X}|h_M(\x)-f_0(\x)|>\rho\right)\le \upsilon
\]
for all sufficiently large $M$. Since $\upsilon>0$ is arbitrary,
\[
\sup_{\x\in\X}|h_M(\x)-f_0(\x)|\overset{p}{\to}0.
\]
Consequently,
\[
|\hat v'_M-\hat v|
\le \sup_{\x\in\X}|h_M(\x)-f_0(\x)|\overset{p}{\to}0.
\]

It remains to prove convergence of minimizers. For $\delta>0$, define $\mathcal A_\delta\coloneqq\{\x\in\X:\textup{dist}(\x,\mathcal S)\ge\delta\}$.
If $\mathcal A_\delta$ is empty, the claim is trivial. Otherwise, the function $f_0$ is continuous and $\X$ is compact, so $\rho_\delta\coloneqq \inf_{\x\in\mathcal A_\delta} f_0(\x)-\hat v>0$.
Whenever $\sup_{\x\in\X}|h_M(\x)-f_0(\x)|<\rho_\delta/3$, no minimizer of $h_M$ can belong to $\mathcal A_\delta$. Therefore
\[
\Prob\bigl(\textup{dist}(\hat\x'_M,\mathcal S)\ge\delta\bigr)
\le \Prob\left(\sup_{\x\in\X}|h_M(\x)-f_0(\x)|\ge \rho_\delta/3\right)\to0.
\]
This proves $\textup{dist}(\hat\x'_M,\mathcal S)\overset{p}{\to}0$.

Lastly, we prove the suboptimality bound. Let $\bar \x_M$ be the exact optimizer from the statement. Since $\hat\x'_M$ minimizes the sampled objective $h_M$ and $\mP'_{\epsilon,M}\subseteq\mP'^M_\epsilon$, we have
\[
\begin{aligned}
\sup_{\QQ\in\mP'_{\epsilon,M}}\EE_\QQ[\ell(\hat\x'_M,\txi)]-\sup_{\QQ\in\mP'_{\epsilon,M}}\EE_\QQ[\ell(\bar \x_M,\txi)]
&\le \sup_{\QQ\in\mP'^M_\epsilon}\EE_\QQ[\ell(\hat\x'_M,\txi)]-\sup_{\QQ\in\mP'_{\epsilon,M}}\EE_\QQ[\ell(\bar \x_M,\txi)]\\
&\le \sup_{\QQ\in\mP'^M_\epsilon}\EE_\QQ[\ell(\bar \x_M,\txi)]-\sup_{\QQ\in\mP'_{\epsilon,M}}\EE_\QQ[\ell(\bar \x_M,\txi)].
\end{aligned}
\]
Next, by Theorem~\ref{thm:containment_unbounded_Xi}, setting $\eta=\eta_M$, we obtain the high-probability bound
\[
\sup_{\QQ\in\mP'^M_\epsilon}\EE_\QQ[\ell(\bar \x_M,\txi)]-\sup_{\QQ\in\mP'_{\epsilon,M}}\EE_\QQ[\ell(\bar \x_M,\txi)]
\le
\sup_{\QQ\in\mP'_{\epsilon,M}(\eta_M)}\EE_\QQ[\ell(\bar \x_M,\txi)]-\sup_{\QQ\in\mP'_{\epsilon,M}}\EE_\QQ[\ell(\bar \x_M,\txi)]
\]
with probability at least $1-\tau$. By the same moment-duality argument as in Proposition~\ref{thm:dual_exact_subweibull}, the relaxed unbounded-support problem admits the dual obtained by augmenting \eqref{eq:dual_exact_subweibull} with the slack multiplier $\gamma\eta_M$ and the domination constraint $\gamma\ZZ-\nu\in\M_+(\X)$. Thus, the associated exact dual optimizer $(\bar \x_M,\alpha_M^\star,\beta_M^\star,\kappa_M^\star,\nu_M^\star)$ from the statement yields a feasible, though generally suboptimal, relaxed-dual solution $(\bar \x_M,\alpha_M^\star,\beta_M^\star,\kappa_M^\star,\gamma^\star_M,\nu_M^\star)$. We can therefore further upper bound the preceding display by
\[
\begin{aligned}
&\sup_{\QQ\in\mP'_{\epsilon,M}(\eta_M)}\EE_\QQ[\ell(\bar \x_M,\txi)]-\sup_{\QQ\in\mP'_{\epsilon,M}}\EE_\QQ[\ell(\bar \x_M,\txi)]\\
\le\;& \alpha_M^\star+\beta_M^\star\Omega-\kappa_M^\star(1-2/M)+\int_\X \left(\EE_\PPhat[\ell(\z,\txi)]+\epsilon\right)\nu_M^\star(\mathrm d\z)+\gamma^\star_M\eta_M\\
&\qquad-\left(\alpha_M^\star+\beta_M^\star\Omega-\kappa_M^\star(1-2/M)+\int_\X \left(\EE_\PPhat[\ell(\z,\txi)]+\epsilon\right)\nu_M^\star(\mathrm d\z)\right)\\
=\;&\gamma^\star_M\eta_M.
\end{aligned}
\]
Substituting the explicit expression for $\eta_M$ proves the theorem.
\qed
\end{proof}

\begin{theorem}\label{thm:tuned_unbounded_original}
Fix $\epsilon>0$ and set $t_M^\dagger\coloneqq M^{1/4}/(\log M)^{3/4}$. Let $\hat\x_M''$  be an optimal solution of the sampled problem \eqref{eq:dual_SAA_subweibull} with $t=t_M^\dagger$. Let $\x_M^\dagger$ be an exact optimizer of the threshold-$t_M^\dagger$ problem, and assume there exists an optimal dual measure $\nu_M^\dagger$ in Proposition~\ref{thm:dual_exact_subweibull} such that $\gamma\ZZ-\nu_M^\dagger\in\M_+(\X)$ for some $\gamma\in\RR_+$. Define $\gamma_M^\dagger\coloneqq\inf\{\gamma\in\RR_+:\gamma\ZZ-\nu_M^\dagger\in\M_+(\X)\}$. Then, the following suboptimality bound holds with probability at least $1-\tau$:
\[
\begin{aligned}
\sup_{\QQ\in\mP_\epsilon}\EE_\QQ[\ell(\hat\x_M'',\txi)]
&\le
\min_{\x\in\X}\sup_{\QQ\in\mP_\epsilon}\EE_\QQ[\ell(\x,\txi)]\\
&\quad+\gamma_M^\dagger\Bigg[
\mathcal O\left(\sqrt{J\log J}\frac{(\log M)^{3/4}}{M^{1/4}}\right)\\
&\qquad\qquad
+2(c_1+c_2\sqrt{2\Omega})\left(\sqrt{\frac{1}{2M}\log\left(\frac1\tau\right)}+\exp\left\{-\frac{M^{\vartheta/4}}{2K_{\mathrm{sw}}^\vartheta(\log M)^{3\vartheta/4}}\right\}\right)\Bigg]\\
&\quad+\mathcal O\left(\frac{(\log M)^{3/4}}{M^{1/4}}\right).
\end{aligned}
\]
\end{theorem}

\begin{proof}{Proof of Theorem~\ref{thm:tuned_unbounded_original}.}
Set
\[
\begin{aligned}
\Delta_M^\dagger\coloneqq\;&
\mathcal O\left(\sqrt{J\log J}\frac{(\log M)^{3/4}}{M^{1/4}}\right)
+2(c_1+c_2\sqrt{2\Omega})\left(\sqrt{\frac{1}{2M}\log\left(\frac1\tau\right)}
+\exp\left\{-\frac{M^{\vartheta/4}}{2K_{\mathrm{sw}}^\vartheta(\log M)^{3\vartheta/4}}\right\}\right).
\end{aligned}
\]
Applying Proposition~\ref{prop:hinge_SAA_guarantee_subweibull} with $t = t_M^\dagger$, we observe that, for all sufficiently large $M$, the two exponential tail terms are bounded above by
\[
2(c_1+c_2\sqrt{2\Omega})
\exp\left\{-\frac{M^{\vartheta/4}}{2K_{\mathrm{sw}}^\vartheta(\log M)^{3\vartheta/4}}\right\}.
\]
Repeating the argument in Theorem~\ref{thm:containment_unbounded_Xi}, we obtain, with probability at least $1-\tau$,
\[
\mP_{\epsilon,t_M^\dagger}^{\prime M}\subseteq \mP'_{\epsilon,t_M^\dagger}(\Delta_M^\dagger).
\]
The suboptimality argument in the proof of Theorem~\ref{thm:unbounded_value_solution_convergence}, applied with $t_M^\dagger$ in place of $t_M$, gives
\[
\sup_{\QQ\in\mP'_{\epsilon,t_M^\dagger}}\EE_\QQ[\ell(\hat\x_M'',\txi)]
\le
\min_{\x\in\X}\sup_{\QQ\in\mP'_{\epsilon,t_M^\dagger}}\EE_\QQ[\ell(\x,\txi)]
+\gamma_M^\dagger\Delta_M^\dagger.
\]
We now compare the threshold exact objective with the original objective. Fix any $\QQ\in\mP_\epsilon$ and any $t>0$. Let $\Pi_t(\bxi)=\bxi$ if $\|\bxi\|\le t$ and $\Pi_t(\bxi)=t\bxi/\|\bxi\|$ otherwise, and define $\widetilde\QQ_t\coloneqq \QQ\circ\Pi_t^{-1}$. By Assumption~\ref{as:linear_growth} and the second-moment constraint,
\[
\begin{aligned}
\rho(t)&\coloneqq
\sup_{\z\in\X}
\left|
\EE_{\widetilde\QQ_t}[\ell(\z,\txi)]-\EE_\QQ[\ell(\z,\txi)]
\right|\\
&\le
\EE_\QQ\!\left[(2c_1+2c_2\|\txi\|)\I_{\{\|\txi\|>t\}}\right]
\le \frac{2c_1\Omega}{t^2}+\frac{2c_2\Omega}{t}.
\end{aligned}
\]
Set $\lambda_t\coloneqq\rho(t)/(\epsilon+\rho(t))$ and $\QQ_t\coloneqq(1-\lambda_t)\widetilde\QQ_t+\lambda_t\PPhat$. For all sufficiently large $t$, the fixed empirical reference $\PPhat$ is supported on $\{\|\bxi\|\le t\}$, and hence $\QQ_t$ satisfies the threshold probability constraint. The same calculation as in Lemma~\ref{lem:moving_threshold_objective_convergence} shows that $\QQ_t$ also satisfies the loss-discrepancy constraints, so $\QQ_t\in\mP'_{\epsilon,t}$. Moreover, for every $\x\in\X$,
\[
\EE_\QQ[\ell(\x,\txi)]-\EE_{\QQ_t}[\ell(\x,\txi)]\le 2\rho(t)\le \frac{4c_1\Omega}{t^2}+\frac{4c_2\Omega}{t}.
\]
Taking the supremum over $\QQ\in\mP_\epsilon$ gives
\[
\sup_{\QQ\in\mP_\epsilon}\EE_\QQ[\ell(\x,\txi)]
\le
\sup_{\QQ\in\mP'_{\epsilon,t}}\EE_\QQ[\ell(\x,\txi)]
+\frac{4c_1\Omega}{t^2}+\frac{4c_2\Omega}{t}
\qquad\forall \x\in\X.
\]
Applying this inequality at $t=t_M^\dagger$ and $\x=\hat\x_M''$, using $t_M^\dagger=M^{1/4}/(\log M)^{3/4}$, using the threshold suboptimality bound above, and finally using $\mP'_{\epsilon,t_M^\dagger}\subseteq\mP_\epsilon$, we obtain
\[
\begin{aligned}
\sup_{\QQ\in\mP_\epsilon}\EE_\QQ[\ell(\hat\x_M'',\txi)]
&\le
\sup_{\QQ\in\mP'_{\epsilon,t_M^\dagger}}\EE_\QQ[\ell(\hat\x_M'',\txi)]
+\mathcal O\left(\frac{(\log M)^{3/2}}{M^{1/2}}\right)
+\mathcal O\left(\frac{(\log M)^{3/4}}{M^{1/4}}\right)\\
&\le
\min_{\x\in\X}\sup_{\QQ\in\mP'_{\epsilon,t_M^\dagger}}\EE_\QQ[\ell(\x,\txi)]
+\gamma_M^\dagger\Delta_M^\dagger
+\mathcal O\left(\frac{(\log M)^{3/2}}{M^{1/2}}\right)
+\mathcal O\left(\frac{(\log M)^{3/4}}{M^{1/4}}\right)\\
&\le
\min_{\x\in\X}\sup_{\QQ\in\mP_\epsilon}\EE_\QQ[\ell(\x,\txi)]
+\gamma_M^\dagger\Delta_M^\dagger
+\mathcal O\left(\frac{(\log M)^{3/4}}{M^{1/4}}\right).
\end{aligned}
\]
Substituting the definition of $\Delta_M^\dagger$ proves the displayed bound.\qed
\end{proof}

\subsection{Conic Programming Reformulations}

\begin{proof}{Proof of Theorem~\ref{theorem:Conic_reformulation_subweibull}.}
For each sampled point $\z_m$, define
\[
g_m(\bxi)\coloneqq \max_{k\in[J]}\left(\bm a_k(\z_m)^\top\bxi+b_k(\z_m)\right).
\]
The semi-infinite constraint in \eqref{eq:dual_SAA_subweibull} can be equivalently rewritten as
\begin{align}
\ell(\x,\bxi)+\kappa&\leq\alpha+\beta\|\bxi\|^2+\frac{1}{M}\sum_{m\in[M]}\nu_m\ell(\z_m,\bxi) \qquad\forall\bxi\in\RR^{ D_\bxi}:\|\bxi\|\leq t,\label{eq:semiinf_in_ball_app}\\
\ell(\x,\bxi)&\leq\alpha+\beta\|\bxi\|^2+\frac{1}{M}\sum_{m\in[M]}\nu_m\ell(\z_m,\bxi) \qquad\forall\bxi\in\RR^{ D_\bxi}:\|\bxi\|> t.\label{eq:semiinf_out_ball_app}
\end{align}
Since $\kappa\geq 0$, \eqref{eq:semiinf_in_ball_app} implies
\[
\ell(\x,\bxi)\leq\alpha+\beta\|\bxi\|^2+\frac{1}{M}\sum_{m\in[M]}\nu_m\ell(\z_m,\bxi)\qquad\forall\bxi\in\RR^{ D_\bxi}:\|\bxi\|\leq t.
\]
Hence \eqref{eq:semiinf_out_ball_app} can be simplified to
\[
\ell(\x,\bxi)\leq\alpha+\beta\|\bxi\|^2+\frac{1}{M}\sum_{m\in[M]}\nu_m\ell(\z_m,\bxi)\qquad\forall\bxi\in\RR^{ D_\bxi}.
\]
We now reformulate these two semi-infinite systems separately.
Assume first that $\beta>0$; the case $\beta=0$ will be treated separately after deriving the conic reformulations for the global and ball constraints.

For the global constraint, finiteness of the supremum over $\RR^{D_\bxi}$ requires $\bm c_j=\mathbf 0$, and the conic constraint reduces to $\|(\bm c_j,\zeta_j)\|\le\zeta_j$, which imposes no positive lower bound on $\zeta_j$ beyond $\zeta_j\ge 0$. Since $\zeta_j$ enters the scalar inequality for the global constraint with coefficient one, the least restrictive feasible choice is $\zeta_j=0$. For the ball constraint, the conic constraint reduces to $\|(\bm c'_j-\bm\theta_j,\zeta'_j)\|\le\zeta'_j$, which forces $\bm\theta_j=\bm c'_j$ and imposes no positive lower bound on $\zeta'_j$.
For the ball constraint \eqref{eq:semiinf_in_ball_app}, introduce simplex multipliers $\bm\lambda'_{jm}\in\RR_+^J$ with $\mathbf e^\top\bm\lambda'_{jm}=\nu_m$. Then the constraint is equivalent to
\[
\sup_{\|\bxi\|\le t}\left\{\left(\bm a_j(\x)-\frac{1}{M}\sum_{m\in[M],k\in[J]}\lambda_{jm}^{'k}\bm a_k(\z_m)\right)^\top\bxi-\beta\|\bxi\|^2\right\}
+b_j(\x)+\kappa-\frac{1}{M}\sum_{m\in[M],k\in[J]}\lambda_{jm}^{'k} b_k(\z_m)\le \alpha
\]
for every $j\in[J]$. Let
\[
\bm c'_j\coloneqq \bm a_j(\x)-\frac{1}{M}\sum_{m\in[M],k\in[J]}\lambda_{jm}^{'k}\bm a_k(\z_m).
\]
Using Fenchel duality \citep[Theorem~16.4]{Rockafellar70}, with the support function of the Euclidean ball $\{\bxi:\|\bxi\|\le t\}$, whose support function is $t\|\cdot\|$, we obtain
\[
\sup_{\|\bxi\|\le t}\left\{\bm c_j^{' \top}\bxi-\beta\|\bxi\|^2\right\}
=\inf_{\bm\theta_j\in\RR^{D_\bxi}}\left\{t\|\bm\theta_j\|+\sup_{\bxi\in\RR^{D_\bxi}}\bigl[(\bm c'_j-\bm\theta_j)^\top\bxi-\beta\|\bxi\|^2\bigr]\right\}.
\]
Since $\beta>0$, the remaining unconstrained quadratic supremum equals $\|\bm c'_j-\bm\theta_j\|^2/(4\beta)$. Hence the ball constraint is equivalent to the existence of $(\bm\theta_j,\zeta'_j)\in\RR^{D_\bxi}\times\RR_+$ such that
\[
\zeta'_j+t \|\bm\theta_j\|+b_j(\x)+\kappa-\frac{1}{M}\sum_{m\in[M],k\in[J]} \lambda_{jm}^{'k} b_k(\z_m)\leq \alpha
\quad\text{and}\quad
\left\|(\bm c'_j-\bm\theta_j,\,\zeta'_j-\beta)\right\|\leq\zeta'_j+\beta.
\]
The case $\beta=0$ is handled exactly as in Theorem~\ref{theorem:Conic_reformulation}. For the global constraint, finiteness of the supremum over $\RR^{D_\bxi}$ requires $\bm c_j=\mathbf 0$, and the conic constraint reduces to
\[
\left\|\begin{bmatrix}\bm c_j\\\zeta_j\end{bmatrix}\right\|\le \zeta_j,
\]
so the reduced conic constraint imposes no positive lower bound on $\zeta_j$ beyond $\zeta_j\ge 0$. Since $\zeta_j$ enters the scalar inequality for the global constraint with coefficient one, the least restrictive feasible choice is $\zeta_j=0$. For the ball constraint, the conic constraint reduces to
\[
\left\|\begin{bmatrix}\bm c'_j-\bm\theta_j\\\zeta'_j\end{bmatrix}\right\|\le \zeta'_j,
\]
which forces $\bm\theta_j=\bm c'_j$ and imposes no positive lower bound on $\zeta'_j$. Since $\zeta'_j$ enters the scalar inequality for the ball constraint with coefficient one, the least restrictive feasible choice is $\zeta'_j=0$, and the scalar inequality becomes precisely the support-function condition for the Euclidean ball. Collecting the global and ball constraints yields exactly the finite conic reformulation stated in the theorem.

The displayed reformulation is convex. The objective is affine, while the sign, simplex, and equality constraints are affine. Since $\bm a_j(\x)$ and $b_j(\x)$ are affine in $\x$, the scalar global constraint is affine. The scalar ball constraint has the form
\[
\zeta'_j+t\|\bm\theta_j\|+b_j(\x)+\kappa-\frac{1}{M}\sum_{m\in[M],k\in[J]}\lambda_{jm}^{'k}b_k(\z_m)\leq \alpha,
\]
which is convex because $t\geq0$ and the Euclidean norm is convex. The two norm inequalities are standard second-order-cone constraints with affine arguments in the decision variables, and are therefore convex. Together with the convexity of the constraint $\x\in\X$, this shows that the finite reformulation is a convex conic program. The final tractability claim follows as in Theorem~\ref{theorem:Conic_reformulation}, since the Euclidean norm is directly second-order cone representable.\qed
\end{proof}

\section{Tractable approximations for two-stage linear loss}
In this section, we develop a tractable approximation to the two-stage DRO problem via linear decision rules.  
\label{appendix:conic_two_stage}
\begin{theorem}
\label{thm:two_stage_conservative_approx}
The problem \eqref{eq:dual_SAA} under the two-stage loss function in \eqref{eq:two-stage_primal} can be conservatively approximated by the following conic program: 
\begin{equation*} 
\begin{aligned}
\min \;&  \alpha  + \beta\Omega + \frac{1}{M}\sum_{m\in[M]}\nu_m(\EE_\PPhat[\ell(\z_m,\txi)]+ \epsilon)  \\
\st & \x\in\X, \; \alpha\in\RR,\;\beta\in\RR_+,\;\bm\nu\in\RR_+^M\\
& \bm\lambda_{m}\in\RR_+^L\;\;\forall m\in[M],\; \bm\theta\in\RR^{D_\bxi},\zeta\in\RR_+\\
& \bm Y\in\RR^{D_\y\times D_\bxi},\;\bm y_0\in\RR^{D_\y}\\
&\sigma_\Xi\left([\bm T(\x)-\bm W\bm Y]_{l:}^\top\right)\leq [\bm W\y_0-\bm h(\x)]_l\qquad\forall l\in[L]\\
&  \zeta+\sigma_\Xi(\bm\theta)+ \bc^\top\x + \bm q^\top\y_0 - \frac{1}{M}\sum_{m\in[M]}\left(\nu_m\bc^\top\z_m + \bm h(\z_m)^\top\bm\lambda_m\right) \leq \alpha \\
& \left\|\begin{bmatrix}\bm Y^\top\bm q-\frac{1}{M}\sum_{m\in[M]} \bm T(\z_m)^\top\bm\lambda_{m} -\bm\theta\\\zeta-\beta\end{bmatrix}\right\|\leq\zeta+\beta\\
&\bm W^\top\bm\lambda_m=\nu_m\bm q\qquad\forall m\in[M]. 
\end{aligned}
\end{equation*}
\end{theorem}
\begin{proof}{Proof}
Consider the semi-infinite constraint in \eqref{eq:dual_SAA}:
\begin{equation}
\label{eq:twostage_semiinf}
\begin{aligned}
&\sup_{\bxi\in\Xi}\left[\ell(\x,\bxi)- \beta\|\bxi\|^2-\frac{1}{M}\sum_{m\in[M]}\nu_m\ell(\z_m,\bxi)\right]\leq \alpha.
\end{aligned}
\end{equation}
We approximate the recourse decision in the first term by the linear decision rule 
\begin{equation*}
\bm y(\bxi)=\bm Y\bxi+\bm y_0,
\end{equation*}
where $\bm Y\in\RR^{D_\y\times D_\bxi} $ and $\bm y_0\in\RR^{D_\y}$ satisfy 
\begin{equation*}
\label{eq:decision_rule_feasibility}
\bm T(\x)\bxi + \bm h(\x)\leq \bm W(\bm Y\bxi+\bm y_0)\qquad\forall\bxi\in\Xi.
\end{equation*}
This robust linear constraint is equivalent, row by row, to 
\begin{equation*}
\sup_{\bxi\in\Xi}[\bm T(\x)-\bm W\bm Y]_{l:}\bxi\leq [\bm W\y_0-\bm h(\x)]_l\qquad\forall l\in[L],
\end{equation*}
or, equivalently,
\begin{equation*}
\sigma_\Xi\left([\bm T(\x)-\bm W\bm Y]_{l:}^\top\right)\leq [\bm W\y_0-\bm h(\x)]_l\qquad\forall l\in[L].
\end{equation*}
Then, the two-stage loss function in \eqref{eq:two-stage_primal} can be upper bounded by 
\begin{equation*}
\ell(\x,\bxi)\leq \bc^\top\x+ \bm q^\top(\bm Y\bxi+\bm y_0)\qquad\forall\bxi\in\Xi. 
\end{equation*}

Replacing only $\ell(\x,\bxi)$ by this upper bound, while keeping each sampled term $\ell(\z_m,\bxi)$ exact, yields a conservative approximation to \eqref{eq:twostage_semiinf}:
\begin{equation*}
\left.
\begin{aligned}
\max\;\;& \bc^\top\x + \bm q^\top(\bm Y\bxi+\bm y_0)- \beta\|\bxi\|^2-\frac{1}{M} \sum_{m\in[M]}\nu_m (\bc^\top\z_m + \bm q^\top\bm y_m) \\
\st\;&\bxi\in\Xi,\;\bm y_m\in\RR^{D_\y}\quad\forall m\in[M]\\
&\bm T(\z_m)\bxi + \bm h(\z_m)\leq \bm W\bm y_m\quad\forall m\in[M]
\end{aligned}\right\}\leq \alpha. 
\end{equation*}
Next, dualizing the recourse variables $\{\bm y_m\}_{m\in[M]}$ and applying the support-function reformulation to the remaining maximization over $\bxi$ yields the equivalent conic conditions: there exist $\bm\lambda_{m}\in\RR_+^L\;\forall m\in[M]$, $\bm\theta\in\RR^{D_\bxi}$, $\zeta\in\RR_+$, such that 
\begin{equation*}
\begin{aligned}
\;&  \zeta+\sigma_\Xi(\bm\theta)+ \bc^\top\x + \bm q^\top\y_0 - \frac{1}{M}\sum_{m\in[M]}\left(\nu_m\bc^\top\z_m + \bm h(\z_m)^\top\bm\lambda_m\right) \leq \alpha \\
& \left\|\begin{bmatrix}\bm Y^\top\bm q-\frac{1}{M}\sum_{m\in[M]} \bm T(\z_m)^\top\bm\lambda_{m} -\bm\theta\\\zeta-\beta\end{bmatrix}\right\|\leq\zeta+\beta\\
&\bm W^\top\bm\lambda_m=\nu_m\bm q\qquad\forall m\in[M]. 
\end{aligned}
\end{equation*}
This proves the result. \qed
\end{proof}
\begin{remark}[Two-stage approximation alternatives]
Two-stage (distributionally) robust optimization is generically NP-hard, which motivates the approximation in Theorem \ref{thm:two_stage_conservative_approx}. If desired, tighter approximations via piecewise-linear decision rules \citep{georghiou2015generalized,ben2020tractable} or quadratic decision rules \citep{fan2024decision,xu2025improved} can also be developed. One could also, in principle, formulate a convex copositive program for the two-stage problem via the techniques of \cite{hanasusanto2018conic,xu2018copositive}. The resulting conic program admits tractable conservative semidefinite programming approximations that can be solved with standard off-the-shelf solvers.  
\end{remark}

\section{Details of experiments}

\subsection{Numerical illustrations of the Monte Carlo approximation}\label{appendix:mc_numerical_illustration}

We present two experiments that share an identical base design and differ only in the reference distribution of $\txi$. Both serve to isolate the sampling error in the finite-sample dual approximation: we set $\epsilon=0$ and evaluate the sampled-dual solution against the empirical objective $\EE_\PPhat[\ell(\x,\tilde\bxi)]$. The symmetric construction described below yields a known optimizer of this benchmark, avoiding the need for a large-reference Monte Carlo solution. The first experiment (bounded support) corresponds to Remark~\ref{rem:mc_numerical_illustration}; the second (unbounded support) corresponds to the practical recommendation at the end of Section~\ref{sec:unbounded_support}, verifying that the $\widetilde{\mathcal O}(1/\sqrt{M})$ sampling rate is preserved when the reference distribution has heavy tails.

\paragraph{Shared setup.}
For each dimension $d\in\{5,25,100\}$, we take $D_\x=D_\bxi=d$ and set the decision set to the Euclidean unit ball, $\X=\{\bm u\in\RR^d:\|\bm u\|\leq 1\}$. The known optimizer is $\x_\star=0.1\mathbf e$, which belongs to $\X$ for all three dimensions. The empirical reference distribution contains $N=100$ points and is constructed symmetrically: we draw $50$ independent points from the reference distribution and add their negatives, so that the sample mean is exactly zero and $\x_\star$ is an optimizer of the empirical objective for the shifted loss. The loss is a shifted weighted projection max-affine loss,
\[
\ell(\x,\bxi)=\max_{r=1,\ldots,K} c_r\left|\bm q_r^\top((\x-\x_\star)-\bxi)\right|,
\]
with $K=4$. The directions $\bm q_r\in\RR^d$ are generated as independent dense Gaussian vectors normalized to unit length, and the weights $c_r$ are linearly spaced between $0.5$ and $1.5$ and rescaled to average one. The loss has $J=2K=8$ affine pieces with slopes $\{\pm c_r\bm q_r:r=1,\ldots,K\}$, so the conic reformulation in Section~\ref{sec:mc_sampling} applies directly. For each dimension, the experiment uses $50$ independent replications. In each replication, a pool of $500$ auxiliary decisions is drawn uniformly from $\X$ and nested prefixes of this pool are used for $M\in\{2,5,10,20,50,100,250,500\}$. The reported percentage gap is
\[
\frac{
\EE_\PPhat[\ell(\widehat \x_M,\tilde\bxi)]
-
\EE_\PPhat[\ell(\x_\star,\tilde\bxi)]
}{
\EE_\PPhat[\ell(\x_\star,\tilde\bxi)]
}
\times 100\%.
\]
The plotted curves show the mean over the $50$ replications, shaded interquartile bands, and a pooled $C/\sqrt{M}$ reference line fitted across dimensions.

\paragraph{Bounded support (uniform on unit ball).}
The uncertainty $\txi$ is drawn uniformly from the unit ball, $\Xi=\{\bm u\in\RR^d:\|\bm u\|\leq 1\}$. The second-moment constraint is set to the empirical training second moment, $\Omega=N^{-1}\sum_{i=1}^N\|\hat\bxi_i\|^2$. The sampled dual \eqref{eq:dual_SAA} is solved using the compact-support conic reformulation with $\sigma_\Xi(\bm\theta)=\|\bm\theta\|$. As Figure~\ref{fig:mc-suboptimality-gap} shows, the empirical gaps decay at least as fast as the fitted $C/\sqrt{M}$ reference curve (fitted constant $C=26.98$); for larger $M$, the observed decay is even faster, showing that the theoretical rate in Theorem~\ref{thm:approximation_convergence} is conservative in practice.

\paragraph{Unbounded support (Student-$t_5$).}
\phantomsection\label{appendix:mc_numerical_illustration_student_t}%
The uncertainty $\txi$ now follows a scaled Student-$t$ distribution with $\nu=5$ degrees of freedom, giving $\Xi=\RR^d$. Each coordinate is drawn independently from a mean-zero Student-$t_5$ and scaled to coordinate variance $1/(d+2)$, matching the per-coordinate variance of the uniform distribution on the unit ball; under this scaling $\EE\|\txi\|^2=d/(d+2)$, so the second-moment constraint is set to $\Omega=d/(d+2)$. The sampled dual is solved with $\Xi=\RR^d$: the semi-infinite constraint is enforced via the second-moment term $\beta\|\bxi\|^2$, and the support-function term is set to zero in the second-order-cone reformulation. No tail constraint from Section~\ref{sec:unbounded_support} is included, consistent with the practical recommendation there. Figure~\ref{fig:mc-suboptimality-gap-student-t} reports the mean percentage gap and interquartile bands together with a pooled $C/\sqrt{M}$ reference line. The gaps decrease consistently with $M$ and reach roughly the one-to-two percent level at $M=500$ across all dimensions, confirming that the sampled-dual approximation remains effective under heavy-tailed unbounded uncertainty.

\begin{figure}[t]
    \centering
    \includegraphics[width=0.5\textwidth]{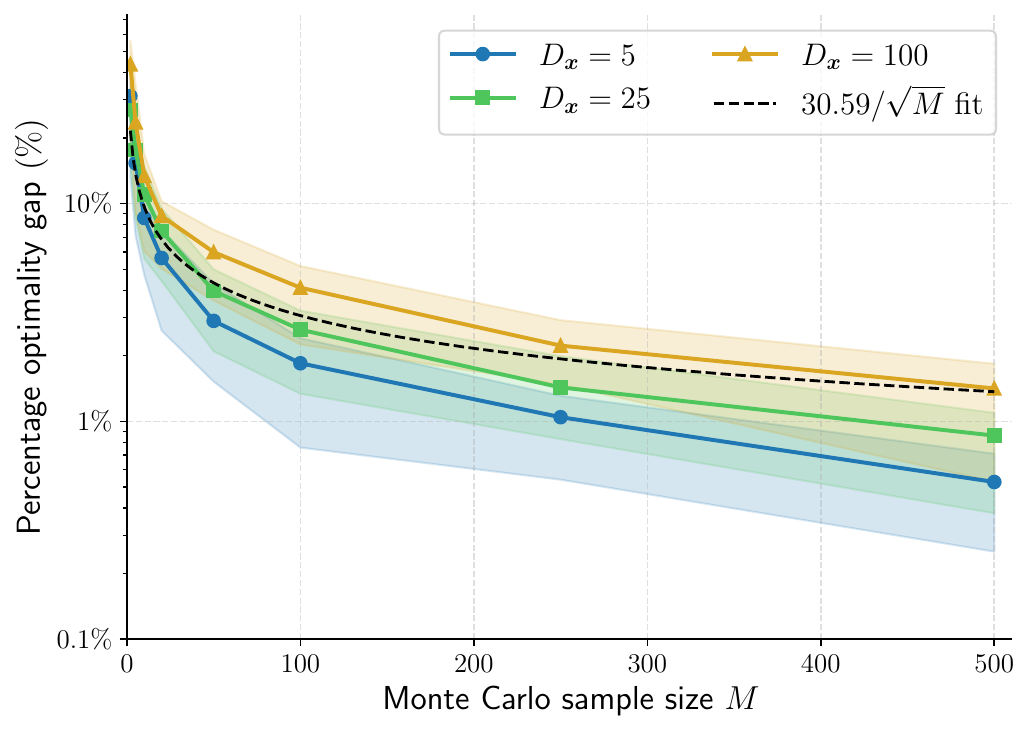}
    \caption{Percentage optimality gap of the sampled-dual solution as a function of the Monte Carlo sample size $M$, with Student-$t_5$ uncertainty ($\nu=5$, unbounded support, variance-matched to the unit-ball setting). The dashed curve is a fitted $C/\sqrt{M}$ reference line. Shaded bands show interquartile ranges over $50$ independent replications.}
    \label{fig:mc-suboptimality-gap-student-t}
\end{figure}

\subsection{Newsvendor}\label{appendix:newsvendor}

The newsvendor figures in Section~5.1 are generated by the final selected-trial plotting blocks in the four three-dimensional notebooks. Earlier exploratory cells use different parameter values, but those preliminary settings are not used for the figure files included in the manuscript.

Figure~\ref{fig:nv_oos_by_n} experiment studies a three-product newsvendor with order vector $\x\in\R_+^3$, demand vector $\txi\in\R_+^3$, and loss
\[
\ell(\x,\bxi)=\sum_{j=1}^3\left[h_j(x_j-\xi_j)_+ + b_j(\xi_j-x_j)_+\right].
\]
Across all four final plotting blocks, the common decision-side parameters are $h=(0.1,0.2,0.3), b=(1,1,1),\quad\rho=0.05,\quad U=(50,50,50)$. The demand coordinates are generated independently with mean vector $\mu=(5,6,6)$ and standard-deviation vector $\sigma=(5,6,8)$. The in-sample data are split into training and validation subsets in an 80/20 ratio. We use $N\in\{10,20,30,50,75,100,150\}$ with $T=50$ independent trials for each sample size, and each trial is evaluated on an independent OOS sample of size $N_{\mathrm{OOS}}=10000$.

For both 2-WDRO and TIPM-DRO, the ambiguity radius is tuned over the common grid $\epsilon \in \{a\times 10^b : a\in\{1,3,5,7,9\},\; b\in\{-3,-2,-1,0,1,2\}\}$. For TIPM-DRO, the expectation over the auxiliary distribution is approximated by Monte Carlo sampling. We use $M=800$ sampled decision points drawn uniformly from $[0,50]^3$, and the sampled set is augmented with the SAA solution from the same trial. The second-moment constraint is set to the empirical training second moment,
$
\Omega=\frac{1}{N_{\mathrm{tr}}}\sum_{i=1}^{N_{\mathrm{tr}}}\|\hat\bxi_i\|^2
$.

To state the dual reformulation, let $\mathcal B_{\mathrm{nv}}\coloneqq \{0,1\}^3$ denote the eight affine pieces of the separable newsvendor loss, and let $\mathcal R_{\mathrm{nv}}\coloneqq \mathcal B_{\mathrm{nv}}\cup\{T\}$ include the additional CVaR epigraph piece. For each $s\in\mathcal B_{\mathrm{nv}}$, define
\[
p^s_j=\begin{cases}
h_j,& s_j=0,\\
-b_j,& s_j=1,
\end{cases}
\qquad
g^s_j=\begin{cases}
-h_j,& s_j=0,\\
b_j,& s_j=1,
\end{cases}
\qquad j=1,2,3.
\]
For the epigraph piece, set $\bm p^T=\bm 0$, $\bm g^T=\bm 0$, and $q_T=1$, while $q_r=0$ for $r\in\mathcal B_{\mathrm{nv}}$. Then
\[
\ell(\x,\bxi)=\max_{s\in\mathcal B_{\mathrm{nv}}}\Bigl\{(\bm p^s)^\top \x+(\bm g^s)^\top \bxi\Bigr\}.
\]
If $\z_1,\ldots,\z_M\in[0,50]^3$ are the sampled decision points and
\[
\widehat L_m\coloneqq \frac{1}{N_{\mathrm{tr}}}\sum_{i=1}^{N_{\mathrm{tr}}}\ell(\z_m,\hat\bxi_i),
\qquad m\in[M],
\]
then the sampled second-moment-constrained TIPM-DRO problem solved by the implementation is
\begin{equation}
\label{eq:newsvendor_ipm_dual_appendix}
\begin{aligned}
\min \;& \left(1-\frac{1}{\rho}\right)t+\frac{1}{\rho}\left(\alpha+\frac{1}{M}\sum_{m\in[M]}\beta_m(\widehat L_m+\epsilon)+\gamma\Omega\right)\\
\st\;& \x\in\R_+^3,\;\x\le U,\; t,\alpha\in\R,\;\bm\beta\in\R_+^M,\;\gamma\in\R_+,\\
& \eta_r\in\R_+,\;\bm u_r\in\R^3,\;\bm u_r\le \mathbf 0,\;\lambda_{rms}\in\R_+ \qquad \forall r\in\mathcal R_{\mathrm{nv}},\forall m\in[M],\forall s\in\mathcal B_{\mathrm{nv}},\\
& \eta_r+(\bm p^r)^\top \x+q_r t-\frac{1}{M}\sum_{m\in[M]}\sum_{s\in\mathcal B_{\mathrm{nv}}}\lambda_{rms}(\bm p^s)^\top \z_m\le \alpha \qquad \forall r\in\mathcal R_{\mathrm{nv}},\\
& \sum_{s\in\mathcal B_{\mathrm{nv}}}\lambda_{rms}=\beta_m \qquad \forall r\in\mathcal R_{\mathrm{nv}},\forall m\in[M],\\
& \left\|\begin{bmatrix}
\bm g^r-\frac{1}{M}\sum_{m\in[M]}\sum_{s\in\mathcal B_{\mathrm{nv}}}\lambda_{rms}\bm g^s-\bm u_r\\
\eta_r-\gamma
\end{bmatrix}\right\|\le \eta_r+\gamma \qquad \forall r\in\mathcal R_{\mathrm{nv}}.
\end{aligned}
\end{equation}
This is exactly the finite conic reformulation implemented in the experiment and used to generate the TIPM-DRO curves in Section~5.1.

For the four demand models, the Gaussian benchmark samples each coordinate from a Normal distribution with the target mean and variance and rejects negative draws. The rescaled $\chi^2$ case uses $\xi_j=s_jY_j$, where $Y_j\sim\chi^2_{\nu_j}$, $\nu_j=2(\mu_j/\sigma_j)^2$, and $s_j=\sigma_j^2/(2\mu_j)$. The Lognormal case uses $\xi_j\sim\mathrm{LogNormal}(\mu_{\log,j},\sigma_{\log,j})$, where $\sigma_{\log,j}^2=\log(1+\sigma_j^2/\mu_j^2)$ and $\mu_{\log,j}=\log(\mu_j)-\tfrac{1}{2}\sigma_{\log,j}^2$. The Pareto case uses $\xi_j=x_{m,j}(1+Y_j)$, where $Y_j\sim\mathrm{Pareto}(\alpha_j)$, $\alpha_j=1+\sqrt{1+(\mu_j/\sigma_j)^2}$, and $x_{m,j}=\mu_j(\alpha_j-1)/\alpha_j$.

\label{appendix:newsvendor_Ksweep}
Figure~\ref{fig:nv_K_sweep} uses an independent set of random seeds from the four-distribution study above.
The demand distribution is Lognormal throughout, with marginal parameters matching those of the
Lognormal panel above wherever applicable.
The parameter vectors for each product count are:

\begin{center}
\begin{tabular}{clllll}
\toprule
$K$ & $\bm{h}$ & $\bm{b}$ & $\bm{\mu}$ & $\bm{\sigma}$ & $\bm{U}$ \\
\midrule
$1$ & $(0.2)$ & $(1.0)$ & $(6.0)$ & $(6.0)$ & $(50)$ \\
$3$ & $(0.1,\,0.2,\,0.3)$ & $(1,\,1,\,1)$ & $(5,\,6,\,6)$ & $(5,\,6,\,8)$ & $(50,\,50,\,50)$ \\
$5$ & $(0.1,\,0.15,\,0.2,\,0.25,\,0.3)$ & $(1,\,1,\,1,\,1,\,1)$ & $(5,\,5,\,6,\,6,\,6)$ & $(4,\,5,\,6,\,7,\,8)$ & $(50,\ldots,50)$ \\
\bottomrule
\end{tabular}
\end{center}

All three settings share $\rho=0.05$, training sizes $N\in\{10,30,50,75,100,150\}$, an 80/20
train--validation split, $N_{\mathrm{OOS}}=10{,}000$, and 50 independent trials. The ambiguity
radius is selected by validation from the grid
$\epsilon\in\{a\times 10^b : a\in\{1,5\},\; b\in\{-4,-3,-2,-1,0,1,2\}\}$.
For TIPM-DRO, the auxiliary distribution is approximated by $M=300$ decision points drawn
uniformly from $[0,U_j]^K$, augmented with the SAA solution from the same trial. The
second-moment parameter is the empirical training second moment,
$\Omega=N_{\mathrm{tr}}^{-1}\sum_{i=1}^{N_{\mathrm{tr}}}\|\hat\bxi_i\|^2$.
The conic reformulation follows the same structure as the $K=3$ case above, with $2^K$
business pieces indexed by $\{0,1\}^K$ and one CVaR epigraph piece.

\subsection{Outlier-Corrupted Regression}\label{appendix:outlier_regression}

The figures in Section~5.2 are based on the following common experimental procedure. Each trial draws $\x^\star$ uniformly from the unit sphere in $\RR^{D_\bxi}$, generates clean covariates $\bm{u}_i\sim\mathcal{N}(\bm{0},\bm{I}_{D_\bxi})$ and responses $v_i=\x^{\star\top}\bm{u}_i$, and evaluates the fitted estimator on an independent clean test sample of size $N_{\mathrm{OOS}}$. It then corrupts $\lfloor \omega N\rfloor$ training samples via $\bm{u}_i\leftarrow C\bm{u}_i$ and $v_i\leftarrow -C^2v_i+\rho$, randomly shuffles the sample, and splits it into training and validation subsets with ratio $0.8/0.2$. Hyperparameters are selected by a trimmed validation loss that removes the largest $\lceil \omega_{\mathrm{tv}}N_{\mathrm{val}}\rceil$ absolute residuals.

For the dimension sweep in Figure~\ref{fig:dim_sweep}, we use $N\in\{10,20,50,75,100\}$ and $D_{\bxi}\in\{1,5,10,50\}$ with $T=50$, $N_{\mathrm{OOS}}=1000$, $\omega=\omega_{\mathrm{tv}}=0.2$, $C=10$, $\rho'=0.1$, and $\delta=0.1$. The OR-WDRO scale parameter is set to $\omega_{\mathrm{tv}}=0.2$, and Std WDRO and OR-WDRO tune the Wasserstein radius over $r_{\mathrm{W}}\in\{0.01,0.1,0.2,0.5,1.0\}$. Each trial also computes the plain LAD estimator, but we report only Std WDRO, OR-WDRO$(2\omega)$, OR-WDRO$(\omega)$, and MoM TIPM-DRO.

For MoM TIPM-DRO, we use $M=800$ sampled decision points. The decision set radius is $R=\sqrt{D_{\bxi}}$, and the constraint $\|\bm{\theta}\|\leq R$ is enforced with the same value of $R$. Let $N_{\mathrm{tr}}=\lfloor 0.8N\rfloor$. The block count is chosen from $K'\in \{K_{\mathrm{theory}},\min(\lfloor N_{\mathrm{tr}}/3\rfloor,2K_{\mathrm{theory}}),\lfloor N_{\mathrm{tr}}/2\rfloor\}\cap\{1,\ldots,N_{\mathrm{tr}}\}$, where $K_{\mathrm{theory}}=\left\lceil \frac{4(1+2\omega)}{(1-2\omega)^2}\log\frac{1}{\delta}\right\rceil$. The TIPM radius is tuned over $\epsilon_{\mathrm{IPM}}\in \{\epsilon_{\mathrm{theory}} : s\in\{0.01,0.05,0.1,0.2,0.5,1,5,10\}\}$ with $\epsilon_{\mathrm{theory}}=\frac{4\sqrt{e}\,\sigma_{\mathrm{loss}}\Gamma_\omega}{\sqrt{N_{\mathrm{tr}}}}$ and $\Gamma_\omega=(1-2\omega_{\mathrm{tv}})^{-1}$. The second-moment constraint is set to $\Omega = 2\,\Omega_{\mathrm{tr}}$, where $\Omega_{\mathrm{tr}}=N_{\mathrm{tr}}^{-1}\sum_{i=1}^{N_{\mathrm{tr}}}\|(\bm{u}_i,v_i)\|^2$. For each candidate $R$, the same set of sampled decision points is reused across all $(K,\epsilon_{\mathrm{IPM}})$ candidates in that trial to stabilize validation comparisons.

For scalar-response MAD regression, with decision variable $\bm\theta\in\R^{D_{\bxi}}$, the loss admits the piecewise-affine representation
\[
\ell(\bm\theta,\bxi)=\max\left\{\begin{bmatrix}\bm\theta\\ -1\end{bmatrix}^{\!\top}\bxi,\begin{bmatrix}-\bm\theta\\ 1\end{bmatrix}^{\!\top}\bxi\right\},
\qquad \bxi=(\bm u,v)\in\R^{D_{\bxi}+1}.
\]
Let $\{\bm z_m\}_{m=1}^M\subset\R^{D_{\bxi}}$ be the sampled decision points and let $\bm Z\in\R^{M\times D_{\bxi}}$ collect the row vectors $\bm z_m^\top$. Writing $\widehat\mu_{\mathrm{MoM}}(\bm z_m)$ for the MoM reference from Section~\ref{sec:outlier}, the implementation solves the conic program
\begin{equation}
\label{eq:regression_ipm_sec_appendix}
\begin{aligned}
\min \;& \alpha + \beta\Omega + \frac{1}{M}\sum_{m\in[M]}\nu_m\bigl(\widehat\mu_{\mathrm{MoM}}(\bm z_m)+\epsilon\bigr)\\
\st\;& \bm\theta\in\R^{D_{\bxi}},\;\|\bm\theta\|\le R,\;\alpha\in\R,\;\beta\in\R_+,\;\bm\nu\in\R_+^M,\\
& \bm\lambda_1^+,\bm\lambda_1^-,\bm\lambda_2^+,\bm\lambda_2^-\in\R_+^M,\;\zeta_1,\zeta_2\in\R_+,\\
& \bm\lambda_1^+ + \bm\lambda_1^- = \bm\nu,\qquad
  \bm\lambda_2^+ + \bm\lambda_2^- = \bm\nu,\\
& \zeta_1\le \alpha,\qquad \zeta_2\le \alpha,\\
& \left\|\begin{bmatrix}
\bm\theta-\frac{1}{M}\bm Z^\top(\bm\lambda_1^+-\bm\lambda_1^-)\\
-1+\frac{1}{M}\mathbf e^\top(\bm\lambda_1^+-\bm\lambda_1^-)\\
\zeta_1-\beta
\end{bmatrix}\right\| \le \zeta_1+\beta,\\
& \left\|\begin{bmatrix}
-\bm\theta-\frac{1}{M}\bm Z^\top(\bm\lambda_2^+-\bm\lambda_2^-)\\
1+\frac{1}{M}\mathbf e^\top(\bm\lambda_2^+-\bm\lambda_2^-)\\
\zeta_2-\beta
\end{bmatrix}\right\| \le \zeta_2+\beta.
\end{aligned}
\end{equation}
This is exactly the formulation implemented by the MoM TIPM-DRO solver used in Section~5.2, written in the notation of Section~2.

For Figure~\ref{fig:cont_scale}, the final plotting block fixes $N=50$, $D_{\bxi}=10$, and $C\in\{5,6,\ldots,20\}$ while keeping $T=50$, $N_{\mathrm{OOS}}=1000$, $\omega=\omega_{\mathrm{tv}}=0.2$, $\rho=0.1$, $\delta=0.1$, and $M=800$. In this block, the OR-WDRO scale parameter is fixed at $\omega_{\mathrm{tv}}=0.2$, and the Wasserstein radius grid is expanded to $r_{\mathrm{W}}\in\{0.01,0.05,0.1,0.5,1.0,2.0\}$. The plotting cell for Figure~\ref{fig:cont_scale} reports only OR-WDRO$(2\omega)$, OR-WDRO$(\omega)$, and MoM TIPM-DRO. Apart from the fixed dimension and the sweep over $C$, these are the only substantive differences relative to the dimension-sweep experiments.

\end{APPENDICES}

\end{document}